\documentclass[12pt,reqno]{amsart}
\usepackage{amssymb,mathrsfs,longtable}

\newcommand{\B}{\mathcal{B}}

\newcommand{\T}{\mathcal{T}}

\newcommand{\M}{\mathcal{M}}
\newcommand{\V}{\mathcal{V}}
\newcommand{\X}{\mathcal{X}}
\newcommand{\Y}{\mathcal{Y}}

\newcommand{\cL}{\mathcal{L}}
\newcommand{\cS}{\mathcal{S}}
\newcommand{\cH}{\mathcal{H}}
\newcommand{\cP}{\mathcal{P}}
\newcommand{\cQ}{\mathcal{Q}}
\newcommand{\cN}{\mathcal{N}}
\newcommand{\cZ}{\mathcal{Z}}
\newcommand{\cR}{\mathcal{R}}
\newcommand{\cU}{\mathcal{U}}

\newcommand{\C}{\mathbb{C}}
\newcommand{\N}{\mathbb{N}}
\newcommand{\R}{\mathbb{R}}
\newcommand{\Z}{\mathbb{Z}}

\newcommand{\bA}{\mathbb{A}}
\newcommand{\bE}{\mathbb{E}}
\newcommand{\bM}{\mathbb{M}}
\newcommand{\bK}{\mathbb{K}}
\newcommand{\bG}{\mathbb{G}}
\newcommand{\bH}{\mathbb{H}}

\newcommand{\bJ}{\mathbb{J}}

\newcommand{\rS}{\mathrm{S}}

\newcommand{\sB}{\mathscr{B}}

\newcommand{\sC}{\mathscr{C}}
\newcommand{\sV}{\mathscr{V}}

\newcommand{\sM}{\mathscr{M}}
\newcommand{\sN}{\mathscr{N}}
\newcommand{\sU}{\mathscr{U}}
\newcommand{\sZ}{\mathscr{Z}}

\newcommand{\CD}{C_{\!D}}
\newcommand{\CS}{C_{\!S}}
\newcommand{\CM}{C_{\!M}}
\newcommand{\CK}{C_{\!K}}

\newcommand{\Soli}{\mathscr{S}}

\newcommand{\Stz}{\operatorname{Stz}}
\newcommand{\ST}{\mathfrak{st}}
\newcommand{\STN}{\operatorname{ST}}
\newcommand{\loc}{_{\operatorname{loc}}}
\newcommand{\ul}{_{\text{\sc ul}}}
\newcommand{\FS}{\operatorname{FS}}

\newcommand{\al}{\alpha}
\newcommand{\be}{\beta}
\newcommand{\ga}{\gamma}
\newcommand{\de}{\delta}
\newcommand{\e}{\varepsilon}
\newcommand{\fy}{\varphi}
\newcommand{\om}{\omega}
\newcommand{\la}{\lambda}
\newcommand{\lb}{b}
\newcommand{\te}{\theta}
\newcommand{\si}{\sigma}
\newcommand{\ta}{\tau}
\newcommand{\ka}{\kappa}
\newcommand{\x}{\xi}
\newcommand{\y}{\eta}
\newcommand{\z}{\zeta}

\newcommand{\De}{\Delta}
\newcommand{\Om}{\Omega}
\newcommand{\OM}{\varpi}
\newcommand{\Ga}{\Gamma}
\newcommand{\La}{\Lambda}

\newcommand{\p}{\partial}
\newcommand{\na}{\nabla}
\newcommand{\Cu}{\bigcup}

\newcommand{\re}{\mathop{\mathrm{Re}}}
\newcommand{\im}{\mathop{\mathrm{Im}}}
\newcommand{\weak}{\operatorname{w-}}
\newcommand{\weakto}{\rightharpoonup}
\newcommand{\supp}{\operatorname{supp}}
\newcommand{\sign}{\operatorname{sign}}
\newcommand{\dist}{\operatorname{dist}}

\newcommand{\lec}{\lesssim}
\newcommand{\gec}{\gtrsim}

\newcommand{\IN}[1]{\text{ in }#1}

\newcommand{\etc}{,\ldots,}
\newcommand{\I}{\infty}

\newcommand{\Ker}{\operatorname{Ker}}
\newcommand{\Image}{\operatorname{Ran}}
\newcommand{\Span}{\operatorname{span}}

\newcommand{\ti}{\widetilde}
\newcommand{\ck}{\check}
\newcommand{\ba}{\overline}

\newcommand{\U}{\underline}
\newcommand{\LR}[1]{{\langle #1 \rangle}}
\newcommand{\Lim}{\lim\limits}
\newcommand{\Liminf}{\liminf\limits}

\newcommand{\diff}[1]{{\triangleleft #1}}
\newcommand{\pa}{^\triangleright}
\newcommand{\on}{^{\scriptscriptstyle{\U{1}}}}
\newcommand{\zr}{^{\scriptscriptstyle{\U{0}}}}
\newcommand{\oj}{^{\scriptscriptstyle{\U{j}}}}
\newcommand{\pb}{^\diamond}
\newcommand{\cs}{{\mathfrak{cs}}}
\newcommand{\ce}{{\mathfrak{c}}}
\newcommand{\pc}{\pb_\cs}

\newcommand{\EQ}[1]{\begin{equation}\begin{split} #1 \end{split}\end{equation}}

\newcommand{\BR}[1]{\left[#1\right]}

\newcommand{\tand}{\ \text{ and }\ }

\newcommand{\Del}[1]{}
\newcommand{\CAS}[1]{\begin{cases} #1 \end{cases}}
\newcommand{\SAC}[1]{\left.\begin{aligned} #1 \end{aligned}\right\}}

\newcommand{\pt}{&}
\newcommand{\pr}{\\ &}
\newcommand{\pq}{\quad}

\newcommand{\pn}{}
\newcommand{\prq}{\\ &\quad}
\newcommand{\prQ}{\\ &\qquad}

\numberwithin{equation}{section}

\newtheorem{thm}{Theorem}[section]

\newtheorem{lem}[thm]{Lemma}

\theoremstyle{remark}
\newtheorem{rem}[thm]{Remark}

\newcommand{\Sg}{\mathfrak{S}}

\newcommand{\pp}{\mathfrak{p}}
\newcommand{\sE}{\mathscr{E}}
\newcommand{\ml}[1]{\lceil #1 \rfloor}
\newcommand{\st}{\operatorname{\mathfrak{st}}}
\newcommand{\deri}{\acute}

\newcommand{\EN}[1]{\begin{enumerate} #1 \end{enumerate}}
\newcommand{\ENI}[1]{{\begin{enumerate}\renewcommand{\theenumi}{\roman{enumi}} #1 \end{enumerate}}}
\newcommand{\ENA}[1]{{\begin{enumerate}\renewcommand{\theenumi}{\alph{enumi}} #1 \end{enumerate}}}

\begin{document}

\title[Above the first excited energy]{Global dynamics above the first excited energy for the nonlinear Schr\"odinger equation \\ with a potential}

\begin{abstract}
Consider the focusing nonlinear Schr\"odinger equation (NLS) with a potential with a single negative eigenvalue. 
It has solitons with negative small energy, which are asymptotically stable, and solitons with positive large energy, which are unstable. 
We classify the global dynamics into 9 sets of solutions in the phase space including both solitons, restricted by small mass, radial symmetry, and an energy bound slightly above the second lowest one of solitons. 
The classification includes a stable set of solutions which start near the first excited solitons, approach the ground states locally in space for large time with large radiation to the spatial infinity, and blow up in negative finite time. 
\end{abstract}

\author[K.~Nakanishi]{Kenji Nakanishi}

\address{Department of Pure and Applied Mathematics
Graduate School of Information Science and Technology
Osaka University, Suita, Osaka 565-0871, JAPAN}

\email{nakanishi@ist.osaka-u.ac.jp}


\maketitle

\tableofcontents

\section{Introduction}
We continue from \cite{NLSP1} the study of global dynamics for the nonlinear Schr\"odinger equation with a potential $V=V(|x|):\R^3\to\R$ which decays as $|x|\to\I$,
\EQ{ \label{NLSP}
 i\dot u + H u = |u|^2u, \pq H:=-\De+V,  \pq u(t,x):\R^{1+3}\to\C, }
in the case $H$ has a bound state $0<\phi_0\in L^2(\R^3)$ 
\EQ{ \label{def e0phi0}
 H\phi_0 = e_0\phi_0, \pq e_0<0, \pq \|\phi_0\|_2=1,}
and no other eigenfunction nor resonance (so $\phi_0$ is the ground state of $H$). Henceforth, $\|\cdot\|_p$ denotes the $L^p(\R^3)$ norm. 
See Section \ref{ss:asm V} for the precise assumptions on $V$. 
As a simple case, it suffices to assume $V=V(|x|)\in\cS(\R^3)$ besides the above spectral condition. 

In \cite{NLSP1}, the global behavior was investigated for all radial solutions $u$ with small mass and energy below the first excited state. 
In this paper, the analysis goes slightly {\it above the threshold energy}: 
\EQ{ \label{def ME}
 \pt \bM(u):=\int_{\R^3}\frac{|u|^2}{2}dx \ll 1,
 \pr \bE(u):=\int_{\R^3}\frac{|\na u|^2+V|u|^2}{2}-\frac{|u|^4}{4}dx
 < \sE_1(\bM(u))(1+\e^2),}
for some small $\e>0$, 
where $\sE_1(\mu)$ denotes {\it the second lowest energy of solitons} for the prescribed mass $\bM(u)=\mu$. 
The goal of this paper is to give a complete classification of global dynamics including both stable and unstable solitons, as well as scattering and blow-up, in a phase space restricted only by the conserved quantities and the symmetry. 
The main questions are which initial data $u(0)$ lead to each type of solutions, and how the solution $u$ can change its behavior from one type to another along its evolution. 
See \cite[Introduction]{NLSP1}, \cite{book} and references therein for more background and motivation of this setting. 

\subsection{Solitons}
In order to state the main result precisely, we first need to define the energy levels of the ground state and the excited states. 
Consider the elliptic equation for the solution of the form $u(t)=e^{-it\om}\fy(x)$ for any time frequency $\om\in\R$
\EQ{ \label{sNLSP}
 (H+\om)\fy = |\fy|^2\fy}
and let $\Soli$ be the set of all radial solutions 
\EQ{ \label{def Soli}
 \Soli:=\{\fy\in H^1_r(\R^3) \mid \exists\om>0,\text{ s.t. } \eqref{sNLSP}\},}
where $H^1_r(\R^3)$ denotes the subspace of radially symmetric functions of $H^1(\R^3)$ with the norm $\|\fy\|_{H^1}^2=\|\na\fy\|_2^2+\|\fy\|_2^2$. 
The restriction to $\om>0$ comes from the absence of embedded eigenvalue for the Schr\"odinger operator $-\De+V-|\fy|^2$, which follows from the ODE in the radial setting. 
The ground state energy level is defined for each prescribed mass $\mu>0$ by
\EQ{ \label{def E0}
 \sE_0(\mu):=\inf\{\bE(\fy)\mid \fy\in\Soli,\ \bM(\fy)=\mu\},}
and the $j$-th excited state energy level is defined inductively by 
\EQ{ \label{def Ej}
 \sE_j(\mu):=\inf\{\bE(\fy)\mid \fy\in\Soli,\ \bM(\fy)=\mu,\ \bE(\fy)>\sE_{j-1}(\mu) \}}
together with the corresponding set of solitons
\EQ{ \label{def Solij}
 \Soli_j:=\{\fy\in\Soli\mid \bE(\fy)=\sE_j(\bM(\fy))\}.}

The small mass constraint $\bM(u)\ll 1$ enables us to identify the ground states $\Soli_0$ as bifurcation of the linear ground state $\phi_0$ for $\om\to-e_0+0$, and the first excited states $\Soli_1$ as rescaled perturbation for $\om\to\I$ of the ground state $Q$ of the nonlinear Schr\"odinger equation without the potential: 
\EQ{ \label{NLS}
 i\dot u - \De u = |u|^2 u.}
More precisely, let $Q\in H^1(\R^3)$ be the unique positive radial solution of 
\EQ{ \label{eq Q}
 -\De Q + Q = Q^3.}
There are constants $0<\mu_*,z_*\ll 1\ll \om_*<\I$ and $C^1$ maps 
\EQ{ \label{def PhiPsi}
 \pt (\Phi,\Om) :Z_*:=\{z\in\C\mid|z|<z_*\}\to H^1_r(\R^3)\times(-e_0,\I)
 \pr \Psi:[\om_*,\I)\to H^1_r(\R^3),}
such that $(\fy,\om)=(\Phi[z],\Om[z]),(\Psi[\om],\om)$ are solutions of \eqref{sNLSP} satisfying 
\EQ{
 \pt \Phi[z]=z\phi_0+\ga, \pq \ga\perp\phi_0, \pq \|\ga\|_{H^1}\lec|z|^3,
 \pq \Om[z]=-e_0+O(|z|^2), 
 \pr \Psi[\om](|x|) = \om^{1/2}(Q+\ga)(\om^{1/2} x), \pq \|\ga\|_{H^1}\lec \om^{-1/4},}
with asymptotic formulas of mass and energy 
\EQ{
 \pt \bM(\Phi[z])=|z|^2/2+O(|z|^6), \pq \bE(\Phi[z])=e_0|z|^2/2+O(|z|^4),
 \pr \bM(\Psi[\om])= \om^{-1/2}\bM(Q)+O(\om^{-3/4}),
 \pq \bE(\Psi[\om])= \om^{1/2}\bE(Q)+O(\om^{1/4}),}
as well as their monotonicity $\frac{d}{da}\bM(\Phi[a])\sim a$, $\frac{d}{d\om}\bM(\Psi[\om])\sim-\om^{-3/2}$, and 
\EQ{ \label{Soli by PhiPsi}
 \pt \Soli_0|_{\bM<\mu_*}=\{\Phi[z]\mid z\in Z_*\}, 
 \pq \Soli_1|_{\bM<\mu_*}=\{e^{i\te}\Psi[\om] \mid \te\in\R,\ \om>\om_*\}.}
Moreover, as $\mu_*>\mu\to+0$, 
\EQ{ \label{formula E012}
 \pt\sE_0(\mu)=e_0\mu(1+O(\mu)), 
 \pr\sE_1(\mu)=\bM(Q)^2\mu^{-1}(1+O(\mu^{1/2})),
 \pr\sE_2(\mu)>4\bM(Q)^2\mu^{-1}.}
Note that the energy for \eqref{NLS}
\EQ{
 \bE^0(\fy):=\int_{\R^3}\frac{|\na \fy|^2}{2}-\frac{|\fy|^4}{4}dx}
is identical to $\bM(\fy)$ if $\fy$ is a solution of \eqref{eq Q}. 
A proof of the above statements is given in \cite[Lemma 2.1]{gnt} for the ground state part, and in Lemma \ref{lem:sum} for the excited state part. 

\subsection{Types of behavior} \label{ss:tob}
In this paper, we consider the following three types of behavior of the solution $u$, both in positive time and in negative time, which leads to a classification into 9 non-empty sets of solutions. 
\EN{
\item Scattering to the ground states $\Soli_0$.
\item Blow-up. 
\item Trapping by the first excited states $\Soli_1$.
}
All the solutions below the excited states $\bE(u)<\sE_1(\bM(u))$ are completely split into (1) and (2) with the same behavior in $t>0$ and in $t<0$, which is explicitly predictable by the initial data, using the virial functional:
\EQ{ \label{def K2}
 \bK_2(u):=\p_{\al=1}\bE(\al^{3/2}u(\al x))=\int_{\R^3}|\na u|^2-\frac{rV_r|u|^2}{2}-\frac{3|u|^4}{4}dx,}
where $r:=|x|$ is the radial variable. 
See \cite[Theorem 1.1]{NLSP1} for the precise statement. 
The difference between below and above the excited energy are the new type (3), and solutions with different types of behavior in $t>0$ and in $t<0$, namely {\it transition} among (1)--(3). 

The following are precise definitions for (1)--(3), under the small mass constraint. 
Let $u$ be a solution of \eqref{NLSP}. 
The local wellposedness in $H^1(\R^3)$ implies that the maximal existence interval $(T_-(u),T_+(u))\subset\R$ is uniquely defined such that 
\EQ{
 u\in C((T_-(u),T_+(u));H^1(\R^3))} 
solves \eqref{NLSP} for $T_-(u)<t<T_+(u)$. 

We say that {\it $u$ blows up in $t>0$}, if $T_+(u)<\I$. 
Otherwise, we say that {\it $u$ is global in $t>0$}. 
We say that {\it $u$ scatters to the ground states} (or {\it scatters to $\Phi$} in short) as $t\to\I$, if for some $C^1$ function $z:(T_-(u),\I)\to Z_*$ and $u_+\in H^1_r(\R^3)$ 
\EQ{ 
 \lim_{t\to\I}\|u(t)-\Phi[z(t)]-e^{-it\De}u_+\|_{H^1(\R^3)}=0.}

For each $\om>0$, we introduce the following equivalent norm of $H^1(\R^3)$
\EQ{ \label{def H1om}
 \|\fy\|_{H^1_\om}^2:=\int_{\R^3}\om^{-1/2}|\na\fy|^2+\om^{1/2}|\fy|^2 dx. }
It dominates the homogeneous Sobolev norm $\dot H^{1/2}$ uniformly, and is the appropriate rescaling for the first excited states, because  
\EQ{
 \|\Psi[\om]\|_{H^1_\om}=\|Q\|_{H^1}(1+O(\om^{-1/4})).}
The open neighborhood of the small-mass part of $\Soli_1$ within distance $\de>0$ in this metric is denoted by 
\EQ{ \label{def Nde}
 \cN_\de(\Psi):=\{\fy\in H^1_r(\R^3) \mid \exists\te\in\R,\ \exists\om>\om_*,\ \|e^{i\te}\Psi[\om]-\fy\|_{H^1_\om}<\de\}.}
We say that {\it $u$ is trapped by the first excited states} (or {\it trapped by $\Psi$} in short) as $t\to\I$, if $u(t)\in\cN_\de(\Psi)$ for large $t$ and some small fixed $0<\de\ll 1$. 

We can easily distinguish the above three types using the $L^4_x$ norm for small mass solutions as follows. 
If $u$ scatters to $\Phi$ as $t\to\I$, then  
\EQ{
 \|u(t)\|_4 = \|\Phi[z]\|_4 + o(1) \sim \|\Phi[z]\|_2 + o(1) \lec \sqrt{\bM(u)}+o(1) \ll 1,}
as $t\to\I$. 
If $u$ blows up in $t>0$, then as $t\to T_+(u)$, 
\EQ{
 \|u(t)\|_4^4/2 \pt= -2\bE(u) + \int_{\R^3} |\na u(t)|+V|u|^2dx
 \pr \ge -2\bE(u)+\|\na u(t)\|_2^2+O(\|u(t)\|_2^2+\|u(t)\|_4^2)\to \I,}
since $V\in (L^2+L^\I)(\R^3)$. 
If $u$ is trapped by $\Psi$ as $t\to\I$, then for large $t$,
\EQ{
 \|u(t)\|_4 = \|\Psi[\om]\|_4-O(\de\om^{1/8}) \sim \om^{1/8} \gg 1,}
where $\Psi[\om]$ is estimated by \eqref{def Psi Qom}-\eqref{Psi next term} and the remainder is bounded by Gagliardo-Nirenberg
\EQ{
 \|\fy\|_4 \lec \|\na\fy\|_2^{3/4}\|\fy\|_2^{1/4}
 \lec \om^{1/8}\|\fy\|_{H^1_\om}.}

\subsection{The main result}
We consider the following set of initial data where the mass is bounded and the energy is allowed to exceed the first excited state slightly:
\EQ{
 \cH^{\mu,\e}:=\{\fy\in H^1_r(\R^3) \mid \bM(\fy)<\mu,\ \bE(\fy)<\sE_1(\bM(\fy))(1+\e^2)\}.}
These initial data can be classified by behavior of the solution in $t>0$:
\EQ{ \label{def SBT}
 \pt \cS:=\{u(0)\in H^1_r(\R^3) \mid \text{$u$ scatters to $\Phi$ as $t\to\I$}\},
 \pr \B:=\{u(0)\in H^1_r(\R^3) \mid \text{$u$ blows up in $t>0$}\},
 \pr \T_\de:=\{u(0)\in H^1_r(\R^3) \mid \text{$u(t)\in\cN_\de(\Psi)$ for large $t$}\},}
where $u$ is the solution of \eqref{NLSP} for the initial data $u(0)$. 
The same classification for $t<0$ is given by their complex conjugate, thanks to the time inversion symmetry: if $u(t,x)$ is a solution of \eqref{NLS}, then so is $\bar u(-t,x)$. 

\begin{thm} \label{thm:main}
If $\mu$ and $\e$ are small enough, then we have the following. 
\EQ{
 \cH^{\mu,\e} \subset \cS \cup \B \cup \T_{\de},}
for some $\de\le C\e$, where $C>0$ is some constant independent of $\mu,\e>0$. 
Each of the 9 combinations of dynamics in positive and negative time 
\EQ{
  \cH^{\mu,\e}\cap \X \cap \ba{\Y}, \pq \X,\Y\in\{\cS,\B,\T_{\de}\},}
contain infinitely many orbits. 
$\cS$ is open. $\T_\de\cap\cH^{\mu,\e}$ is a $C^1$ manifold of codimension $1$ in $H^1_r(\R^3)$, connected and unbounded. 
$\T_\de\cap\ba{\T_\de}\cap\cH^{\mu,\e}$ is a $C^1$ manifold of codimension $2$, connected and contained in $\cN_{\de'}(\Psi)$ for some $\de'\le C\de$. 
There is a connected open neighborhood of $\T_\de\cap\cH^{\mu,\e}$ which is separated by $\T_\de$ into two connected open sets contained respectively in $\cS$ and in $\B$. $\ba{\T_\de}\cap\cH^{\mu,\e}$ is also separated by $\T_\de$ into two manifolds contained respectively in $\cS$ and in $\B$. 
\end{thm}

\subsection{Notation} \label{ss:nota}
First recall some notation in \cite{NLSP1}. 
$L^p$, $B^s_{p,q}$, and $H^s_p$ denote respectively the Lebesgue, inhomogeneous Besov and Sobolev spaces on $\R^3$, and $H^s:=H^s_2$. 
$\dot H^s=\dot H^s_2$ denotes the homogeneous Sobolev space, and 
$\|\cdot\|_p$ denotes the $L^p$ norm on $\R^3$. 
$\cS'(\R^3)$ denotes the space of tempered distributions. 
The complex-valued and the real-valued $L^2$ inner products are denoted respectively by 
$(f|g):=\int_{\R^3}f(x)\ba{g(x)}dx$ and $\LR{f|g}:=\re(f|g)$. 
For any Banach function space $X$ on $\R^3$, its subspace of radial functions is denoted by $X_r$, the space equipped with the weak topology is denoted by $\weak{X}$, the weak limit is denoted by $\weak\Lim$, and $L^p_t X$ denotes the $L^p$ space for $t\in\R$ with values in $X$. Some standard Strichartz norms on $\R^{1+3}$ are denoted by 
\EQ{
 \Stz^s:=L^\I_tH^s\cap L^2_t B^s_{6,2}, \pq \ST:=L^4_tL^6.}
For any function space $Z$ on $\R^{1+3}$ and a set $I\subset\R$, the restriction of $Z$ onto $I\times\R^3$ is denoted by $Z(I)$. 
For any $w:\R^3\to\R$ and $\fy\in H^1$, the following define some functionals 
\EQ{ \label{def funct}
 \pt\ml{w}(\fy):=\frac12\LR{w\fy|\fy}, 
  \pq \bM:=\ml{1}, \pq \bG(\fy):=\frac14\|\fy\|_4^4,
  \pr \bH^0(\fy):=\frac12\|\na\fy\|_2^2, 
  \pq \bE^0:=\bH^0-\bG, \pq \bE:=\bE^0+\ml{V}. }
For any $p\in(0,\I]$, $t\in\R$ and $\fy\in\cS'(\R^3)$, the $L^p$ preserving dilation and its generator are denoted by 
\EQ{ \label{def Spt}
 \cS_p^t\fy(x):=e^{3t/p}\fy(e^t x), \pq \cS_p'F(\fy):=\lim_{t\to 0}\frac{F(\cS_p^t\fy)-F(\fy)}{t},}
for any $F$ acting on functions on $\R^3$. We have $\cS_p'\fy=(x\cdot\na+3/p)\fy$. Then we define 
\EQ{ \label{def K2'}
 \bK_2:=\cS_2'\bE=2\bH^0-3\bG-\ml{\cS_\I'V}.}

Next, some new notation and symbols are introduced. 
For any symbols $F,X,Y$, the difference is denoted by (this is a slight modification from \cite{NLSP1})
\EQ{
 \diff F(X\pa,Y\pa\etc):= F(X\on,Y\on\etc)-F(X\zr,Y\zr\etc),}
where the symbols $\pa,\on,\zr$ and $\diff{}$ are reserved for this purpose, and the underline is to avoid confusion with exponents. 
The subspace and the projection orthogonal (in the real sense) to $\fy\in\cS(\R^3)$ are denoted by 
\EQ{ \label{def perp}
 \fy^\perp:=\{\psi\in\cS'(\R^3) \mid \LR{\fy|\psi}=0\},
 \pq P_\fy^\perp:=1-\|\fy\|_2^{-2}\fy\langle\fy|.}
The projection to the continuous spectral subspace of $H$ is denoted by 
\EQ{ \label{def Pc}
 P_c:=P_{\phi_0}^\perp P_{i\phi_0}^\perp=1-\phi_0(\phi_0|.}
For two Banach spaces $X$ and $Y$, the Banach space of bounded linear operators from $X$ to $Y$ is denoted by $\B(X,Y)$. 
For any $\om>0$ and $\fy\in H^1(\R^3)$, denote 
\EQ{ \label{def funcom}
 \pt \bA_\om:=\bE+\om\bM,
 \pq \bK_{0,\om}(\fy):=\LR{\bA_\om'(\fy)|\fy}=2(\bA_\om-\bG)(\fy),
 \pr \bE^\om(\fy):=\om^{-1/2}\bE(\rS_\om^{-1} \fy)
 =(\bE^0+\ml{V^\om})(\fy), 
 \pr \bA^\om(\fy):=\om^{-1/2}\bA_\om(\rS_\om^{-1} \fy)=(\bE^\om+\bM)(\fy),
 \pr \bK^\om_2(\fy):=\om^{-1/2}\bK_2(\rS_\om^{-1}\fy)
 =(2\bH^0-3\bG-\ml{\cS_\I'V^\om})(\fy),
 \pr \bJ^\om:=\bA^\om-\frac12\bK^\om_2=\frac12\bG+\frac12\ml{\cS_{3/2}'V^\om}+\bM,}
where the rescaling operator $\rS_\om$ and the rescaled potential $V^\om$ are defined by
\EQ{ \label{def rSom}
 \rS_\om\fy(x):=\om^{-1/2}\fy(\om^{-1/2}x), \pq V^\om(x):=\om^{-1/2}\rS_\om V,}
and the version without the potential 
\EQ{ \label{def A}
 \bA:=\bE^0+\bM=\lim_{\om\to\I}\bA^\om.}

In this paper, most of the analysis will be done in the variables rescaled by $\rS_\om$, where the smallness of $\bM$ corresponds to the largeness of $\om$. 
This is to avoid getting large scaling factors in the estimates around the first excited states, and to make the formulations similar in the leading order to the case without the potential. 
Of course, we still need to take care of the long-time impact of the potential and the ground states, even if they are small in some sense.  

\subsection{Assumptions on the potential} \label{ss:asm V}
In addition to the assumptions on $V$ in \cite{NLSP1}, we assume that $V\in L^2(\R^3)$. Hence the precise list of assumptions on $V$ is 
\ENI{
\item $V:\R^3\to\R$ is radially symmetric.
\item $V\in L^2(\R^3)\cap |x|L^1(\R^3)$ and $x\na V,x^2\na^2V\in (L^2+L^\I_0)(\R^3)$. 
\item $-\De+V$ on $L^2_r(\R^3)$ has a unique and negative eigenvalue, denoted by $e_0$. 
\item The wave operator $\Lim_{t\to\I}e^{itH}e^{it\De}$ and its adjoint are bounded on $H^k_\pp(\R^3)$ for some $\pp>6$ and $k=0,1$,
}
where $L^\I_0(\R^3):=\{\fy\in L^\I(\R^3)\mid \Lim_{R\to\I}\|\fy\|_{L^\I(|x|>R)}=0\}$. 
For example, if $V_0$ is a radial positive Schwartz function on $\R^3$, then there exist $b>a>0$ such that $V=-cV_0$ satisfies the above assumptions for $a<c<b$. See \cite{NLSP1} for more comments. 
$V$ is fixed throughout the paper, so that some ``constants" can depend on $V$. 

\section{The first excited state and the linearized operator} \label{s:excited}
In this section, we analyze the first excited state for small mass and the spectrum of the linearized operator around it. 

\subsection{Zero-mass asymptotic of energy}
For any solution $\fy\in H^1_r$ of \eqref{sNLSP} with $\bM(\fy)\ll 1$, we may apply the small mass dichotomy \cite[Lemma 2.3]{NLSP1}, since $\bK_2(\fy)=\LR{\bA_\om'(\fy)|\cS_2'\fy}=0$. Then we have either $\|\fy\|_{H^1}\sim\|\fy\|_2\ll 1$ or 
\EQ{ \label{SMD}
 \bG(\fy) \gec \bH^0(\fy) \gec \bM(\fy)^{-1} \gg 1.}
If $\bM(\fy)$ is small enough, then the former case implies that $\fy=\Phi[z]$ for some $z\in Z_*$, see \cite[Propositions 2.4 and 2.5]{NLSP1}. 
Hence we may concentrate on the latter case \eqref{SMD} for excited states. 
Since $V,\cS_p'V\in L^2+L^\I$, we have \cite[Lemma 2.1]{NLSP1}
\EQ{
 |\ml{V}(\fy)|+|\ml{\cS_p'V}(\fy)| \le \e\sqrt{\bG(\fy)} + C_{p,\e}\bM(\fy), }
for any $\e,p>0$. 
Combining this, \eqref{SMD} and $\bK_2(\fy)=0=\bK_{0,\om}(\fy)$, 
we obtain approximate Pohozaev identities as $\mu:=\bM(\fy)\to 0$, 
\EQ{ \label{asy fy funct}
 \bE(\fy)=\om\bM(\fy)+o(\mu^{-1/2})
 =\frac12\bG(\fy)+o(\mu^{-1/2})
 =\frac13\bH^0(\fy)+o(\mu^{-1/2}) \gec \mu^{-1},}
and so $\om\gec\mu^{-2}$. 
Then rescaling the solution by 
$\fy_\om:=\rS_\om\fy=\om^{-1/2}\fy(\om^{-1/2}x)$ 
leads to the rescaled equation with time frequency $1$
\EQ{ \label{rsNLS}
 (-\De+V^\om + 1)\fy_\om = |\fy_\om|^2\fy_\om,
 \pq V^\om(x)=\om^{-1}V(\om^{-1/2}x),}
and rescaled functionals
\EQ{ \label{equiv fy norms}
 (\bH^0,\bM,\bG,\bA^\om)(\fy_\om)=\om^{-1/2}(\bH^0,\om\bM,\bG,\bA_\om)(\fy)\sim\om^{1/2}\mu \gec 1.}
The rescaled potential is small in the following sense. 
Decompose $V=W_2+W_\I$ such that $\|W_2\|_2+\|W_\I\|_\I=\|V\|_{L^2+L^\I}$ and rescale $W^\om_p:=\om^{-1/2}\rS_\om W_p$. Then 
\EQ{ \label{Vom small}
 |\ml{V^\om}(\fy)| \pt\le \frac 12[\|W^\om_2\|_2\|\fy\|_4^2+\|W^\om_\I\|_\I\|\fy\|_2^2]
 \pr= \|V\|_{L^2+L^\I}[\om^{-1/4}\sqrt{\bG(\fy)}+\om^{-1}\bM(\fy)]
 \pr\lec \om^{-1/4}\|\na\fy\|_2^{3/2}\|\fy\|_2^{1/2}+\om^{-1}\|\fy\|_2^2 \lec \om^{-1/4}\|\fy\|_{H^1}^2,}
where the third inequality is by Gagliardo-Nirenberg. 
We have the same estimates on $\cS_p'V^\om$ and $\cS_p'\cS_q'V^\om$. In particular, for large $\om$, 
\EQ{ \label{AK bd H1}
 \bA^\om(\fy)-\frac 13\bK^\om_2(\fy)=\frac 16\|\na\fy\|_2^2+\frac 12\|\fy\|_2^2+\frac13\ml{\cS_1'V^\om}(\fy)
 \sim \|\fy\|_{H^1}^2.}

The small mass dichotomy \cite[Lemma 2.3]{NLSP1} is rewritten for the rescaled function. 
\begin{lem} \label{lem:dich}
There exists a constant $\CD\in(1,\I)$ such that  for any $\om\in(0,\I)$ and $\fy\in H^1(\R^3)$ satisfying 
\EQ{ \label{dich0}
 \CD\bK^\om_2(\fy)<\bM(\fy)^{-1} \tand \CD\bM(\fy)<\om^{1/2},} 
we have one of the following \eqref{dich1}--\eqref{dich3}
\begin{align}
 \pt \bH^0(\fy)\le \CD\om^{-1}\bM(\fy), \pq \bG(\fy)\le \CD\om^{-3/2}\bM(\fy)^2. \label{dich1}
 \pr \om^{-1}\bM(\fy)\le \CD\bH^0(\fy) \le \CD^2\bK^\om_2(\fy) \le \CD^3\bH^0(\fy). \label{dich2}
 \pr \bM(\fy)^{-1} \le \CD\bH^0(\fy) \le \CD^2\bG(\fy). \label{dich3}
\end{align}
For any $p,q>0$, there is a constant $C_{p,q}>0$ such that in the case \eqref{dich3},
\EQ{ \label{Vom small dich}
 |\ml{V^\om}(\fy)|+|\ml{\cS_p'V^\om}(\fy)|+|\ml{\cS_p'\cS_q'V^\om}(\fy)| \le C_{p,q}(\om^{-1/2}\bM(\fy))^{1/2}\bH^0(\fy).}
If $\fy_n$ is in the case \eqref{dich3} for all $n$, weakly converging in $H^1_r$ as $n\to\I$ and $\Liminf_{n\to\I}\bK^\om_2(\fy_n)\le 0$, then the weak limit is also in the case \eqref{dich3}. 
\end{lem}
\begin{proof}
The estimate on $\bG$ in the case \eqref{dich1} follows by Gagliardo-Nirenberg. 
The left side of \eqref{Vom small dich} is bounded using \eqref{Vom small} by
\EQ{
 (\be^{1/4}+\be)\bH^0(\fy) \lec \be^{1/4}\bH^0(\fy),}
where $\be:=\om^{-1}\bH^0(\fy)^{-1}\bM(\fy)\le \CD(\om^{-1/2}\bM(\fy))^2<1/\CD$. Thus we obtain \eqref{Vom small dich}. 

Suppose that $\fy_n\in H^1_r$ are all in the case \eqref{dich3}, $\fy_n\to\fy$ weakly in $H^1$ and $\Liminf_{n\to\I}\bK^\om_2(\fy_n)\le 0$, then $\fy_n\to\fy$ strongly in $L^4$. 
The assumption \eqref{dich0} is preserved by the weak limit, and $\bK^\om_2(\fy)\le\Liminf_{n\to\I}\bK^\om_2(\fy_n)\le 0$.  Since $\fy_n$ is in the case \eqref{dich3}, 
\EQ{ \label{3limG}
 \om^{-1/2}< \CD^{-1}\bM(\fy_n)^{-1} \le \CD\bG(\fy_n)\to\CD\bG(\fy).} 
In particular, $\fy\not=0$. Since $\bK^\om_2(\fy)\le 0$, \eqref{dich2} is impossible. Since \eqref{dich1} implies 
\EQ{
 \CD G(\fy)\le \om^{-3/2}\CD^2\bM(\fy)^2<\om^{-1/2},} 
contradicting \eqref{3limG}, hence we have \eqref{dich3}. 
The rest of the lemma follows simply by rescaling \cite[Lemma 2.3]{NLSP1}. 
\end{proof}

Now consider any sequence of $(\fy,\om)$ solving \eqref{sNLSP} such that $\mu:=\bM(\fy)\to 0$ and $\fy\not\in\Phi[Z_*]$. 
Then \eqref{equiv fy norms} implies that boundedness of the rescaled functions $\fy_\om:=\rS_\om\fy$ in $L^2(\R^3)$ is equivalent to boundedness in $L^4(\R^3)$, boundedness in $\dot H^1(\R^3)$,  
\EQ{
 \om \sim \mu^{-1/2},}
and weak convergence in $H^1(\R^3)$ of $\fy_\om$ along a subsequence. 
In that case, let $\fy_\I\in H^1_r(\R^3)$ be the weak limit along the subsequence. Then it solves the static nonlinear Schr\"odinger equation without the potential 
\EQ{ \label{sNLS}
 (-\De + 1)\fy_\I = |\fy_\I|^2\fy_\I,} 
and the weak convergence implies that the limit energy satisfies 
\EQ{ \label{limit ene fy}
 \bA(\fy_\I) \pt\le \liminf \bA(\fy_\om)
 \pn= \liminf \om^{-1/2}\bA_\om(\fy) = \liminf 2\sqrt{\bM\bE(\fy)},}
where the last two equalities follow from \eqref{equiv fy norms} and \eqref{asy fy funct}. 

By the classical results on radial solutions of \eqref{sNLS}, 
all the solutions $\psi$ are real-valued (modulo complex rotation $\psi\mapsto e^{i\te}\psi$), satisfying $\bE^0(\psi)=\bM(\psi)$. 
The least energy non-trivial solution is the unique positive solution, namely the ground state $Q$. 
The other radial solutions have at least one zero point in $r=|x|>0$, and those $\psi\not=0$ with $m$ zeros have at least $m+1$ times energy: 
\EQ{
 \bA(\psi)>(m+1)\bA(Q)=2(m+1)\bM(Q)>0.}

The asymptotic \eqref{limit ene fy} of energy implies that if $(\fy,\om)$ is a sequence of solutions to \eqref{sNLSP} such that $\bM(\fy)\to 0$, $\fy\not\in\Phi[Z_*]$, and 
$\bA^\om(\rS_\om\fy) \le 2\bA(Q)$, then we have 
$e^{i\te}\fy\to Q$ for some sequence of $\te\in\R/2\pi\Z$, 
strongly in $H^1_r(\R^3)$. 
The convergence has to be strong, since otherwise 
\EQ{
 0=2(\bA-\bG)(Q) \le \liminf \om^{-1/2}\bK_{0,\om}(\fy)=0}
would become a strict inequality. 
Thus we have obtained 
\begin{lem} \label{lem:excited energy}
For any $\de>0$, there exists $\mu(\de)>0$ such that 
if $(\fy,\om)\in H^1_r\times(0,\I)$ satisfies the soliton equation \eqref{sNLSP} and 
\EQ{
 \bM(\fy)\le \mu(\de), \pq \bA^\om(\rS_\om\fy) \le 2\bA(Q),}
then either $\fy\in\Phi[Z_*]$ or 
\EQ{ \label{soliton near Q}
 \|\rS_\om\fy-e^{i\te}Q\|_{H^1} < \de}
for some $\om\sim\bM(\fy)^{-2}$ and $\te\in\R$. 
\end{lem}
The first excited states satisfy the above energy constraint \cite[Proposition 2.5]{NLSP1}. 
In the next subsection, we prove that they are indeed the only solitons satisfying \eqref{soliton near Q} for small mass. Then the above lemma implies the estimate on $\sE_2$ in \eqref{formula E012}. 

\subsection{Construction of the first excited state}
The above lemma allows us to expand the first excited state in the rescaled variables around $Q$. 
For that purpose, consider the linearized operators for \eqref{sNLS} around $Q$ 
\EQ{ \label{def L}
 \cL v:=L_+v_1+iL_-v_2, \pq 
 L_+ := -\De + 1 - 3Q^2, \pq L_-:=-\De+1-Q^2,}
where $v_1:=\re v$ and $v_2:=\im v$ for any $v\in\cS'(\R^3)$. 
The null space of $\cL$ on $L^2_r$ equals to $\Span\{iQ\}$. 
$\cL$ is invertible on the radial subspace orthogonal to $iQ$, and $\cL^{-1}$ is bounded $H^{-1}_r \cap (iQ)^\perp \to H^1_r \cap (iQ)^\perp$. 
In other words, $(L_+)^{-1}:H^{-1}_r\to H^1_r$ and $(L_-)^{-1}:H^{-1}_r\cap Q^\perp \to H^1_r\cap Q^\perp$ are bounded. 

Let $\om>0$ and let $\psi\in H^1_r(\R^3)$ be a solution of \eqref{rsNLS} close to the ground states of \eqref{sNLS}, in other words 
\EQ{
 \de:=\inf_{\te\in\R}\|\psi-e^{i\te}Q\|_{H^1}}
is small enough. Then there is a unique $\te\in\R/2\pi\Z$ such that 
\EQ{
 \psi=e^{i\te}(Q+v) \implies 0=\LR{iQ|v}=\LR{iQ|e^{-i\te}\psi}, \pq \|v\|_{H^1}\sim \de.}
Indeed, it is explicitly given by $\te=\arg(\psi|Q)$. 
Then \eqref{rsNLS} is rewritten into the following equation for $v$: 
\EQ{ 
 \cL v = N(v) - V^\om (Q+v), \pq N(v):=2Q|v|^2+Qv^2+|v|^2v.}
and, since $v\perp iQ$, 
\EQ{ \label{seq v}
 v = (\cL|_{(iQ)^\perp})^{-1}P_{iQ}^\perp(N(v) - V^\om (Q+v)),}
where the orthogonal projection $P_{iQ}^\perp:\cS'\to(iQ)^\perp$ is bounded on $H^s$ for any $s\in\R$. Using Sobolev, H\"older and \eqref{Vom small}, we have 
\EQ{
 \|N(v)-V^\om(Q+v)\|_{H^{-1}} \pt\lec \|v\|_{H^1}^2+\|v\|_{H^1}^3+\om^{-1/4}(\|Q\|_{H^1}+\|v\|_{H^1})
 \pr\lec \om^{-1/4}(1+\|v\|_{H^1})+\|v\|_{H^1}^2+\|v\|_{H^1}^3,}
and similarly for any small $v\zr,v\on\in H^1$, 
\EQ{
 \|\diff[N(v\pa)-V^\om(Q+v\pa)]\|_{H^{-1}}
 \lec (\om^{-1/4}+\|v\zr\|_{H^1}+\|v\on\|_{H^1})\|\diff v\pa\|_{H^1}.}
Hence the right hand side of \eqref{seq v} is a contraction map for small $v\in H^1_r$ if $\om\gg 1$, having a unique fixed point $v\in H^1_r\cap(iQ)^\perp$ which is small in $H^1$. 
Thus 
\EQ{ \label{def Psi Qom}
 \Psi[\om]:=\rS_\om^{-1}Q_\om, \pq Q_\om:=Q+v,}
is a family of solutions to \eqref{sNLSP}, smoothly depending on $\om\gg 1$ with  
\EQ{ \label{Psi next term}
 v=-(L_+)^{-1}V^\om Q + O(\om^{-1/2}) = O(\om^{-1/4}) \IN{H^1_r(\R^3)}.}
Denote the orbit of the rescaled soliton by 
\EQ{ \label{def Qom orbit}
 \cQ_\om:=\{e^{i\te}Q_\om\mid \te\in\R\}.}

For the energy and mass, we deduce from \eqref{def Psi Qom}--\eqref{Psi next term}
\EQ{
 \sE_1(\bM(\Psi[\om]))=\pt\bE(\Psi[\om])=\om^{1/2}(\bE^0(Q)+O(\om^{-1/4})),
 \pr \bM(\Psi[\om])=\om^{-1/2}(\bM(Q)+O(\om^{-1/4})).}
Hence putting $\mu=\bM(\Psi[\om])$, we obtain 
\EQ{
 \om=\mu^{-2}(\bM(Q)^2+O(\mu^{1/2})),
 \pq \sE_1(\mu)=\mu^{-1}(\bM\bE^0(Q)+O(\mu^{1/2})).}
Since the ground states $\Phi[Z_*]$ have smaller $\om\sim-e_0$, these solitons have the least energy for fixed $\om\gg 1$. 
Hence they are also the constrained minimizers 
\EQ{
 \bA_\om(\Psi[\om])=\inf\{\bA_\om(\fy)\mid 0\not=\fy\in H^1_r(\R^3),\ \bK_{0,\om}(\fy)=0\}.}
In particular, $\Psi[\om]>0$ on $\R^3$. 
For the analysis of dynamics, it is more important to relate it to the virial identity or $\bK_2$. 
\begin{lem} \label{lem:minid}
There is a constant $\CM\in(1,\I)$ such that if $\om\in(1,\I]$ is large enough then we have the following for any $\fy\in H^1_r(\R^3)$. 
\EN{
\item $\bM(\fy)<2\bA(Q)$ and \eqref{dich1} $\implies(\bM+\om\bH^0)(\fy)<\CM$. 
\item $\eqref{dich3}\implies\CS\bM\bH^0(\fy)\ge 1\implies (\bM+\om\bH^0)(\fy)>\CM$, where 
\EQ{ \label{def CS}
 \CS:=\max(\CD,2 +(\bM\bH^0(Q))^{-1}).} 
\item If $(\bM+\om\bH^0)(\fy)\le\CM$ then the solution $u$ of \eqref{NLSP} with $u(0)=\rS_\om^{-1}\fy$ scatters to $\Phi$ as $t\to\pm\I$, satisfying $\bH^0(\rS_\om u(t))<\bH^0(Q)/4$ for all $t\in\R$. 
\item $\cQ_\om=\{e^{i\te}Q_\om\}_\te$ is the set of minimizers of 
\EQ{ \label{Iom}
 \pt\inf\{\bA^\om(\fy)\mid 0\not=\fy\in H^1_r,\ \bK^\om_2(\fy)=0,\ (\bM+\om\bH^0)(\fy)>\CM\}
 \pr=\inf\{\bJ^\om(\fy) \mid 0\not=\fy\in H^1_r,\ \bK^\om_2(\fy)\le 0,\ (\bM+\om\bH^0)(\fy)>\CM\}. }
}
\end{lem}
The above choice \eqref{def CS} of the constant $\CS$ will be used later for scaling invariant separation between the ground states and the excited states. 
\begin{proof}
Since \eqref{dich1} and $\bM(\fy)<2\bA(Q)$ imply 
$(\bM+\om\bH^0)(\fy)\le(1+\CD)\bM(\fy)<2(1+\CD)\bA(Q)$, 
taking $\CM\ge 2(1+\CD)\bA(Q)$ yields (1). (2) follows from 
\EQ{
 2\sqrt{\bM\bH^0} \le \om^{-1/2}(\bM+\om\bH^0),}
and taking $\om$ so large that $\om^{-1/2}\CM<2\CS^{-1/2}$.  A similar condition yields (3), because of $\|\rS_\om^{-1}\fy\|_{H^1}^2=2\om^{-1/2}(\bM+\om\bH^0)(\fy)$, $\bH^0(\rS_\om u)=\om^{-1/2}\bH^0(u)$ and the asymptotic stability \cite{gnt} of the ground states in $H^1$. 

If $\om=\I$, then the constraint on $\bM+\om\bH^0$ in \eqref{Iom} becomes trivial and (4) is reduced to a well known statement for the NLS without potential. 
So we may restrict (4) to the case $\om<\I$ (though the argument is essentially the same). 

For $\om$ large enough, $Q_\om$ satisfies all the constraints in \eqref{Iom}. 
To show the equality in \eqref{Iom}, it suffices to show that $\bJ^\om(\fy)$ is bigger than the first line of \eqref{Iom} for any $\fy$ satisfying the constraints and $\bK^\om_2(\fy)<0$, since $\bJ^\om=\bA^\om-\bK^\om_2/2$. 

Suppose that $\bJ^\om(\fy)$ is close to the second infimum and $\bK^\om_2(\fy)<0$. 
Then by $\eqref{AK bd H1}<\bJ^\om(\fy)$ we deduce that $\bM(\fy)<2\bA(Q)<\CD^{-1}\om^{1/2}$ for large $\om$. 
Then $(\bM+\om\bH^0)(\fy)>\CM$ and (1) preclude \eqref{dich1}, while $\bK^\om_2(\fy)<0$ precludes \eqref{dich2}. 
Hence we have \eqref{dich3} by Lemma \ref{lem:dich}. 

Consider the $L^2$-invariant scaling $v(t):=\cS_2^t\fy$, starting from $t=0$ and decreasing. 
As long as $\bK^\om_2(v(t))\le 0$, Lemma \ref{lem:dich} applies to $v(t)$, and (1)-(2) with the continuity of $v$ in $t$ imply that $v(t)$ stays in the case \eqref{dich3}. 
Meanwhile, we have, using \eqref{Vom small dich}, 
\EQ{ \label{dK2}
 \pt 2\cS_2'\bJ^\om(v) = 3\bG(v) - \ml{\cS_\I'\cS_{3/2}'V^\om}(v) \sim \bG(v) \gec 1, 
 \pr \cS_2'\bK^\om_2 = 2\bK^\om_2 - 2\cS_2'\bJ^\om \lec -\bG(v) \lec -1.}
Hence at some $t<0$, we have $\bJ^\om(v(t))<\bJ^\om(\fy)$, $\bK^\om_2(v(t))=0$ and \eqref{dich3} for $v(t)$, so $(\bM+\om\bH^0)(v(t))>\CM$ by (2). 
This implies the equality in \eqref{Iom}. 

Next we prove the existence of minimizer. 
Take any sequence $\fy_n\in H^1_r$ satisfying $\fy_n\not=0$, $\bK^\om_2(\fy_n)=0$, $(\bM+\om\bH^0)(\fy_n)>\CM$ and $\bA^\om(\fy_n)\to\eqref{Iom}$. 
The same argument as above implies that $\fy_n$ is in the case of \eqref{dich3}. 
Passing to a subsequence, we have $\fy_n\to\exists\fy$ weakly in $H^1_r$, then $\bK^\om_2(\fy)\le 0$, $\bJ^\om(\fy)\le\eqref{Iom}$, and by Lemma \ref{lem:dich}, $\fy$ also satisfies \eqref{dich3}, so $(\bM+\om\bH^0)(\fy)>\CM$ by (2). Hence $\fy$ is a minimizer of the second line of \eqref{Iom}. 

For any minimizer $\fy$ of \eqref{Iom}, there is a Lagrange multiplier $\mu\in\R$ such that 
$(\bA^\om)'(\fy)=\mu(\bK^\om_2)'(\fy)$. Then 
\EQ{
 0=\bK^\om_2(\fy)=\LR{(\bA^\om)'(\fy)|\cS_2'\fy}=\mu\LR{(\bK^\om_2)'(\fy)|\cS_2'\fy}
 =\mu\cS_2'\bK^\om_2(\fy)}
together with $\cS_2'\bK^\om_2(\fy)\not=0$ by \eqref{dK2} implies $\mu=0$. 
Therefore $\fy$ is a solution of \eqref{rsNLS}, satisfying $\bG(\fy)\sim 1\gec\bM(\fy)$ and $\bA^\om(\fy)\le\bA^\om(Q_\om)$. 
Then Lemma \ref{lem:excited energy} implies that $\fy\in\cQ_\om$ if $\om$ is large enough. 
\end{proof}

\subsection{Rescaled linearization and spectrum}
Next we consider the linearization of \eqref{NLSP} around the first excited soliton in the rescaled variables. Let 
\EQ{ \label{def Lom}
 \cL^\om v:=L^\om_+v_1+iL^\om_-v_2,
 \pq \CAS{L^\om_+:=-\De+1+V^\om-3Q_\om^2,\\ L^\om_-:=-\De+1+V^\om-Q_\om^2,} }
where $Q_\om$ is the rescaled excited state as in \eqref{def Psi Qom}. 
The linearized operator for the evolution is given by $i\cL^\om$ in the rescaled variables. 
The asymptotics \eqref{def Psi Qom}--\eqref{Psi next term} of $\Psi[\om]$ together with the smallness \eqref{Vom small} of $V^\om$ implies 
\EQ{
 \cL^\om = \cL + O(\om^{-1/4}) \IN{\B(H^1,H^{-1})}.}

The gauge invariance for $e^{i\te}\times$ implies the trivial null direction 
\EQ{
 i\cL^\om iQ_\om =0.}
Another direction comes from $\om$. 
Differentiating the equation \eqref{sNLSP} for $\Psi[\om]$ yields
\EQ{
 (H+\om-3\Psi[\om]^2)\Psi'[\om]=-\Psi[\om],}
which is rescaled to 
\EQ{ \label{eq Qom'}
 L^\om_+ Q_\om' = - Q_\om,} 
where $Q_\om'\in H^1_r(\R^3)$ is defined by 
\EQ{ \label{def Qom'}
 Q_\om':= \rS_\om \om\Psi'[\om]
 = \frac 12\cS_3'Q_\om + \om\p_\om Q_\om.}
Remark that $Q'_\om\not=\p_\om Q_\om$. 
The above equation \eqref{eq Qom'} is equivalent to 
\EQ{
 i\cL^\om Q_\om' = -iQ_\om}
Since 
$L^\om_+ = L_+ + O(\om^{-1/4})$ in $\B(H^1,H^{-1})$ and $L_+:H^1_r\to H^{-1}_r$ is invertible, 
$L^\om_+$ is also invertible with 
\EQ{
 \|(L^\om_+)^{-1}-(L_+)^{-1}\|_{\B(H^{-1}_r,H^1_r)} \lec \om^{-1/4}.}
Thus we obtain 
\EQ{
 Q_\om' = (L^\om_+)^{-1}(-Q_\om)= (L^\om_+)^{-1}(-Q+O(\om^{-1/4}))=Q'+O(\om^{-1/4}) \IN{H^1_r},}
where 
\EQ{ \label{def Q'}
 Q':=\frac 12\cS_3'Q=-(L_+)^{-1}Q \in H^1_r(\R^3).}
This also tells us asymptotic formulas for $\sE_1''$ as follows. 
Since $\bA_\om'=0$ on solitons, putting $\mu=\bM(\Psi[\om])$ we have 
\EQ{
 \pt \sE_1'(\mu)=-\om=-\mu^{-2}(\bM(Q)^2+O(\mu^{1/2})),
 \pq \sE_1''(\mu)=-\frac{d\om}{d\mu},}
where the last term is computed by 
\EQ{
 \frac{d\mu}{d\om}\pt=\LR{\Psi[\om]|\Psi'[\om]}
 =\om^{-3/2}\LR{Q_\om|Q_\om'}
 \pr=\om^{-3/2}(\LR{Q|Q'}+O(\om^{-1/4}))
 =-\om^{-3/2}(\bM(Q)/2+O(\om^{-1/4})).}
Therefore, for small $\mu>0$,  
\EQ{
 \sE_1''(\mu) = \om^{3/2}(2/\bM(Q)+O(\om^{-1/4}))
 = \mu^{-3}(2\bM(Q)^2+O(\mu^{1/2}))>0.}
Using $\bE^0(Q)=\bM(Q)$, the above formulas can also be written as 
\EQ{ \label{formula E1''}
 \sE_1'(\mu) = -\mu^{-2}(\bM\bE^0(Q)+O(\mu^{1/2})),
 \pq \sE_1''(\mu) = 2\mu^{-3}(\bM\bE^0(Q)+O(\mu^{1/2})).}

Next we look for a pair of positive and negative eigenvalues. 
In the limit $\om\to\I$, we have some $\al\in(0,\I)$ and $g_\pm\in\cS_r(\R^3)$ satisfying 
\EQ{ \label{eq g}
 i\cL g_\pm = \pm \al g_\pm, \pq g_-=\ba{g_+}, \pq \al\LR{ig_+|g_-}=2,
 \pq \LR{iQ|g_+}>0, }
cf.~\cite{NLS}. Put $g_\pm=g_1\pm ig_2$. 
Consider the eigenvalue problem 
\EQ{ \label{eq gom}
 i\cL^\om g^\om = \al_\om g^\om,}
in the form
\EQ{ \label{exp alg}
 \al_\om=\al(1+c),\pq g^\om=g_++\ga,\pq \LR{i\ga|g_-}=0,}
and $|c|+\|\ga\|_{H^1}=o(1)$ as $\om\to\I$, where $c\in\R$ and $\ga\in H^1_r$ also depend on $\om$. 
Putting $R:=i\cL^\om-i\cL$, the above equation \eqref{eq gom} is equivalent to 
\EQ{
 (i\cL-\al)\ga = (-R+\al c)(g_+ + \ga),}
while the orthogonality yields an equation for $c$ 
\EQ{
 0=\LR{ig^\om|(i\cL+\al)g_-}=\LR{(i\cL-\al)g^\om|ig_-}
 \pt=\LR{(i\cL-\al)\ga|ig_-}
 \pr=-2c - \LR{R(g_++\ga)|ig_-}.}
Injecting it into the previous equation yields an equation for $\ga$ by itself 
\EQ{ \label{seq ga}
 (i\cL-\al)\ga = (-R+ \al \LR{iR(g_++\ga)|g_-}/2)(g_+ + \ga) =:\cR(\ga),}
and the above computation for $c$ implies that $\LR{i\cR(\ga)|g_-}=0$ if $\LR{i\ga|g_-}=0$. 

Since $\|R\|_{\B(H^1_r,H^{-1}_r)}\lec \om^{-1/4}$, we have 
\EQ{
 \|\cR(\ga)\|_{H^{-1}} \lec \om^{-1/4}(1+\|\ga\|_{H^1})^2,}
as well as a similar estimate for the difference. 
Hence \eqref{seq ga} has a unique fixed point $\ga\in H^1_r\cap(ig)^\perp$ for $\om\gg 1$, provided that $(i\cL-\al)$ has a bounded inverse. Indeed 

\begin{lem} 
$(i\cL-\al)$ has a bounded inverse $H^{-1}\to H^1$ on $(ig_-)^\perp$. More precisely, for any $h\in H^{-1}(\R^3)\cap(ig_-)^\perp$, there exists a unique $f\in H^1(\R^3)\cap(ig_-)^\perp$ such that $(i\cL-\al)f=h$, and moreover $\|f\|_{H^1}\lec\|h\|_{H^{-1}}$. 
\end{lem}
\begin{proof}
First remark that $\Ker(i\cL\mp\al)=\Span\{g_\pm\}$ follows from the fact that $L_-\ge 0$ and $L_+$ has only one negative eigenvalue. Indeed, if $(i\cL-\al)g=0$ for some $g=g_1+ig_2\in H^1(\R^3)$, then $\LR{L_+g_1|g_1}=-\LR{L_-g_2|g_2}<0$ and $g_2=L_+g_1/\al$, hence such a function $g$ should live in one dimensional subspace, because of the spectrum of $L_+$.

The free operator $i\cL_0-\al:=i(1-\De)-\al$ is invertible 
\EQ{
 (i\cL_0-\al)^{-1}=-[(1-\De)^2+|\al|^2]^{-1}[i(1-\De)+\al]}
which can be written as a Fourier multiplier, and bounded $H^{-1}\to H^1$. Moreover, $i\cL-\al=(i\cL_0-\al)(I+K)$, where the operator $K$ is defined by 
\EQ{
 K\fy:=(i\cL_0-\al)^{-1}Q^2(\fy_2-3i\fy_1),}
and compact on $H^1$, hence $\Image(I+K)=\Ker(I+K^*)^\perp$. 
Noting that 
\EQ{
 (i\cL-\al)^* = -\cL i - \al = i(i\cL+\al)i,}
we have $\Ker(I+K^*)=(i\cL_0-\al)^*i\Ker(i\cL+\al)=\Span\{(i\cL_0-\al)^*ig_-\}$, and so 
\EQ{
 \Image(I+K)=(i\cL_0-\al)^{-1}(H^{-1}\cap(ig_-)^\perp). }
Since $g_+\not\in X:=H^1\cap(ig_-)^\perp$, we have $H^1=X\oplus\Span\{g_+\}$. 
This and $\Ker(I+K)=\Span\{g_+\}$ imply that $I+K$ is bijective $X\to\Image(I+K)$. Hence the equation $(i\cL-\al)f=h$ has the unique solution 
\EQ{
 f=(I+K)|_X^{-1}(i\cL_0-\al)^{-1}h \in X}
together with the boundedness $\|f\|_{H^1} \lec \|(i\cL_0-\al)^{-1}h\|_{H^1}
 \lec \|h\|_{H^{-1}}$. 
\end{proof}

Thus we have obtained a pair of eigenfunctions for $\om\gg 1$
\EQ{ \label{def gom}
 \pt (i\cL^\om \mp \al_\om) g^\om_\pm = 0,
 \pq g^\om_\pm = g^\om_1 \pm i g^\om_2, 
  \pq \al_\om\LR{ig^\om_+|g^\om_-}=2,} 
satisfying $(\al_\om,g^\om_\pm)=(\al,g_\pm)(1+O(\om^{-1/4}))$ in $\R\times H^1_r(\R^3)$. 
The eigenfunction $g^\om_+$ is not exactly the above $g^\om$, but it is normalized by a factor $1+O(\om^{-1/4})$ to realize the last identity of \eqref{def gom}. 

In using the virial identity around $Q_\om$, we will need that $\LR{(\bK^\om_2)'(Q_\om)|g^\om_1}>0$, which follows from $\LR{iQ|g_+}=\LR{Q|g_2}>0$. Indeed, 
\EQ{ \label{Qg2 pos}
 \pt\LR{(\bK^\om_2)'(Q_\om)|g^\om_1}
 =\LR{(L^\om_+/2+3L^\om_-/2-2-\cS_{3/2}'V^\om)Q_\om|g^\om_1}
 \pr=\al_\om\LR{Q_\om|g^\om_2}/2-\LR{\cS_{3/2}'V^\om Q_\om|g^\om_1}
 =\al\LR{Q|g_2}+O(\om^{-1/4})>\al\LR{Q|g_2}/2>0,}
if $\om$ is large enough. 

\subsection{Expansion of the rescaled energy}
Using the linearized operator and its spectral decomposition, we can expand 
\EQ{
 \pt \bA^\om(e^{i\te}(Q_\om + v))
 = \bA^\om(Q_\om) + \frac12\LR{\cL^\om v|v}-C^\om(v),}
for $v\in H^1_r$, where the cubic and quartic terms are collected into 
\EQ{ \label{def Com}
 C^\om(v):=\LR{|v|^2v|Q_\om}+\bG(v)=O(\|v\|_{H^1}^3).}
Expand $v$ by the eigenfunctions of $i\cL^\om$
\EQ{ \label{exp v}
 v = \lb_+g^\om_+ + \lb_-g^\om_- + \z
 = \lb_1g^\om_1 + \lb_2g^\om_2 + \z, } 
where $\lb_\pm,\lb_1,\lb_2\in\R$ are defined by 
\EQ{ \label{def Pom}
 \pt \lb_\pm := P^\om_\pm v := \pm\al_\om\LR{iv|g^\om_\mp}/2,
 \pr \lb_1:= P^\om_1 v := -\al_\om\LR{v_1|g^\om_2} =\lb_+ + \lb_- ,
 \pr \lb_2:= P^\om_2 v := -\al_\om\LR{v_2|g^\om_1} =\lb_+ - \lb_- ,}
such that 
\EQ{ \label{def Pomc}
 \z:= P^\om_\ce v:=v-(P^\om_+ v)g^\om_+ - (P^\om_- v)g^\om_-
 \implies 0=\LR{i\z|g^\om_\pm}=\LR{\z_1|g^\om_2}=\LR{\z_2|g^\om_1}.}
Then using $(i\cL^\om \mp \al_\om) g^\om_\pm=0$ and $\al_\om\LR{ig^\om_+|g^\om_-}=2\al_\om\LR{g^\om_1|g^\om_2}=2$, we obtain
\EQ{ \label{exp Aom}
 \bA^\om(Q_\om + v)-\bA^\om(Q_\om)
 \pt= - 2\lb_+ \lb_- + \frac12\LR{\cL^\om\z|\z}-C^\om(v)
 \pr= \frac12[-\lb_1^2+\lb_2^2+\LR{\cL^\om\z|\z}]-C^\om(v).}

If $\fy\in H^1_r(\R^3)$ is close to the rescaled excited states 
\EQ{ \label{def d0om}
 d_{0,\om}(\fy):=\inf_{\te\in\R}\|\fy-e^{i\te}Q_\om\|_{H^1} \ll 1,}
for some $\om$, then there exists a unique $\te\in \R/2\pi\Z$ such that 
\EQ{ \label{choice te}
 \fy=e^{i\te}(Q_\om+v) \implies 0=\LR{iQ_\om'|v}=\LR{iQ_\om'|e^{-i\te}\fy},
 \pq \|v\|_{H^1}\sim d_{0,\om}(\fy).}
Indeed, it is explicitly given by $\te=\arg(\fy|Q_\om')$, well-defined in the region $\LR{\fy|Q_\om'}>0$.  
This orthogonality is inherited by the radiation component $\z= P^\om_\ce v$, 
since $\LR{iQ_\om'|g^\om_\pm}=0$. 
The energy controls $P^\om_\ce v$ through the following coercivity.

\begin{lem}
There exists a constant $C\in[1,\I)$ such that for large $\om$ and any $\fy\in H^1_r(\R^3)$, we have 
\EQ{ \label{energy equiv}
 \|\fy\|_{H^1}^2/C \le \LR{\cL^\om\fy|\fy}+C\LR{\fy_1|g^\om_2}^2+C\LR{\fy_2|Q_\om'}^2 \le C^2\|\fy\|_{H^1}^2.}
\end{lem}
\begin{proof}
In the limit case $\om=\I$, namely for NLS without the potential, this is \cite[Lemma 2.2]{NLS}. The above estimate is just a perturbation of that, since 
$\cL^\om-\cL=O(\om^{-1/4})$ in $\B(H^1,H^{-1})$, 
$g^\om_2= g_2+O(\om^{-1/4})$ and $Q_\om'=Q'+O(\om^{-1/4})$ in $H^1$. 
Injecting these asymptotics into the limit estimate yields 
\EQ{
 \|\fy\|_{H^1}^2 \pt\lec \LR{\cL\fy|\fy}+C\LR{\fy_1|g_2}^2+C\LR{\fy_2|Q'}^2
 \pr=\LR{\cL^\om\fy|\fy}+C\LR{\fy_1|g^\om_2}^2+C\LR{\fy_2|Q_\om'}^2
 +O(\om^{-1/4}\|\fy\|_{H^1}^2),}
and then the left estimate of \eqref{energy equiv} after the last term is absorbed by the left side, while the other estimate of \eqref{energy equiv} is trivial. 
\end{proof}
In view of the expansion \eqref{exp Aom}, it is natural to introduce the following norm 
\EQ{ \label{def normom}
 \|v\|^2_\om \pt:= \frac 12\BR{(P^\om_1v)^2+(P^\om_2v)^2+\LR{Q_\om'|v_2}^2+\LR{\cL^\om P^\om_\ce v|P^\om_\ce v}}
 \pr= (P^\om_+v)^2+(P^\om_-v)^2+\frac12\LR{iQ_\om'|v}^2+\frac 12\LR{\cL^\om P^\om_\ce v|P^\om_\ce v}, }
which is equivalent to the $H^1$ norm on the radial subspace $H^1_r$, uniformly in $\om\gg 1$. 
Using this norm, the expansion is rewritten as 
\EQ{ \label{exp Aom2}
 \bA^\om(e^{i\te}(Q_\om+v))
 = \bA^\om(Q_\om) - \lb_1^2 + \|v\|^2_\om - C^\om(v),}
for any $v$ satisfying the orthogonality
\EQ{ \label{def Vom}
 v \in \V_\om:=\{\fy\in H^1_r(\R^3) \mid \LR{iQ_\om'|v}=0\}.}
Similarly, the orthogonal subspace for $P^\om_\ce v$ is denoted by 
\EQ{ \label{def Zom}
 \cZ_\om:=\{\z\in H^1_r(\R^3) \mid 0=\LR{i\z|Q_\om'}=\LR{i\z|g^\om_\pm}\}.}

The following lemma is a summary of this section. 
\begin{lem}\label{lem:sum}
There are constants $\mu_*,z_*\in(0,1)$, $\om_*\in(1,\I)$, and $C^1$ maps $(\Phi,\Om),\Psi$ satisfying \eqref{def PhiPsi}--\eqref{Soli by PhiPsi} and the following. 
For $\om\ge\om_*$, we have
\EQ{ \label{Vom small1}
 |\ml{V^\om}(\fy)|+|\ml{\cS_\I'V^\om}(\fy)|<\frac{1}{10}\|\fy\|_{H^1}^2,}
and (1)-(4) of Lemma \ref{lem:minid}. In particular, $\cQ_\om=\{e^{i\te}Q_\om\}$ is the set of minimizers for \eqref{Iom}, where $Q_\om=\rS_\om\Psi[\om]=Q+O(\om^{-1/4})$ in $H^1$. More specifically, 
\EQ{ \label{Qom close}
 \bA^\om(Q_\om)<\frac{11}{10}\bA(Q), \pq \bM\bH^0(Q_\om)<\frac{11}{10}\bM\bH^0(Q), \pq \bE^\om(Q_\om)>\frac 12\bE(Q)>0.}
The linearized operator $i\cL^\om$ has the generalized kernel
\EQ{
 i\cL^\om iQ_\om = 0, \pq i\cL^\om Q_\om'=-iQ_\om,}
where $Q'_\om=\rS_\om\om\Psi'[\om]=Q'+O(\om^{-1/4})$ in $H^1$, and a real eigenvalue $\al_\om=\al+O(\om^{-1/4})\in(\frac{9}{10}\al,\frac{11}{10}\al)$ with the eigenfunctions 
\EQ{
 i\cL^\om g^\om_\pm = \pm\al_\om g^\om_\pm,
 \pq g^\om_\pm = g^\om_1\pm ig^\om_2 = g_\pm+O(\om^{-1/4}) \IN{H^1},}
satisfying \eqref{Qg2 pos}. 
There are constants $\de_C,\de_D\in(0,1)$ such that for any $\om\ge\om_*$, 
\EQ{ \label{def sCom}
 \sC_\om:(\te,b_+,b_-,\z)\pt \mapsto e^{i\te}(Q_\om+b_+g^\om_++b_-g^\om_-+\z)}
is a diffeomorphism from the set 
\EQ{ \label{def sUom}
 \sU_\om:=\{(\te,b_+,b_-,\z)\in(\R/2\pi\Z)\times\R^2\times\Z_\om\mid |b_+|^2+|b_-|^2+\|\z\|_{H^1}^2<\de_C^2\}} 
into $H^1_r$, whose image contains the following neighborhood of $\cQ_\om$ 
\EQ{ \label{def sNom}
 \sN_\om:=\{\fy\in H^1_r(\R^3) \mid d_{0,\om}(\fy)<\de_D\} \subset \sC_\om(\sU_\om).}
From $\fy:=\sC_\om(\te,b_+,b_-,\z)\in\sN_\om$ we can recover 
\EQ{ \label{def tebz}
 \pt \te=\arg(\fy|Q_\om'),\pq (b_+,b_-,\z)=(P^\om_+,P^\om_-,P^\om_\ce)(e^{-i\te}\fy-Q_\om),
 \pr b_1:=b_++b_-, \pq b_2:=b_+-b_-,}
where $\LR{\fy|Q_\om'}>0$. 
For later use, they are denoted by 
\EQ{ \label{def sBsZ}
 \sB^\om_*(\fy):=b_* \pq(*=+,-,1,2), \pq \sZ^\om(\fy):=\z}
The linearized energy norm $\|v\|_\om$ is equivalent to the $H^1$ norm uniformly for $\om\ge\om_*$. 
\end{lem}

\section{Dynamics around the first excited state}
\subsection{Expansion around the excited state}
For any solution $u$ of \eqref{NLSP} and any $\om>0$, consider the parabolic rescaling which preserves the equation and $\dot H^{1/2}(\R^3)$ 
\EQ{
 u_\om(t,x):=\rS_\om u(t/\om) = \om^{-1/2} u(\om^{-1}t,\om^{-1/2}x).}
Then \eqref{NLSP} is rescaled to 
\EQ{ \label{rscNLS}
 (i\p_t-\De+V^\om)u_\om = |u_\om|^2u_\om.}
We say that a solution $u_\om$ of \eqref{rscNLS} either {\it blows up}, {\it scatters to $\Phi$}, or {\it is trapped by $\Psi$}, if it happens for the unscaled solution $u(t):=\rS_\om^{-1}u_\om(\om t)$ (cf.~Section \ref{ss:tob}). 

Suppose that $u_\om$ is close to the orbit of the excited state $\cQ_\om$ at some $t$. 
More precisely, assume $u_\om(t)\in\sN_\om$ or $d_{0,\om}(u_\om(t))<\de_D$ for some $\om\ge\om_*$ at some $t\in\R$. 
Here we could restrict $\om$ by specifying the mass $\bM(u_\om)=\bM(Q_\om)$ or equivalently $\bM(u)=\bM(\Psi[\om])$ as in \cite{NLS}, but it is more convenient to keep the freedom of $\om$ in constructing the center-stable manifold (see Section \ref{s:mfd}). 

Expanding the solution $u_\om$ of \eqref{rscNLS} in the form
\EQ{
 u_\om(t,x)=e^{i\te(t)}(Q_\om(x)+v(t,x))}
with $\te(t)\in\R/2\pi\Z$ and $v(t)\in H^1(\R^3)$ yields an equation for $v$ 
\EQ{ \label{eq v0}
 i\dot v= -\cL^\om v +(\dot\te+1)(Q_\om+v)  + N^\om(v),}
where $N^\om:H^1\to H^{-1}$ is the Fr\'echet derivative of $C^\om$ given by 
\EQ{ \label{def Nom}
 N^\om(v):= 2Q_\om |v|^2 + Q_\om v^2 + |v|^2v.}
In the real value, the equation is written as 
\EQ{
 \CAS{ \dot v_1 = -L^\om_- v_2  + (\dot\te+1)v_2     + 2Q_\om v_1v_2+|v|^2v_2, 
 \\ \dot v_2 = L^\om_+ v_1 - (\dot\te+1)(Q_\om+v_1) - Q_\om(3v_1^2+v_2^2)-|v|^2v_1.} }

\subsection{Orthogonality and equations}
In order to exploit the coercivity of $\cL^\om$, we choose $\te(t)$ by the local coordinate $\sC_\om$, see \eqref{def tebz}, or the orthogonality 
\EQ{ \label{def te}
 0=\LR{iQ_\om'|v}=\LR{iQ_\om'|e^{-i\te}u_\om}, \pq \|v\|_{H^1}\sim d_{0,\om}(u_\om).}
Differentiating the above orthogonality condition in $t$ yields
\EQ{
 \pt 0=\p_t\LR{iv|Q'_\om}=\LR{v|Q_\om}+(\dot\te+1)\LR{Q_\om+v|Q'_\om}+\LR{N^\om(v)|Q'_\om},}
which can be rewritten as an equation for $\te(t)$
\EQ{ \label{eq te}
 \dot\te+1 = m^\om(v),}
where $m^\om(v)$ is defined and $C^1$ for small $v\in H^1_r$ by the equation 
\EQ{ \label{eq m}
 0=\LR{Q_\om+v|Q'_\om}m^\om(v)+\LR{v|Q_\om}+\LR{N^\om(v)|Q'_\om}.}
Since $\LR{u_\om|Q'_\om}>0$ as long as $u_\om\in\sN_\om$ (cf.~\eqref{def tebz}), $m^\om(v)$ is well-defined, satisfying 
\EQ{ \label{est mom}
 |m^\om(v)| \lec \|P^\om_\ce v\|_2+\|v\|_{H^1}^2,} 
since $\LR{g^\om_\pm|Q_\om}=0$.  
Plugging it into \eqref{eq v0} yields an autonomous equation of $v$ 
\EQ{ \label{eq v}
 \pt \dot v = i\cL^\om v -m^\om(v)iQ_\om + \cN^\om(v),
 \pq \cN^\om(v):=-i(m^\om(v)v+ N^\om(v)).}
It can be rewritten in the local coordinate of $\sC_\om$. Denoting
\EQ{ \label{def cNom*}
 \pq \cN^\om_*(v):=P^\om_* \cN^\om(v)}
for $*=\pm,1,2,\ce$, we obtain the following equations for each $\lb_*=P^\om_* v$
\EQ{ \label{eq la}
 \pt \dot\lb_\pm = \pm\al_\om \lb_\pm + \cN^\om_\pm(v), 
 \pq \CAS{ \dot \lb_1 = \al_\om \lb_2 + \cN^\om_1(v),\\ \dot \lb_2 = \al_\om \lb_1 + \cN^\om_2(v),}}
as well as for $\z=P^\om_\ce v$
\EQ{
 \dot\z = i\cL^\om \z - m^\om(v)iQ_\om + \cN^\om_\ce(v).}

\subsection{Energy distance function}
In view of the expansion \eqref{exp Aom2}, the linearized energy norm $\|v\|_\om$ is better suited than $d_{0,\om}$ to measure the distance from $\Soli_1$ to solutions $u$.  
In order to avoid the regularity loss from the higher order term $C^\om$, we can either include it into the distance, or mollify the distance in time. 
Here we take the latter option for a simpler (convex) dynamics of the (square) distance. 

Fix a radial decreasing function $\chi\in C_0^\I(\R)$ satisfying  
\EQ{ \label{def chi}
 \chi(t)=\CAS{1 &|t|\le 1, \\ 0 &|t|\ge 2.}}
For any $\fy\in\sN_\om$, let $u_\om$ be the solution of \eqref{rscNLS} with $u_\om(0)=\fy$. 
Decomposing $u_\om=e^{i\te}(Q_\om+v)$ as above, a local distance function $d_{1,\om}$ is defined at $\fy$ by 
\EQ{ \label{def d1om}
 d_{1,\om}(\fy)^2:=\int_\R \chi(-t)\|v(t)\|^2_\om dt.}
The local wellposedness for \eqref{eq v} in $v\in H^1(\R^3)$, which is uniform in $\om$, yields 
\begin{lem} \label{lem:LWP}
There are constants $\de_E\in(0,\de_D/2]$ and $C\in[1,\I)$ such that for any $\om\ge\om_*$ and any $\fy\in H^1_r(\R^3)$ with $d_{0,\om}(\fy)\le2\de_E$, the solution $u_\om$ of \eqref{rscNLS} with the initial condition $u_\om(0)=\fy$ exists at least for $|t|\le 2$, satisfying 
\EQ{ \label{time equivalence}
 |t|\le 2\implies C^{-1}d_{0,\om}(u_\om(0)) \le d_{0,\om}(u_\om(t)) \le C d_{0,\om}(u_\om(0)).}
\end{lem}
Hence $d_{1,\om}$ is well-defined for $d_{0,\om}(\fy)\le2\de_E$ and uniformly equivalent to $d_{0,\om}$. 
Then a global distance function $d_\om:H^1_r\to[0,\I)$ is defined by 
\EQ{ \label{def dom}
 d_\om(\fy):=\chi(d_{0,\om}(\fy)/\de_E)d_{1,\om}(\fy)+(1-\chi(d_{0,\om}(\fy)/\de_E))d_{0,\om}(\fy).}
$d_\om:H^1_r\to[0,\I)$ satisfies $d_\om(\fy)\sim d_{0,\om}(\fy)$ uniformly and 
\EQ{ \label{d 2 d1}
 d_{0,\om}(\fy)\le\de_E \implies d_\om(\fy)=d_{1,\om}(\fy).}

\subsection{Instability and ejection}
The crucial property of the dynamics around $\Soli_1$ is that the rescaled solution $u_\om$ can get away from $Q_\om$ only by growing instability. More precisely, we have
\begin{lem} \label{lem:inst}
There are constants $c_X\in(0,1)$ and $\de_I\in(0,\de_E]$ such that for any $\om\ge\om_*$ and any $\fy\in H^1_r(\R^3)$ we have uniformly
\EQ{ \label{instdom}
 \SAC{\bA^\om(\fy)-\bA^\om(Q_\om) \le c_X d_\om(\fy)^2 \\
  \tand d_{\om}(\fy)\le\de_I}
 \implies d_\om(\fy)=d_{1,\om}(\fy)\sim|\sB^\om_1(\fy)|.}
\end{lem}
\begin{proof}
Since $d_\om\sim d_{0,\om}$ and \eqref{d 2 d1}, choosing $\de_I$ small ensures that $d_\om(\fy)=d_{1,\om}(\fy)$ for $d_\om(\fy)\le\de_I$.  
Then by the definition of $d_{1,\om}$ and equivalence of distance functions, 
\EQ{
 \|v\|^2_\om=\bA^\om(\fy)-\bA^\om(Q_\om)+\lb_1^2-C^\om(v)
 \lec (c_X+\de_I) \|v\|^2_\om + \lb_1^2,}
where $v\in\V_\om$ is determined from $\fy$ by \eqref{choice te} as before. 
Choosing $c_X$ and $\de_I$ small enough, we obtain 
$d_\om(\fy)^2\sim \|v\|^2_\om\lec\lb_1^2$. 
\end{proof}

Next we investigate the evolution of $d_\om$. 
For any solution $u_\om$ of \eqref{rscNLS} close to $Q_\om$,  
\EQ{
 d_{1,\om}(u_\om)^2 = \chi*\|v\|^2_\om,}
where $*$ denotes the convolution in $t$, and $u_\om=e^{i\te}(Q_\om+v)$ with the orthogonality $v\in\V_\om$ as before. 
Then using the equation \eqref{eq la} for $\lb_j$, \eqref{exp Aom2} and conservation of $\bA^\om(u_\om)$, we derive 
\EQ{
 \pt \p_td_{1,\om}(u_\om)^2
  =\chi*2\al_\om[\lb_1\lb_2+O(\|v\|_{H^1}^3)]+\chi'*O(\|v\|_{H^1}^3), 
 \pr \p_t^2d_{1,\om}(u_\om)^2
  =\chi*2\al_\om^2[\lb_1^2+\lb_2^2]+\sum_{j=0}^2 \chi^{(j)}*O(\|v\|_{H^1}^3).}
Note that we can not differentiate the cubic terms, that is the reason for the mollifier. 
If $u_\om(t)$ is in the instability dominance \eqref{instdom}, then 
\EQ{ \label{eq d2}
 \pt \p_td_\om(u_\om)^2
  =\chi*2\al_\om[\lb_1\lb_2]+O(\lb_1^3), 
 \pr \p_t^2d_\om(u_\om)^2
  =\chi*2\al_\om^2[\lb_1^2+\lb_2^2]+O(\lb_1^3)
 \sim \lb_1^2,}
using the equivalence \eqref{time equivalence} as well, and assuming if necessary that $d_\om(u_\om)$ is even smaller. 
The last equation implies the convexity of $d_{1,\om}^2$, from which we deduce that if $u_\om$ satisfies at some time $t=t_0$, 
\EQ{ \label{eject0}
 \pt d_\om(u_\om) < \de_X,  
 \pq \bA^\om(u_\om)-\bA^\om(Q_\om) \le c_X d_\om(u_\om)^2, 
  \pq \p_t d_{\om}(u_\om)\ge 0,}
for some small $\de_X\le\de_I$, 
then $d_\om(u_\om)$ will keep growing for $t\ge t_0$ until the first condition is violated. 
$d_\om(u_\om)\sim|\lb_1|$ implies that $\si:=\sign\lb_1(t)\in\{\pm 1\}$ is fixed during that time. 

Since $0\le \p_t d_\om(u_\om)=\p_t d_{1,\om}(u_\om)$, the first equation in \eqref{eq d2} implies that $\lb_1\lb_2(t_1)\gec-|\lb_1(t_1)|^3$ at some $t_1\in(t_0-2,t_0+2)$, then (taking $\de_X$ small), $\si\lb_+(t_1)\ge|\lb_1(t_1)|/3$ and $\si\lb_-(t_1)\ge 0$. 
Because $\p_t(\lb_1\lb_2)\ge 0$ and $\lb_1\sim\lb_1(t_0)$ for $|t-t_0|<2$, we may choose $t_1>t_0$. 

Let $R:=|\lb_1(t_1)|$ and suppose that on some interval $[t_1,t_2]$ we have \eqref{eject0} and 
\EQ{
 |\lb_1| \le 2Re^{\al_\om(t-t_1)}.}
Then the equations of $\lb_\pm$ in \eqref{eq la} imply 
\EQ{
 |\lb_\pm - e^{\pm\al_\om(t-t_1)}\lb_\pm(t_1)| \lec (Re^{\al_\om(t-t_1)})^2 \lec \de_X Re^{\al_\om(t-t_1)},}
and so, taking $\de_X$ smaller if necessary, 
\EQ{
 |\lb_1|=\si(\lb_++\lb_-)\CAS{ \le Re^{\al_\om(t-t_1)}(1+C\de_X)<2Re^{\al_\om(t-t_1)}, \\ \ge Re^{\al_\om(t-t_1)}(1/3-C\de_X)>Re^{\al_\om(t-t_1)}/4.}}
Therefore $t_2$ can be increased until $d_\om(u_\om)$ reaches $\de_X$ at some $t=t_X>t_0$, and for $t_0\le t\le t_X$, we have
\EQ{
 d_\om(u_\om) \sim \si \lb_1(t) \sim \si \lb_1(t_0)e^{\al_\om(t-t_0)},}
for a time-independent sign $\si\in\{\pm 1\}$, 
while the equations of $\lb_\pm$ imply 
\EQ{
 \lb_\pm(t) = e^{\pm\al_\om(t-t_0)}\lb_\pm(t_0)+O(\lb_1^2).}
To estimate $\z$, consider the energy projected onto the eigenmodes
\EQ{
 \bA^\om(\lb g^\om)=\frac{1}{2}\BR{-\lb_1^2+\lb_2^2} - C^\om(\lb g^\om),\pq \lb g^\om:=\lb_1g^\om_1+i\lb_2g^\om_2.} 
Using the equations of $\lb$, we have, for $t_0<t<t_X$, 
\EQ{
 \pn \p_t\bA^\om(\lb g^\om)
 \pt=-\lb_1(\al_\om\lb_2+\cN^\om_1(v))+\lb_2(\al_\om\lb_1+\cN^\om_2(v))-\LR{N^\om(\lb g^\om)|\dot\lb g^\om}
 \pr=\al_\om m^\om(v)[\lb_1\LR{v_2|g^\om_2}+\lb_2\LR{v_1|g^\om_1}]
 \pn=O(\|\z\|_\om\|v\|^2_\om+\|v\|_\om^4),}
where \eqref{est mom} is used. On the other hand 
\EQ{
 \bA^\om(u_\om)-\bA^\om(Q_\om)-\bA^\om(\lb g^\om)\pt=\frac 12\LR{\cL^\om\z|\z}-C^\om(v)+C^\om(\lb g^\om)
 \pr= \|\z\|^2_\om + O(\|\z\|_\om\|v\|^2_\om).}
Hence integrating its time derivative in $[t_0,t_X]$ leads to  
\EQ{
 \|\z\|_{L^\I_tH^1}^2 \lec \|\z(t_0)\|_{H^1}^2 + |\lb_1(t_0)|^4 + \|\lb_1\|_{L^4_t}^4,}
so, using the exponential growth of $\lb_1$, 
\EQ{
 \|\z(t)\|_{H^1} \lec \|\z(t_0)\|_{H^1} + |\lb_1(t)|^2.}
Near $t=t_X$, we can also determine the sign and size of 
\EQ{
 \bK^\om_2(u_\om)\pt=\LR{(\bK^\om_2)'(Q_\om)|v}+O(\|v\|_{H^1}^2)
 \pr=\lb_1\LR{(\bK^\om_2)'(Q_\om)|g^\om_1}+O(\|\z_1\|_\om+\|v\|^2_\om).}
\eqref{Qg2 pos} implies that for some constant $\CK>0$ 
\EQ{
 \si\bK^\om_2(u_\om)+\CK\|\z(t_0)\|_{H^1} \sim \si\lb_1 = |\lb_1|}
on $t_0\le t\le t_X$. 
Thus we have obtained the following. 
\begin{lem}[Ejection Lemma] \label{lem:eject}
There are constants $\CK\in(0,\I)$ and $\de_X\in(0,\de_I]$ such that for any $\om\ge\om_*$ and any solution $u_\om$ of \eqref{rscNLS} satisfying the three conditions in \eqref{eject0} at some $t=t_0\in\R$, 
there exists $t_X\in(t_0,\I)$ such that $d_\om(u_\om(t_X))=\de_X$, and for $t_0<t<t_X$, $d_\om(u_\om(t))$ is strictly increasing, 
\EQ{ \label{ejection}
 d_\om(u_\om(t)) \pt\sim \si\sB^\om_1(u_\om(t)) \sim \si\sB^\om_+(u_\om(t)) \sim \si\sB^\om_1(u_\om(t_0))e^{\al_\om(t-t_0)} 
 \pr\sim \si\bK^\om_2(u_\om(t))+\CK\|\sZ^\om(u_\om(t_0))\|_{H^1},}
and $\sB^\om_\pm(u_\om(t)) = e^{\pm\al_\om(t-t_0)}\sB^\om_\pm(u_\om(t_0))+O(d_\om(u_\om(t))^2)$, 
for some $\si\in\{\pm 1\}$ independent of $t\in(t_0,t_X)$. 
\end{lem}
\begin{rem}
In the previous papers such as \cite{NLS}, the sign was opposite between the unstable mode and the virial functional. 
In this paper, the sign of the eigenmode is chosen to match that of virial. In other words, the sign of eigenmode is flipped from \cite{NLS}, by the choice (normalization) of $g^\om_\pm$. 
\end{rem}
Note that by the time inversion symmetry, we can and will apply the above lemma backward in time as well. 
As an immediate consequence, we can describe the behavior of solutions which are not ejected. 
In contrast to the ejected solutions, they are monotonically and exponentially attracted by a small neighborhood of $\cQ_\om$. 
\begin{lem}[Trapping Lemma] \label{lem:stay}
Let $\om\ge\om_*$ and $u_\om$ be a solution of \eqref{rscNLS} on some interval $[t_0,\I)$ satisfying 
\EQ{ \label{staying}
 \sup_{t_0<t<\I}d_\om(u_\om(t))<\de_X.}
Then there exists $t_1\in[t_0,\I]$ such that $d_\om(u_\om(t))$ is strictly decreasing on $[t_0,t_1)$ and 
\EQ{
 \CAS{t_0\le t<t_1\implies d_\om(u_\om(t)) \sim e^{-\al_\om(t-t_0)}d_\om(u_\om(t_0)), \\
 t_1<t<\I \implies c_Xd_\om(u_\om(t))^2 < \bA^\om(u_\om)-\bA^\om(Q_\om).}} 
We have $t_1=\I$ if and only if $u_\om(t)$ converges to $e^{i(a-t)}Q_\om$ strongly in $H^1(\R^3)$ as $t\to\I$ for some $a\in\R$. Moreover, in that case we have 
\EQ{
 \|u_\om(t)-e^{i(a-t)}Q_\om\|_{H^1} \sim e^{-\al_\om(t-t_0)}d_\om(u_\om(t_0)).} 
\end{lem}
\begin{proof}
Abbreviate $d(t):=d_\om(u_\om(t))$. 
The ejection lemma \ref{lem:eject} implies that if \eqref{eject0} holds at any $t\in[t_0,\I)$ then $d(t)$ grows at least to $\de_X$, violating the first condition of \eqref{staying}. 
Since $d(t)<\de_X$ for all $t\ge t_0$, the latter two conditions of \eqref{eject0} should never hold, in other words
\EQ{
 \p_t d(t)<0 \pq \text{ or } \pq c_Xd(t)^2 < \bA^\om(u_\om)-\bA^\om(Q_\om).}
Then $t_1:=\inf\{t\ge t_0\mid c_Xd(t)^2<\bA^\om(u_\om)-\bA^\om(Q_\om)\}$ satisfies the desired properties. 
The exponential decay follows from Lemma \ref{lem:eject} applied backward in time 
\EQ{
 d(t_0) \sim e^{\al_\om(t_0-t)}d(t)}
for $t_0<t<t_1$. If $t_1=\I$ then $d(t)\sim e^{-\al_\om(t-t_0)}d(t_0)\to 0$ as $t\to\I$, hence exponential convergence to the set $\cQ_\om$. 
Conversely, if $u_\om$ is strongly convergent, then $\bA^\om(u_\om)=\bA^\om(Q_\om)$, which forces $t_1=\I$. 
Then the modulation equation \eqref{eq te} yields 
$|\dot\te+1|\lec \|v\|_{H^1}\lec e^{-\al_\om(t-t_0)}\de$, 
and by integration in $t$, the same bound for $|\te-(a-t)|$ for some $a\in\R$. 
\end{proof}

Under an energy constraint 
$\bA^\om(u_\om)<\bA^\om(Q_\om)+c_X\de^2$ 
for some $\de\in(0,\de_X)$, every solution $u_\om$ satisfying \eqref{staying} comes closer to $\cQ_\om$, namely 
\EQ{
 t_1<t<\I \implies d_\om(u_\om(t))<\de<\de_X.}
This distance gap between the ejection and the trapping is a key property of the instability dynamics. 

\section{Static analysis away from the first excited state}
When the solution is away from the first excited states, the above linearization is useless. Instead, we rely on energy-type, variational and topological arguments. 
\subsection{Variational estimate away from the solitons}
We have sign-definiteness of the virial $\bK^\om_2$ away from the solitons. 
\begin{lem} \label{lem:var}
There exist continuous, positive and increasing functions $\e_V$ and $\ka_V$ from $(0,\I)$ to $(0,1)$ with the following property. 
Let $\om\ge\om_*$ and $\fy\in H^1_r(\R^3)$ satisfy the following three conditions: 
\EQ{ \label{var cond}
 d_\om(\fy)\ge\de,\pq \bA^\om(\fy)<\bA^\om(Q_\om)+\e_V(\de)^2, 
 \pq \bM+\om\bH^0(\fy)>\CM.}
Then we have one of the following \rm{(a)-(b)}.
\ENA{
\item $|\bK^\om_2(\fy)|\ge\ka_V(\de)$.
\item $\bH^0(\fy)\le \CD\bK^\om_2(\fy)\le \CD^2\bH^0(\fy)$ and $\bE^\om(\fy)<\de_U/2$, where 
\EQ{
 \de_U:= \inf_{\om\ge\om_*}\bE^\om(Q_\om)>0.}
}
\end{lem}
Note that the first and the third conditions in \eqref{var cond} are to avoid the sign change of $\bK^\om_2$  respectively around the first excited state and around the ground state, but we can not avoid the vanishing at $0$, namely $\|\na\fy\|_2\to 0$ as $\om\to\I$, corresponding to the case (b). 
$\de_U>0$ is ensured by \eqref{Qom close}. 
\begin{proof}
We argue by contradiction. 
Let $(\fy,\om)=(\fy_n,\om_n)$ be a sequence in $H^1_r\times[\om_*,\I)$ such that as $n\to\I$, 
\EQ{ \label{min seq}
 d_\om(\fy)\ge\de,\pq \bA^\om(\fy) < \bA^\om(Q_\om)+o(1), \pq (\bM+\om\bH^0)(\fy)>\CM,\pq \bK^\om_2(\fy)\to 0,}
and that $\fy$ does not satisfy (b). 
Combining the above with \eqref{AK bd H1}, \eqref{Vom small1} and \eqref{Qom close} yields 
\EQ{ \label{H1 bd fy}
 (\bM+\bH^0/3)(\fy)<2\bA(Q),}
for large $n$, so that we can extract a subsequence such that $\om$ converges to some $\OM\in[\om_*,\I]$ and that $\fy$ converges to some $\fy_\I$ weakly in $H^1_r$ and strongly in $L^4$. 
The convergence implies  
\EQ{
 \bJ^{\OM}(\fy_\I) \le \bJ^{\OM}(Q_\OM), \pq \bK^{\OM}_2(\fy_\I)\le 0.}
Apply Lemma \ref{lem:dich} to $\fy$. Lemma \ref{lem:minid} (1) with \eqref{H1 bd fy} precludes the case \eqref{dich1}. 
Since \eqref{dich2} with $\bK^\om_2(\fy)\to 0$ would lead to (b) for large $n$, we deduce that $\fy$ is in the case \eqref{dich3}, so is the weak limit $\fy_\I$. 
Then Lemma \ref{lem:minid} implies that $\fy_\I=e^{i\te}Q_\OM$ for some $\te\in\R$, and so $\bA^\om(\fy)\to\bA^\OM(\fy_\I)$, which implies that the convergence $\fy\to e^{i\te}Q_\OM$ is strong in $H^1$, contradicting $d_\om(\fy)\ge\de$. 
\end{proof}

\subsection{Sign functional}
Combining the above Lemmas \ref{lem:eject} and \ref{lem:var}, we see that the sign $\si$ in \eqref{ejection} can be given by a functional on two separate open sets in $H^1_r$ away from $Q_\om$. More precisely, we have 
\begin{lem} \label{lem:sign} 
There exist constants $\de_V\in(0,\de_X/2)$ and $\e_S\in(0,\e_V(\de_V))$ such that for each $\om\ge \om_*$ there exists a unique continuous functional 
\EQ{
 \Sg_\om: \{\fy\in H^1_r(\R^3) \mid \bA^\om(\fy)-\bA^\om(Q_\om)<\min(\e_S^2,c_X d_\om(\fy)^2)\}=:\ck\cH_\om \to \{\pm 1\}}
satisfying {\rm(i)--(iii)} below. For any $\fy\in\ck\cH_\om$ and $\te\in\R$, 
\ENI{
\item If $\CS \bM\bH^0(\fy)\le 1$ then $\Sg_\om(\fy)=+1$. 
\item If $d_\om(\fy)<2\de_V$ then $\Sg_\om(\fy)=\sign \sB^\om_1(\fy)$. 
\item If $\bA^\om(\fy)-\bA^\om(Q_\om)<\e_V(d_\om(\fy))^2$ and $(\bM+\om\bH^0)(\fy)>\CM$, then $\Sg_\om(\fy)=\sign\bK^\om_2(\fy)$. 
}
Moreover, if $\fy\in H^1_r(\R^3)$ satisfies $\rS_{\om\oj}\fy\in\ck\cH_{\om\oj}$ for some $\om\zr,\om\on\ge\om_*$, then 
\EQ{
 \Sg_{\om\zr}(\rS_{\om\zr}\fy)=\Sg_{\om\on}(\rS_{\om\on}\fy).} 
\end{lem}
If $\bA^\om(\fy)<\bA^\om(Q_\om)+\e_S^2$ and $d_\om(\fy)>\de_V$, then $\bA^\om(\fy)-\bA^\om(Q_\om)<\e_V(\de_V)^2<\e_V(d_\om(\fy))^2$. If $(\bM+\om\bH^0)(\fy)\le\CM$ then $\CS\bM\bH^0(\fy)\le 1$ by Lemma \ref{lem:minid}(2). 
Hence (i)-(iii) determine the value of $\Sg_\om$ on $\ck\cH_\om$, which is independent of the choice of $\de_V$ and $\e_S$ (because (i) and (iii) are independent). 
The continuity of $\Sg_\om$ simply means that it is constant on each connected component of $\ck\cH_\om$. 
The last sentence of lemma allows us to define a functional independent of $\om$ 
\EQ{
 \pt \Sg:\ck\cH \to \{\pm 1\}, \pq \Sg(\fy):=\Sg_\om(S_\om \fy),
 \prQ \ck\cH:=\{\fy\in H^1_r(\R^3)\mid \exists\om>\om_*,\ \rS_\om\fy\in\ck\cH_\om\}.}
As a sufficient condition for $\fy\in\ck\cH$, using $\inf_{\om>0}\bA^\om(\rS_\om\fy)=2\sqrt{\bE(\fy)\bM(\fy)}$ and 
\EQ{ \label{H1/2 dist}
 d_\om(\rS_\om\fy)\sim\dist_{H^1}(\rS_\om\fy,\cQ_\om)
 \gec \dist_{\dot H^{1/2}}(\fy,\{e^{i\te}\Psi[\om]\}_{\te}),}
we see that there exist $0<\mu,\e,c\ll 1$ such that 
$\fy\in H^1_r(\R^3)$ belongs to $\ck\cH$ if 
\EQ{
 \bM(\fy)<\mu \tand \bE(\fy)\bM(\fy)<\bE^0(Q)\bM(Q)+\min(\e^2,c\dist_{\dot H^{1/2}}(\fy,\Soli_1)).}

Notice that $\si$ in the ejection lemma \ref{lem:eject} is {\it not necessarily} equal to $\Sg_\om(u_\om(t_X))$, but it is so if the solution is well accelerated at the ejection time $t_X$, that is the case if $d_\om(u_\om(t_0))\ll\de_X$. 
In any case, the sign functional $\Sg_\om$ will give the correct prediction of dynamics after the ejection. 
It is also worth noting 
\begin{lem} \label{lem:unifbd}
$\Sg_\om^{-1}(\{+1\})$ is uniformly bounded in $H^1$ for $\om\ge\om_*$.
\end{lem} 
\begin{proof}
It is obvious in the case (ii) of Lemma \ref{lem:sign}, because $Q_\om$ is bounded. 
In the case (iii), the uniform bound follows from \eqref{AK bd H1} and $\bK_2^\om(\fy)\ge 0$. 
In the case (i), using \eqref{def CS} and that $Q$ attains the best constant in Gagliardo-Nirenberg $\bG\lec(\bM\bH^0)^{1/2}\bH^0$, we obtain
\EQ{
 \bG(\fy)\le \frac{\bG(Q)}{\bM(Q)^{1/2}\bH^0(Q)^{3/2}}\CS^{-1/2}\bH^0(\fy)
 \le \frac{\bG(Q)}{\bH^0(Q)}\bH^0(\fy)=\frac23 \bH^0(\fy),}
where we also used the Pohozaev identity, cf.~\eqref{asy fy funct}. Using \eqref{Vom small1} as well, we obtain 
\EQ{ \label{bd in (i)}
 \rm{(i)}\implies \bA^\om(\fy)=(\bH^0+\bM+\ml{V^\om}-\bG)(\fy) \ge \frac{2}{9} \bH^0(\fy)+\frac9{10}\bM(\fy).}
Since the cases (i)-(iii) exhaust the region $\ck\cH_\om$ as seen above, we conclude that $\Sg_\om^{-1}(\{+1\})$ is uniformly bounded. 
\end{proof}

\begin{proof}[Proof of Lemma \ref{lem:sign}]
Fix $0<\de_V\ll\de_X$ and $\e_S\in(0,\e_V(\de_V))$. 
To show that $\Sg_\om$ is uniquely, continuously and well defined by (i)-(iii), it suffices to show that (i), (ii) and (iii) do not contradict in the intersections. 

There is no intersection of (i) and (ii) because of \eqref{def CS} and \eqref{Qom close}, if $\de_V>0$ is small enough. 
Choosing $\e_S$ small enough and using \eqref{bd in (i)}, $\fy\in\ck\cH_\om$ and \eqref{Qom close}, we have 
\EQ{
 \rm{(i)}\implies \bM(\fy)<\frac{10}9(\bA^\om(Q_\om)+\e_S^2)<2\bA(Q).} 
Hence in the intersection of (i) and (iii), Lemma \ref{lem:minid} (1)-(2) precludes \eqref{dich1} and \eqref{dich3}, then \eqref{dich2} implies $\bK^\om_2(\fy)>0$. 

For the intersection of (ii) and (iii), let $\fy\in\ck\cH_\om$ satisfy $\bA^\om(\fy)<\bA^\om(Q_\om)+\e_V(d_\om(\fy))$ and $d_\om(\fy)<2\de_V$. 
Let $u_\om$ be the solution of \eqref{rscNLS} with $u_\om(0)=\fy$. 
Since $\fy$ satisfies \eqref{instdom}, the ejection lemma \ref{lem:eject} is applied to $u_\om$, either forward or backward in time from $t=0$. 
In both cases, there exists $t_X\in\R$ such that $d_\om(u_\om)\in(d_\om(\fy),\de_X)$ is monotone between $t=0$ and $t_X$, with $d_\om(u_\om(t_X))=\de_X$. 
Since $d_\om(\fy)<2\de_V\ll\de_X$, \eqref{ejection} implies that 
$\si=\sign\bK^\om_2(u_\om(t_X))=\sign b_1(0)$. 
Since $\bA^\om(u_\om)-\bA^\om(Q_\om)<\e_V(d_\om(\fy))^2\le\e_V(u_\om)^2$ between $t=0$ and $t_X$, the variational lemma \ref{lem:var} implies that $\sign\bK^\om_2(u_\om)$ also remains unchanged. Note that the case (b) of Lemma \ref{lem:var} is precluded by $d_\om(\fy)<\de_V$, since it implies $\bE^\om(\fy)>\bE^\om(Q_\om)/2\ge\de_U/2$ if $\de_V$ is small enough. 
Hence $\sign\bK^\om_2(u_\om)=\sign b_1(0)$, so (ii) and (iii) define the same value of $\Sg_\om$ for $\fy$. 
Therefore $\Sg_\om$ is well defined and continuous. 

To show the invariance with respect to $\om$, let $\fy\in H^1_r(\R^3)$ satisfy $\rS_{\om\oj}\fy\in\ck\cH_{\om\oj}$. 
Let $u$ be the solution of \eqref{NLSP} with $u(0)=\fy$ and let $u\oj:=\rS_{\om\oj}u$. Then $u\oj(t/\om\oj)$ is the solution of \eqref{rscNLS} with $u\oj(0)=\rS_{\om\oj}\fy$ and $\om=\om\oj$. 

Suppose that $\Sg_{\om\zr}(u\zr(0))\not=\Sg_{\om\on}(u\on(0))$ and let $I\ni 0$ be the maximal time interval where $u\oj$ remains in $\ck\cH_{\om\oj}$ for both $j=0$ and $j=1$. 
The discrepancy of $\Sg_\om$ implies that either $u\zr(t)$ or $u\on(t)$ is in the case (ii) at each $t\in I$, since $\bM\bH^0(\rS_\om\fy)$ and $\bK^\om_2(\rS_\om\fy)=\bK_2(\fy)$ are independent of $\om$. 
Suppose that $u\zr(0)$ is in the case (ii). 
By the ejection lemma as above, $u\zr$ exits (ii) into the region (iii) either forward or backward in time. 
Meanwhile, $u\on$ must either enter the region (ii) or exit $\ck\cH_{\om\on}$. 
Since the solution does not blow up, exiting $\ck\cH_{\om\on}$ is possible only through the region (ii). 
Since $\{t\in I\mid\rm(ii)\}$ is open for each solution, we deduce that at some $t=t_0\in I$ both $u\zr$ and $u\on$ are in the case (ii). 

Decompose $u\oj(t_0)$ around $Q_{\om\oj}$ as before 
\EQ{
 v\oj:=e^{-i\te\oj}u\oj(t_0)-Q_{\om\oj}\in\V_{\om\oj}, \pq \lb_1\oj:=P^{\om\oj}_1 v\oj,}
then $d_{\om\oj}(u\oj(t_0))\sim\|v\oj\|_{H^1}\sim|\lb_1\oj|$. 
Using \eqref{H1/2 dist}, we have 
\EQ{
 |\diff\log\om\pa| \pn\sim \|\diff\Psi[\om\pa]\|_{\dot H^{1/2}}
 \pt\lec \sum_{j=0,1}\dist_{\dot H^{1/2}}(u(t_0),\{e^{i\te}\Psi[\om\oj]\}_\te)
 \pr= \sum_{j=0,1}\dist_{\dot H^{1/2}}(u\oj(t_0),\cQ_{\om\oj})
 \lec |\lb_1\zr|+|\lb_1\on|.}
Since $v\oj\in\V_{\om\oj}$, we have 
\EQ{
 0 \pt=\LR{iv\zr|Q_{\om\zr}'}=\LR{ie^{i\diff\te\pa}\rS_{\om\zr/\om\on}(Q_{\om\on}+v\on)|Q_{\om\zr}'}
 \pr=-\sin\diff\te\pa\LR{Q_{\om\on}|Q_{\om\zr}'} +O(|\diff\log\om\pa|+\|v\on\|_2),}
hence $|\diff\te\pa|\lec|\diff\log\om\pa|+\|v\on\|_2\lec|\lb_1\zr|+|\lb_1\on|$. 
Using that $iQ_\om,Q_\om'\in\Ker P_1^\om$, we have  
\EQ{
 \lb_1\zr \pt=P_1^{\om\zr}[e^{i\diff\te\pa}\rS_{\om\zr/\om\on}(Q_{\om\on}+v\on)-Q_{\om\zr}]
 \pr=\lb_1\on + O((|\diff\te\pa|+|\diff\log\om\pa|) |\lb_1\on|+|\diff\te\pa|^2+|\diff\log\om\pa|^2). }
Then using $\sign\lb_1\zr\not=\sign\lb_1\on$, we obtain  
\EQ{
 |\lb_1\zr|+|\lb_1\on|=|\diff\lb_1\pa| \lec |\diff\te\pa|^2+|\diff\log\om\pa|^2+|\lb_1\on|^2
 \lec (|\lb_1\zr|+|\lb_1\on|)^2\lec \de_V^2 \ll 1,}
which is a contradiction, if $\de_V$ is small enough. 
Therefore $\Sg_\om$ is invariant for $\om$. 
\end{proof}

\section{One-pass lemma}
Now we are ready to prove the key dynamical property that any solution can {\it not} pass closely by the first excited states {\it more than once}. 
In the proof below in the region $\Sg_\om=+1$, we will use the results and the arguments in \cite{NLSP1}, which requires smallness of $\bM(u)$, or equivalently largeness of $\om$.  
To be precise about it, we have 
\begin{lem} \label{lem:smallM}
There is a constant $\om_\star\in[\om_*,\I)$ such that for any $\om\ge\om_\star$, every $\fy\in\Sg_\om^{-1}(\{+1\})$ satisfies all the small-mass conditions in \cite{NLSP1}. Specifically, using the constants $\mu_*$ in Lemma \ref{lem:sum} and $\mu_p,\mu_\star$ in \cite[Theorems 1.1 and 7.1]{NLSP1}, we have 
\EQ{
 \bM(\rS_\om^{-1}\fy)<\min(\mu_*,\mu_p,\mu_\star).}
\end{lem}
\begin{proof}
Immediate from Lemma \ref{lem:unifbd} and $\bM(\rS_\om^{-1}\fy)=\om^{-1/2}\bM(\fy)$. 
\end{proof}

\begin{lem} \label{lem:one pass}
There is a constant $\de_*\in(0,\min(\de_X,c_X^{-1/2}\e_S)]$ such that if $u_\om$ is a solution of \eqref{rscNLS} for some $\om\ge\om_\star$ on a maximal existence interval $(T_-,T_+)$, satisfying 
\EQ{ \label{op start}
 d_\om(u_\om(t_1))<\de, \pq \bA^\om(u_\om)<\bA^\om(Q_\om)+c_X\de^2,}
for some $\de\in(0,\de_*]$ at some $t_1\in(T_-,T_+)$, 
then there exists $t_2\in(t_1,T_+]$ such that 
$d_\om(u_\om(t))<\de$ for $t_1\le t<t_2$ and $d_\om(u_\om(t))>\de$ for $t_2<t<T_+$. If $t_2=T_+$, then the trapping lemma \ref{lem:stay} applies to $u_\om$. 
\end{lem}

The rest of this section is devoted to proving the above lemma. 
The solution $u_\om$ of \eqref{rscNLS} is fixed, so that we can abbreviate $d(t):=d_\om(u_\om)$, but all estimates will be uniform with no dependence on the particular choice of $u_\om$. 

The last sentence of the lemma is obvious from $\de\le\de_*\le\de_X$. 
For a proof of the rest and main part of the lemma, it suffices to derive a contradiction from the following: 
Suppose that for some $t_-<t_+$ within $(T_-,T_+)$, 
\EQ{ \label{return path}
 \pt \max_{t\in[t_-,t_+]}d(t)>\min_{t\in[t_-,t_+]}d(t)=d(t_\pm)=:\de\in(0,\de_*].}
Taking $\de_*\le c_X^{-1/2}\e_S$ 
ensures that $u_\om(t)$ stays in $\ck\cH_\om$ for $t\in[t_-,t_+]$, because 
\EQ{ \label{var const A}
 \bA^\om(u_\om)-\bA^\om(Q_\om)< c_X\de^2\le c_X\de_*^2\le\e_S^2.}
Hence $\si:=\Sg_\om(u_\om(t))\in\{\pm 1\}$ is independent of $t\in[t_-,t_+]$. 

Taking $\de_*\ll\de_V$, decompose the time interval $[t_-,t_+]$ as follows. 
Let $\M$ be the set of all minimal points of $d:[t_-,t_+]\to[\de,\I)$ with the minima less than $\de_V$. 
Then applying the ejection lemma \ref{lem:eject} from each $t_0\in\M$ forward and backward in time, we obtain a closed interval $I(t_0)\subset[t_-,t_+]$ such that $d(t)^2$ is strictly convex on $I(t_0)$ with the unique minimal point $t=t_0$ with $d(t)=\de_X$ on $\p I(t_0)\setminus\{t_\pm\}$, and  
\EQ{
 e^{\al_\om|t-t_0|}d(t_0) \sim d(t) \sim \si\lb_1(t) \sim \si \bK^\om_2(u_\om(t)) + \CK\|\z(t_0)\|_{H^1}}
on $I(t_0)$. 
The convexity on each $I(t_0)$ implies that those intervals are mutually disjoint. Putting 
\EQ{ \label{def IHV}
 I_H:=\Cu_{t_0\in\M}I(t_0), \pq I_V:=[t_-,t_+]\setminus I_H,}
we have $d\in[\de,\de_X]$ on $I_H$ and $d\in[\de_V,\I)$ on $I_V$. 
Lemma \ref{lem:minid} (3) implies that $(\bM+\om\bH^0)(u_\om)>\CM$ on $[t_-,t_+]$, since otherwise $\bH^0(u_\om(t_\pm))<\bH^0(Q)/4<\bH^0(Q_\om)/2$ contradicts that $d(t_\pm)\le\de_*\ll 1$. 
Then the variational lemma \ref{lem:var} implies
\EQ{ \label{var bd IV}
 t\in I_V \implies \si \bK^\om_2(u_\om)\ge\ka_V(\de_V)>0.}
Note that the case (b) of Lemma \ref{lem:var} is also precluded by the proximity to $Q_\om$, which implies $\bE^\om(u_\om)>\bE^\om(Q_\om)/2\ge\de_U/2$. 

\subsection{Blow-up region} \label{ss:bup} 
For $\si=-1$, we use a localized virial as in \cite[\S 4.1]{NLS}, \cite{OT}
\EQ{ \label{def virb}
 \pt \sV_m(t):=\LR{m\phi_mu_\om|i\p_r u_\om}, }
where $\phi(r)$ is a smooth non-decreasing function satisfying
\EQ{
 \phi(r)=\CAS{r &(r\le 1)\\ 3/2 &(r\ge 2),}}
and $\phi_m(r)=\phi(r/m)$ for some cut-off radius $m>1$ to be chosen shortly. 
Using the equation \eqref{rscNLS}, we have
\EQ{ \label{vir-b}
 \dot \sV_m \pt=2\bK^\om_2(u_\om)-\LR{|\p_ru_\om|^2|2f_{0,m}}+\LR{|u_\om/r|^2|f_{1,m}}
 \prQ +\LR{|u_\om|^4|f_{2,m}}+2\ml{f_{0,m}\cS_\I'V^\om}(u_\om),}
where $f_{j,m}(r):=f_j(r/m)$ with 
\EQ{
 f_0:=1-\phi_r, \pq f_1:=-r^2\De(\p_r/2+1/r)\phi, \pq f_2:=3/2-(\p_r/2+1/r)\phi.}
The last term in \eqref{vir-b} is the only essential difference from the case \cite{NLS} without the potential. 
Since $\cS_\I'V\in L^2+L^\I_0$, for any $\y>0$ there exists $B(\y)\in[1,\I)$ and a decomposition $\cS_\I'V=W_2+W_\I$ such that 
\EQ{
 \|W_\I\|_\I \le \y, \pq \|W_2\|_2 \le B(\y),}
cf.~ \cite[Lemma 2.1]{NLS}. Let $W_p^\om:=\om^{-1}W_p(\om^{-1/2}x)$. Then 
\EQ{
 |\ml{f_{0,m}\cS_\I'V^\om}(u_\om)| 
 \pt\le \|W_\I^\om\|_\I\|u_\om\|_{L^2(f_{0,m}dx)}^2 + \|W_2^\om\|_2\|f_{0,m}|u_\om|^2\|_2
 \pr\le \om^{-1}\y\|u_\om\|_{L^2(|x|>m)}^2 + \om^{-1/4}B(\y)\|u_\om\|_{L^4(f_{0,m}dx)}^2,}
where we used that $\supp f_{0,m}\subset\{|x|>m\}$ and $0\le f_{0,m}\le 1$. 
The last $L^4$ norm is treated in the same way as \cite[(4.14)]{NLS} by the radial Sobolev for $\fy\in H^1_r(\R^3)$ 
\EQ{
 \pt\|\fy\|_{L^4(f_{0,m}dx)}^4
 \sim \int_m^\I f_{0,m}'(s)\|\fy\|_{L^4(|x|>s)}^4 ds
 \pr\lec\int_m^\I f_{0,m}'(s)s^{-2}\|\fy\|_{L^2(|x|>s)}^3\|\fy_r\|_{L^2(|x|>s)}ds
 \pr\le m^{-2}\|\fy\|_{L^2(|x|>m)}^3\BR{\int_m^\I f_{0,m}'(s)\|\fy_r\|_{L^2(|x|>s)}^2ds \cdot \int_m^\I f_{0,m}'(s)ds }^{1/2}
 \pr\sim m^{-2}\|\fy\|_{L^2(|x|>m)}^3\|\fy_r\|_{L^2(f_{0,m}dx)}.}
The same estimate applies to the second last term of \eqref{vir-b}, because $|f_{2,m}|\lec f_{0,m}$. 
The norm $\|\p_ru_\om\|_{L^2(f_{0,m}dx)}$ can be absorbed by the second term on the right of \eqref{vir-b} after Young. 
Using also that $\|u_\om\|_{L^2(|x|>m)}\lec \|Q_\om\|_{L^2}+\de\lec 1$, we obtain 
\EQ{
 \dot\sV_m + 2\bK^\om_2(u_\om) 
  \pt\lec \om^{-1}\y + m^{-2}+(\om^{-1/4}m^{-1}B(\y))^{4/3}.}
Hence, choosing $\y$ small and then $m$ large such that 
\EQ{ \label{large cond}
 \y \ll \om_*\ka_V(\de_V), \pq m \gg \max(\ka_V(\de_V)^{-1/2},\om_*^{-1/4}B(\y)\ka_V(\de_V)^{-3/4}),}
we have $\dot \sV_m\le-\ka_V(\de_V)<0$ on $I_V$. 
On each $I(t_0)$ in $I_H$, we have 
\EQ{
 \|u_\om\|_{L^2(|x|>m)} \le \|Q_\om\|_{L^2(|x|>m)}+\|v\|_{L^2(|x|>m)}
 \lec m^{-1}+d(t),}
and so, using the hyperbolic estimate on $\bK^\om_2$ in Lemma \ref{lem:eject} as well, 
\EQ{
 [\si \sV_m]_{\p I(t_0)}=\int_{I(t_0)}\si\dot \sV_m dt
 \pt\gec \int_{I(t_0)} (e^{\al_\om|t-t_0|}-2\CK)d(t_0)-O(m^{-2}) dt
 \pr\gec \de_X-Cm^{-2}\log(\de_X/\de).}
On the other hand, $d(t)=\de$ at $t=t_\pm$ and $\|xQ_\om\|_2+\|x\na Q_\om\|_2\lec 1$ imply 
\EQ{
 |[\sV_m]_{t_-}^{t_+}| \lec \de + m\de^2.}
We can choose $m = 1/\de$ satisfying \eqref{large cond} if $\de_*$ is so small that  
\EQ{
 \de_* \le C^{-1}\ka_V(\de_V)^{1/2}\min(1,(\om_*\ka_V(\de_V))^{1/4}B(C\om_*\ka_V(\de_V))^{-1})}
for some large constant $C\in(1,\I)$. Then we have 
\EQ{
 |[\sV_m]_{t_-}^{t_+}| \lec \de
 \ll \de_X \lec \int_{t_-}^{t_+}\si\dot\sV_m dt,}
which is a contradiction. 
Therefore \eqref{return path} is impossible in the case $\si=-1$.

\subsection{Scattering region}
For $\si=+1$, we could argue as in \cite{NLS}, which would however suffer from the loss of sign in the localized virial due to the potential or the ground states, which happens as the solution is expected to be very dispersed in the variational time $I_V$. 
Specifically, the argument would fail at \cite[(4.31)]{NLS}. 
Then one option to overcome it would be to estimate possible dispersion and propagation along any returning orbit so that we can find an appropriate cut-off radius $m$. 
Instead of that, we rely on the minimal contradiction argument of Kenig-Merle \cite{KM} using the profile decomposition in \cite{NLSP1}, to show that there is a positive lower bound on $\de$ for which the return path \eqref{return path} can exist. 

Let $\om_n\ge\om_\star$ be a sequence such that $\om_n\to\OM\in[\om_\star,\I]$ and let $\ti u_n$ be a sequence of solutions to the rescaled equation \eqref{rscNLS} with $\om=\om_n$ and \eqref{return path} at some $t_{\pm,n}$ with 
\EQ{
 \de_X \gg \de=\de_n \to 0, \pq \Sg_{\om_n}(\ti u_n(t_{\pm,n}))=+1.}
After appropriate translation of each $\ti u_n$ in time, there exist sequences $\ti R_n<0<\ti S_n<\ti T_n$ such that, abbreviating $d_n(t):=d_{\om_n}(\ti u_n(t))$ and $\al_n:=\al_{\om_n}$, 
\EQ{ \label{beh dn}
 \pt d_n(\ti R_n)=\de_n=d_n(\ti T_n), \pq d_n(0)=\de_X=d_n(\ti S_n), 
 \pr \ti R_n\le t\le 0 \implies d_n(t)\sim e^{\al_n t}\de_X, 
 \pr \ti R_n<t<\ti T_n \implies d_n(t)>\de_n,
 \pr \ti S_n\le t\le \ti T_n \implies d_n(t)\sim e^{-\al_n(t-\ti S_n)}\de_X.}
Since $\ti u_n$ stays in $\ck\cH_{\om_n}$ with $\Sg_{\om_n}=+1$, by Lemma \ref{lem:unifbd} it is uniformly bounded in $H^1$ on $[\ti R_n,\ti T_n]$. 
Let $u_n:=\rS_{\om_n}^{-1}\ti u_n(\om_n t)$ be the sequence of unscaled solutions, and 
\EQ{
 (R_n,S_n,T_n):=\om_n^{-1}(\ti R_n,\ti S_n,\ti T_n).}
Since $\om_n\ge\om_\star$, Lemma \ref{lem:smallM} allows us to apply the arguments in \cite{NLSP1} to $u_n$. 
Using the coordinate around the ground solitons as in \cite[(4.9)]{NLSP1}, we can decompose 
\EQ{
 u_n(t)=\Phi[z_n(t)]+\y_n(t)=\Phi[z_n(t)]+R[z_n(t)]\x_n(t)}
such that $\y_n(t)\in\cH_c[z_n(t)]$ and $\x_n(t)\in P_c(H^1_r)$ for $t\in[R_n,T_n]$, where 
\EQ{ \label{def HcRz}
 \cH_c[z]:=\{\fy\in H^1_r\mid 0=\LR{i\fy|\p_z\Phi[z]}=\LR{i\fy|\p_{\ba z}\Phi[z]}\}, \pq R[z]=(P_c|_{\cH_c[z]})^{-1}.} 
Let $C_6>0$ be the best constant such that $\inf_\te\|e^{i\te}Q_\om-\fy\|_6\le C_6d_\om(\fy)$ holds for all $\om\ge\om_*$ and $\fy\in H^1_r$, and let 
\EQ{
 \de_W:=\inf\{\|e^{i\te}Q_\om-\rS_\om\Phi[z]\|_6/C_6 \mid \om\ge\om_*,\ \te\in\R,\ z\in Z_*\}.}
Then $\de_W>0$ because both $\{e^{i\te}Q_\om\}_{\om\ge\om_*,\te\in\R}$ and $\{\rS_\om\Phi[z]\}_{z\in Z_*}$ are precompact in $H^1$ and the normand is never $0$ even on their closures. 
Hence for large $n$, there exists $\ti S_n'\in[\ti S_n,\ti T_n]$ such that 
$d_n(\ti S_n')=\min(\de_X,\de_W/2)$. 
Then using the scale invariance of $\ST=L^4_tL^6$, we obtain 
\EQ{ \label{stxi bup}
 \|\x_n\|_{\st(S_n,T_n)}\pt\sim\|\y_n\|_{\st(S_n,T_n)}\ge\|\ti u_n-\rS_{\om_n}\Phi[z_n(t/\om_n)]\|_{\ST(\ti S_n',\ti T_n)} 
 \pr\ge C_6|\ti T_n-\ti S_n'|^{1/4}(\de_W-d_n(\ti S_n'))
 \pr\sim C_6\de_W \log^{1/4}(\de_W/\de_n)\to \I\pq (n\to\I),}
because of the exponential behavior on $[S_n,T_n]$ in \eqref{beh dn}. 
Similarly we have $\ti R_n\lec-\log(\de_X/\de_n)\to-\I$ and $\ti T_n\gec\log(\de_X/\de_n)\to\I$. 
Since $\ti u_n$ are uniformly bounded in $C([\ti R_n,\ti T_n];H^1_r)$, a standard weak compactness argument implies that, passing to a subsequence, $\ti u_n$ converges to some $\ti u_\I$ in $C(\R;\weak{H^1_r})\cap L^\I(\R;H^1)$, which solves the limit equation, that is \eqref{rscNLS} with $\om=\OM<\I$ or \eqref{NLS} if $\OM=\I$. 

The weak convergence implies 
$\bE^\OM(\ti u_\I) \le \bE^\OM(Q_\OM)$ and $\bM(\ti u_\I) \le \bM(Q)$, 
as well as $d_\OM(\ti u_\I(t)) \lec e^{\al_\OM t}\de_X$ for all $t<0$, hence by the conservation law
\EQ{ \label{lim ME}
 (\bM,\bE^\OM)(\ti u_\I)=(\bM,\bE^\OM)(Q_\OM).} 
Therefore the convergence $\ti u_n\to \ti u_\I$ is strong in $H^1$, locally uniformly in $t$. 

\subsubsection{The case of bounded time frequency $\OM<\I$} \label{onepass finiteom} 
In this case, the above strong convergence is translated to that of $u_n$ to 
$u_\I(t):=\rS_\OM^{-1}\ti u_\I(\OM t)$. 
Apply the profile decomposition in \cite[\S 5--7]{NLSP1} to $\x_n$ on $[0,T_n]$. 
Then the strong convergence of $u_n(0)$ implies that there is only one nonlinear profile, which is the strong limit at $t=0$, and the remainder is strongly vanishing in $H^1$. 
Hence \cite[Theorem 7.2]{NLSP1} and \eqref{stxi bup} imply that $u_\I$ does not scatter to $\Phi$ as $t\to\I$.  

Then the main result of \cite{NLSP1} below $\Soli_1$ together with \eqref{lim ME} implies that $u_\I$ is a minimal non-scattering solution, so the argument in \cite[\S 8]{NLSP1} implies that $u_\I(0,\I)\subset H^1$ is precompact. Since 
\EQ{
 d_\I(t):=d_{\OM}(\ti u_\I(t))}
is exponentially decaying as $t\to-\I$, the trapping lemma \ref{lem:stay} implies that $\ti u_\I(t)=e^{i(a-t)}Q_\OM+o(1)$ in $H^1$ as $t\to-\I$ for some $a\in\R$. 
In particular, the entire trajectory $\ti u_\I(\R)$ is precompact in $H^1$. 

Now we use another localized virial identity as in \cite[\S 4.2]{NLS}
\EQ{ \label{vir-s}
 \pt \sV_m:=\LR{\psi_m \ti u_\I|i\cS_2'\ti u_\I}, 
 \pr \dot\sV_m=2\bK^\OM_2(\psi_m\ti u_\I)+\LR{|\ti u_\I/m|^2|f_{3,m}}-\LR{|\ti u_\I|^4|f_{4,m}}
 \prQ -2m^{-1}\ml{r\cS_\I'V^\OM}(\psi_m\ti u_\I),}
where $\psi_m(r)=\psi(r/m)$ and $f_{j,m}(r)=f_j(r/m)$ as before, with 
\EQ{
 \psi:=(1+r)^{-1}, \pq f_3:=(1+r)^{-4}, \pq f_4:=(1+r)^{-4}r(r^2+7r/2+4).}
The last term in \eqref{vir-s} is bounded by 
\EQ{
 |m^{-1}\ml{r\cS_\I'V^\OM}(\psi_m\ti u_\I)|
 \lec \ml{\min(r/m,m/r)|\cS_\I'V^\OM|}(\ti u_\I).}
Combining the precompactness with the above estimate, as well as the decay of $f_{3,m}$ and $f_{4,m}$, yields some $m \ge 1$ such that for all $t\in\R$ 
\EQ{ \label{precomp est}
 |\dot\sV_m-2\bK^\OM_2(\ti u_\I)| \ll \ka_V(\de_V).}
Let $0<\de\ll\min(1/m,\de_V)$, $d_\I(t_-)=\de$ with $t_-<0$ and 
\EQ{ 
 t_+:=\inf\{t>t_-\mid d_\I(t)\le\de\},}
then $t_+>t_-$ because $\p_td_\I(t_-)>0$ by the ejection lemma \ref{lem:eject}. 
Suppose that $t_+<\I$ for contradiction. 
Decompose $[t_-,t_+]=I_H\cup I_V$ as in \eqref{def IHV}.  Using the exponential decay of $Q_\OM$, we have on each $I(t_0)$ 
\EQ{
 \dot \sV_m = 2\bK^\OM_2(\ti u_\I) + O(m^{-1}+d_\I^2)
 \gec (e^{\al_\OM|t-t_0|}-2\CK)d_\I(t_0)-O(m^{-1}),}
and so
\EQ{
 [\sV_m]_{\p I(t_0)}\gec \de_X - Cm^{-1}\log(\de_X/\de)
 \ge \de_X(1-\de/\de_X\log(\de_X/\de))>\de_X/2.}
On the other hand, \eqref{var bd IV} and \eqref{precomp est} imply that $\dot\sV_m\ge \ka_V(\de_V)>0$ on $I_V$. 
Hence 
\EQ{
 \de_X \lec [\sV_m]_{t_-}^{t_+} \lec \de + m\de^2 \lec \de \ll \de_X,}
leading to a contradiction. 
Therefore $t_+=\I$, which implies however that $\sV_m\to\I$ as $t\to\I$ by the above argument on $[t_-,\I)$, contradicting the precompactness of $\ti u_\I(\R)$. 
Thus we have precluded the case $\OM<\I$. 

\subsubsection{The concentrating case $\OM=\I$}
In this case, we have $Q_\OM=Q$, and the limit $\ti u_\I$ is a global solution of \eqref{NLS} exponentially convergent to $\{e^{i\te}Q\}_\te$ as $t\to-\I$. 
Then the classification by \cite{DHR} implies that $\ti u_\I$ is, modulo time translation, either the soliton $e^{-it}Q$ itself or the unique solution $w_+$ which is exponentially converging to $e^{-it}Q$ as $t\to-\I$ and scattering to $0$ as $t\to+\I$. 
The strong convergence at $t=0$ implies $d_\I(\ti u_\I(0))=\de_X>0$, precluding the soliton case. Hence $\ti u_\I=w_+$. 
If $\ti S_n$ converges to some finite $\ti S_\I<\I$ along a subsequence, then $d_\I(\ti u_\I(t)) \lec e^{-\al(t-\ti S_\I)}\de_X$ for $t\ge\ti S_\I$, contradicting the scattering to $0$ of $\ti u_\I$ as $t\to\I$. 
Hence $\ti S_n\to\I$. 

The scattering to $0$ implies that for any $\nu>0$ there exists $\ti\ta>0$ such that 
\EQ{
 \|e^{-i(t-\ti\ta)\De}\ti u_\I(\ti\ta)\|_{\st(\ti\ta,\I)}\le\nu.}
Then, putting $\ta_n:=\om_n^{-1}\ti\ta$ and using $\ti u_n(\ti\ta)\to \ti u_\I(\ti\ta)$ in $H^1$, we obtain 
\EQ{
 \|e^{-i(t-\ta_n)\De}u_n(\ta_n)\|_{\st(\ta_n,\I)}
 \pt=\|e^{-i(t-\ti\ta)\De}\ti u_n(\ti\ta)\|_{\st(\ti\ta,\I)}
 \pr=\|e^{-i(t-\ti\ta)\De}\ti u_\I(\ti\ta)\|_{\st(\ti\ta,\I)}+o(1)<2\nu}
for large $n$. 
We also have uniform bounds  
\EQ{ \label{Hs bd}
 \|u_n(\ta_n)\|_{H^\te} \lec \om_n^{(\te-1/2)/2}\|\ti u_n(\ti\ta)\|_{H^\te} \lec \om_n^{(\te-1/2)/2}} 
for $0\le\te\le 1$, hence in particular, 
\EQ{
 |z_n(\ta_n)| \lec \|u_n(\ta_n)\|_2 \lec \om_n^{-1/4}.}
Let $\x_n\zr:=e^{-i(t-\ta_n)\De}\x_n(\ta_n)$. 
Using 
\EQ{
 \x_n(\ta_n)=P_c(u_n(\ta_n)-\Phi[z_n(\ta_n)]),}
and 
\EQ{
 |(\phi_0|u_n(\ta_n)-\Phi[z_n(\ta_n)])|\lec \|u_n(\ta_n)\|_2+|z_n(\ta_n)| \lec \om_n^{-1/4},}
we have by the free Strichartz estimate, 
\EQ{ \label{small xn0}
 \pt\|\x_n\zr\|_{\st(\ta_n,\I)}
  \le \|e^{-i(t-\ta_n)\De}u_n(\ta_n)\|_{\st(\ta_n,\I)} + C\om_n^{-1/4} <3\nu,
 \pr\|\x_n\zr\|_{L^4_tL^3(\ta_n,\I)} \lec \|\x_n(\ta_n)\|_2 \lec  \om_n^{-1/4}  \ll \nu,}
for large $n$. 
Let $\x_n\on$ be the linearized solution with the same initial data, namely
\EQ{
 (i\p_t+H-B[z_n])\x_n\on=0,\pq \x_n\on(\ta_n)=\x_n(\ta_n)\in P_c(H^1_r),}
where $B[z]$ is the $\R$-linear operator defined by 
\EQ{ \label{def Bz}
 B[z]\fy=P_c\{2|\Phi[z]|^2R[z]\fy+\Phi[z]^2\ba{R[z]\fy}\}.}
Then we have 
\EQ{ \label{free2lin}
 (i\p_t+H-B[z_n])P_c\diff\x_n\pa = -P_cV\x_n\zr+B[z_n]P_c\x_n\zr, \pq \diff\x_n\pa(\ta_n)=0.}
Applying the non-admissible Strichartz \cite[(4.41)]{NLS} to \eqref{free2lin}, we obtain
\EQ{ \label{unif scat}
 \|P_c\diff\x_n\pa\|_{\st(\ta_n,\I)}
 \pt\lec \|V\x_n\zr-B[z_n]P_c\x_n\zr\|_{L^4_tL^{6/5}(\ta_n,\I)}
 \pr\lec \BR{\|V\|_{L^2}+|z_n|}\|\x_n\zr\|_{L^4_tL^3(\ta_n,\I)}\ll \nu. }
Adding it to \eqref{small xn0}, and using $\x_n\on=P_c\x_n\on$, we obtain 
\EQ{
 \|\x_n\on\|_{\st(\ta_n,\I)}<4\nu}
for large $n$. 
Hence, taking $\nu>0$ small, and using \eqref{Hs bd} as well, we deduce from the small data scattering \cite[Lemma 6.2]{NLSP1} that $u_n$ scatters to $\Phi$ as $t\to\I$ with a uniform bound for large $n$
\EQ{
 \|\x_n\|_{\st(\ta_n,\I)} \lec  \nu.}
Since $\ti S_n\to\I$ implies $\ta_n<S_n$ for large $n$, the above bound contradicts \eqref{stxi bup}. 
Therefore $\OM=\I$ is also impossible, which means that there can not exist such a sequence of solutions $\ti u_n$ in \eqref{beh dn}. 
Thus we finish the proof of Lemma \ref{lem:one pass}. \qedsymbol
\begin{rem}
The additional assumption $V\in L^2(\R^3)$ is needed only in the above estimate \eqref{unif scat} and similarly in \eqref{nonlin est 0mas scat}. It could be replaced with the following statement: For any bounded sequence $\fy_n$ in $\dot H^{1/2}(\R^3)$, we have 
\EQ{
 \|e^{-it\De}\fy_n\|_{\st(0,\I)}+\|\fy_n\|_2 \to 0 \implies \|e^{-itH}P_c\fy_n\|_{\st(0,\I)}\to 0,}
$V\in L^2(\R^3)$ is a sufficient condition, as shown above by the Strichartz perturbation, but the latter does not work if we merely assume $V\in L^2+L^\I_0$. 
\end{rem}

\section{Dynamics away from the excited states}
The one-pass lemma \ref{lem:one pass} ensures that if a solution of the rescaled NLS \eqref{rscNLS} leaves the small neighborhood of $\cQ_\om$, then it never returns. 
In this section, we investigate behavior of such solutions $u_\om$ staying away from $\cQ_\om$, after some time or for all time. 
More precisely, let $u_\om$ be any solution of \eqref{rscNLS} for some $\om\ge\om_\star$ satisfying 
\EQ{
 \bA^\om(u_\om)\le\bA^\om(Q_\om)+c_X\de_*^2, 
 \pq \inf_{0\le t<T_+}d_\om(u_\om(t))\ge\de_*,}
where $T_+\in(0,\I]$ is the maximal existence time of $u_\om$. 
Thanks to \eqref{var const A}, we have the same decomposition of $[0,T_+)=I_H\cup I_V$ as in \eqref{def IHV} with $\de=\de_*$, and the sign $\si:=\Sg_\om(u_\om(t))\in \{\pm 1\}$ remains constant for $t\in[0,T_+)$, which distinguishes the scattering and the blow-up cases. 

\subsection{Blow-up region}
In the case $\si=-1$, the solution blows up. 
\begin{lem} \label{lem:bupafter}
For every $\om\ge\om_*$ and every solution $u_\om$ of \eqref{rscNLS} satisfying 
\EQ{ \label{bupafter cond}
 \pt \bA^\om(u_\om)\le \bA^\om(Q_\om)+c_X\de_*^2, 
 \pq \inf_{0\le t<T_+}d_\om(u_\om(t))\ge\de_*, \pq \Sg_\om(u_\om(0))=-1,}
where $T_+$ is the maximal existence time, blows up in finite time, namely $T_+<\I$. 
\end{lem}
Using the localized virial estimate in Section \ref{ss:bup}, the proof is essentially the same as \cite[\S 4.1]{NLS} in the case without the potential, because the region $\Sg_\om=-1$ is away from the ground states, where $\bK^\om_2$ degenerates. The detail is omitted. 

Also note that $\om\ge\om_*$ is enough in this region, since the profile decomposition in \cite{NLSP1} is not needed. In fact, the region of $u_\om(0)$ in the above lemma is unbounded in $L^2(\R^3)$, though it is not essentially new compared with \cite{NLSP1}, since those initial data with large $L^2$ need very negative energy $\bE$ to satisfy \eqref{bupafter cond}, for which proving the blow-up is easier. 

\subsection{Scattering region}
In the case $\si=+1$, $u_\om$ scatters to $\Phi$ as $t\to\I$. 
The global existence is immediate from the $H^1$ bound of Lemma \ref{lem:unifbd}, so the main part is to prove the scattering. 
As in \cite{NLS}, it is done by using the profile decomposition. 
In the same way as in the previous section, we need to distinguish between the case of bounded $\om$ and the case of $\om\to\I$. 

For each $\om\ge\om_\star$ and $A\le c_X\de_*^2$, let $\FS_\om(A)$ be the set of all the solutions $u$ of \eqref{NLSP} global in $t>0$ satisfying 
\EQ{ \label{scat after region}
 \bA^\om(u_\om)\le \bA^\om(Q_\om)+A, \pq \inf_{t\ge 0}d_\om(u_\om(t))\ge \de_*,}
where $u_\om(t):=\rS_\om u(t/\om)$ is the rescaled solution of \eqref{rscNLS}. 
In the original scale, the first condition is equivalent to, putting $\mu:=\bM(\Psi[\om])$,  
\EQ{ \label{Aom bd FS}
 \bA_\om(u) \le \sE_1(\mu)+\om\mu+\om^{1/2}A = \om^{1/2}(\bA(Q)+A+O(\om^{-1/4})).}

Now we look for a minimal solution in $\FS_\om(A)\setminus\cS$. 
Note that \eqref{scat after region} with $A\le c_X\de_*^2$ implies that $u_\om=\rS_\om u(t/\om)$ stays in $\ck\cH_\om$ for all $t\ge 0$, and the case $\Sg=-1$ is precluded by the blow-up Lemma \ref{lem:bupafter}. 
Hence Lemma \ref{lem:smallM} with $\om\ge\om_\star$ allows us to apply the arguments in \cite{NLSP1} to any solution $u$ in $\FS_\om(A)$. 
Using the decomposition $u(t)=\Phi[z(t)]+R[z(t)]\x(t)$ of \cite[Lemma 4.1]{NLSP1}, let 
\EQ{ \label{def A*}
 \pt \STN^\om(A):=\sup_{u\in\FS_\om(A)} \|\x\|_{\ST(0,\I)}, 
 \pr A^*_\om :=\sup\{A<c_X\de_*^2 \mid \STN^\om(A)<\I\},
 \pq A^*:=\inf_{\om\ge\om_\star}A^*_\om.}
Since the region $\{(\mu,e)\in\R^2\mid e+\om\mu < \bA_\om(\Psi[\om])\}$ 
is tangent from below to the graph $e=\sE_1(\mu)$ at $\mu=\bM(\Psi[\om])$ because of $\sE_1''>0$, 
the region $\bA^\om(u_\om)<\bA^\om(Q_\om)$ for $\om\ge\om_\star$ is covered by the scattering below $\Soli_1$ in \cite{NLSP1}. 
Hence  
\EQ{ \label{scat below}
 A_\om^* \ge 0}
for each $\om\ge\om_\star$. 
Now suppose for contradiction that 
\EQ{
 A^* \ll c_X\de_*^2.} 
Then there exist sequences $\om_n\ge\om_\star$, $A_n>0$, and $u_n\in\FS_{\om_n}(A_n)$ such that 
\EQ{
 \pt \om_n\to\OM\in[\om_\star,\I],\pq c_X\de_*^2>A_n \to A^*, 
 \pr  u_n(t) = \Phi[z_n(t)] + R[z_n(t)]\x_n(t), \pq 
 \|\x_n\|_{\ST(0,\I)}\to\I.}
Let 
\EQ{
 \ti u_n(t):=\rS_{\om_n}u_n(t/\om_n)}
Since $\Sg_{\om_n}(\ti u_n)=+1$, Lemma \ref{lem:unifbd} implies that $\ti u_n(t)$ is uniformly bounded in $H^1$. 
By Lemmas \ref{lem:eject} and \ref{lem:one pass}, we may additionally impose, after translation in time, 
\EQ{
 d_{\om_n}(\ti u_n(0))\ge\de_X,}
since if it cannot be achieved by translation, then the trapping lemma \ref{lem:stay} applies to $\ti u_n$, contradicting $u_n\in\FS_{\om_n}(A_n)$ with $A_n< c_X\de_*^2$.

\subsubsection{The case $\OM<\I$} 
Apply the nonlinear profile decomposition of \cite{NLSP1} to the sequence of solutions $u_n$ on the time interval $[0,\I)$. Here the procedure is outlined for the sake of notation. Let (after extracting a subsequence)
\EQ{ \label{linear prof decop}
 \x_n^L = \sum_{0\le j<J}\la_n^j + \ga_n^J}
be the linearized profile decomposition of \cite[Lemma 5.3]{NLSP1}, where $\x_n^L$, $\la_n^j$ and $\ga_n^J$ solve the same linearized equation 
\EQ{
 (i\p_t+H-B[z_n])\x=0, }
with the initial conditions $\x_n^L(0)=\x_n(0)$ and $\la_n^j(s_n^j)=\weak\Lim_{m\to\I}\x_m^L(s_m^j)$ for some time sequences $s_n^j\in[0,\I)$ satisfying $s_n^0=0$ and $s_n^j-s_n^k\to\pm\I$ as $n\to\I$ for $0\le j<k<J$. 
After fixing $J$ large enough, let $\La_n^j$ be the nonlinear profiles, defined by the weak limit 
\EQ{ \label{def La}
 \x_\I^j(t)=\weak\Lim_{n\to\I}\x_n(t+s_n^j), \pq \La_n^j:=\x_\I^j(t-s_n^j),}
after passing to a further subsequence if necessary. 
Then by \cite[Theorem 7.2]{NLSP1}, there exists at least one nonlinear profile which does not scatter as $t\to\I$, since otherwise $\|\x_n\|_{\st(0,\I)}$ would be bounded. 
Let $\x_\I^l$ be a nonlinear profile with $\|\x_\I^l\|_{\st(0,\I)}=\I$ which is minimal in the sense that if $s^j_n-s^l_n\to-\I$ then the nonlinear profile $\x^j_\I$ scatters as $t\to\I$. 
Let $\mu_\I:=\bM(\Psi[\OM])$, 
\EQ{
 u_\I^l(t):=\weak\lim_{n\to\I}u_n(t+s^l_n),\pq 
 \ti u_\I^l(t):=\rS_\OM u^l_\I(t/\OM).} 
Then $\ti u_\I^l$ is a solution of \eqref{rscNLS} with $\om=\OM$, and the weak convergence $u_n(s^l_n)\weakto u^l_\I(0)$ in $H^1_r$ implies
\EQ{ \label{ene prof}
 \bA^\OM(\ti u^l_\I)\le \liminf_{n\to\I}\bA^{\om_n}(\ti u_n)=\bA^\OM(Q_\OM)+A^*.}

Suppose that $d_\OM(\ti u^l_\I(\ta))\ll\de_*$ at some $\ta\in\R$ ($\ta\ge 0$ if $l=0$), and put 
\EQ{
 t_n:=\om_n(s^l_n+\OM^{-1}\ta),\pq \fy_n:=\ti u_n(t_n)-\ti u^l_\I(\ta),} 
then $\fy_n\to 0$ weakly in $H^1_r$. 
Since $d_{\om_n}(\ti u_n(t_n))\gec \de_*\gg d_{\OM}(\ti u^l_\I(\ta))$, we deduce $\|\fy_n\|_{H^1}\gec\de_*$, then \eqref{ene prof} with the weak convergence implies 
\EQ{
 \bA^\OM(Q_\OM)+A^* \ge \bA^\OM(\ti u^l_\I)+\lim_{n\to\I}\|\fy_n\|_{H^1}^2/2 \ge \bA^\OM(Q_\OM)+c\de_*^2,}
for some constant $c>0$, contradicting $A^*\ll\de_*^2$. 

Therefore $\de_*\lec d_\OM(\ti u^l_\I(t))$ at all $t$, and by Lemma \ref{lem:stay}, some translate of $\ti u^l_\I$ in $t$ belongs to $\FS_{\OM}(A^*)$. 
Since $\|\x^l_\I\|_{\st(0,\I)}=\I$, the definition of $A^*$ implies that \eqref{ene prof} must be equality, hence $u_n(s^l_n)\to u^l_\I(0)$ strongly in $H^1$, thus we have obtained a minimal element $u(t):=u_\I^l(t+T)$ for some $T\ge 0$, satisfying 
\EQ{
 u=\Phi[z]+R[z]\x\in\FS_\OM(A^*), \pq \|\x\|_{\st(0,\I)}=\I.}

Next we prove that for such a critical element $u\in\FS_\om(A^*)$ with $\|\x\|_{\ST(0,\I)}=\I$, the orbit $\{u(t)\mid t\ge 0\}$ is precompact in $H^1_r$. 
For any sequence $t_n\to\I$, the same argument as above applies to the sequence of solutions $u_n:=u(t+t_n)$ on $[-t_n,0]$ and on $[0,\I)$, because $\|\x\|_{\ST(0,t_n)}\to\I$ and $\|\x\|_{\ST(t_n,\I)}=\I$ as $n\to\I$.  
Consider the profile decomposition with the first nonlinear profile 
\EQ{
 \x_\I^0(t)=\weak\Lim_{n\to\I}\x_n(t)=\weak\Lim_{n\to\I}\x(t+t_n),} 
after extracting a subsequence. 
Let $u_\I^0$ be the weak limit of $u_n(t)=u(t+t_n)$. 

If $\x_\I^0$ is not scattering as $t\to\I$, then it is a non-scattering profile on $[0,\I)$, and the minimality as above implies strong convergence of $u_n(0)=u(t_n)$ in $H^1$. 
If $\x_\I^0$ is not scattering as $t\to-\I$, then the same argument for $t<0$ implies the same strong convergence. 

Suppose that $\x_\I^0$ is scattering both as $t\to\pm\I$. The scattering implies that $\bA_\OM(u_\I^0)\ge 0$, because $\bA_{\Om[z]}(\Phi[z])>0$ for $z\in Z_*$, which follows from 
\EQ{
 \p_\om\bA_\om(\Phi_\om)=\LR{\bA_\om'(\Phi_\om)|\Phi_\om'}+\bM(\Phi_\om)=\bM(\Phi_\om)>0,}
where $\Phi_\om:=\Phi[\Om|_{[0,z_*)}^{-1}(\om)]$ is the unique positive radial solution of \eqref{rsNLS}. 

Moreover, there are non-scattering profiles in $t>0$ and in $t<0$. 
More precisely, there is a profile $\x_\I^{l_0}$ with $s_n^{l_0}\to\I$ and $\|\x_\I^{l_0}\|_{\ST(0,\I)}=\I$, and another profile $\x_\I^{l_1}$ with $-t_n<s_n^{l_1}\to-\I$ and $\|\x_\I^{l_1}\|_{\ST(-\I,0)}=\I$, both satisfying $\x_\I^{l_j}(-s_n^{l_j})=\la_n^{l_j}(0)+o(1)$ in $H^1$. 
The last two properties, together with the scattering below the excited states (cf.~the argument for \eqref{scat below}), imply that, as $n\to\I$,
\EQ{
 (\OM^{-1/2}\bH^0+\OM^{1/2}\bM)(\la^{l_j}_n(0))=\OM^{-1/2}\bA_\OM(u^{l_j}_\I)+o(1) \ge \bA^\OM(Q_\OM)+o(1).}
Then the asymptotic orthogonality \cite[(7.12)]{NLSP1} of the mass-energy implies 
\EQ{
 \bA^\OM(Q_\OM)+A^* \pt=\bA^\OM(u_\OM) = (\OM^{-1/2}\bE+\OM^{1/2}\bM)(u) 
 \pr\ge \OM^{-1/2}\bA_\OM(u^0_\I) + \sum_{j=0,1}(\OM^{-1/2}\bH^0+\OM^{1/2}\bM)(\la_n^{l_j}(0))+o(1)
 \pr\ge 2\bA^\OM(Q_\OM) + o(1),}
which is a contradiction because $\bA^\OM(Q_\OM)\sim 1 \gg A^*$. 

Therefore $u(t_n)=u_n(0)$ should be strongly convergent (along a subsequence), which means that $\{u(t)\}_{t\ge 0}$ is precompact. 
Then the same virial argument as in Section \ref{onepass finiteom} leads to a contradiction, so the case $\OM<\I$ is precluded. 
 Note that in using the variational lemma \ref{lem:var}, we can eliminate the cases $(\bM+\om\bH^0)(u_\om)\le\CM$ and (b) using the scattering by Lemma \ref{lem:minid} (3) and by \cite{NLSP1} respectively (instead of using the proximity to $Q_\om$ as in Section \ref{onepass finiteom}). 

\subsubsection{The case $\OM=\I$} 
In this case, we apply the profile decomposition of the NLS without potential to the rescaled radiation. Decompose $u_n=\Phi[z_n]+\y_n$ and rescale by $\rS_n:=\rS_{\om_n}$, naming 
\EQ{
 \pt \ti z_n(t):=z_n(t/\om_n),\pq V_n:=V^{\om_n}, 
 \pq \ti\y_n:=\rS_n\y_n(t/\om_n),
 \pq \ti\Phi_n:=\rS_n\Phi[\ti z_n].}
The soliton component is uniformly tending to $0$ as 
\EQ{
 |\ti z_n(t)| \lec \|u_n(t/\om_n)\|_2 = \om_n^{-1/4}\|\ti u_n(t)\|_2 \lec \om_n^{-1/4}.}
The equation for $\ti\y_n$ can be written as 
\EQ{
 eq_n(\ti\y_n) \pt= 2|\ti\Phi_n|^2\ti\y_n+\ti\Phi_n^2\ba{\ti\y_n}+2\ti\Phi_n|\ti\y_n|^2+\ba{\ti\Phi_n}\ti\y_n^2
 \pr-\sum_{j=1,2}i\om_n^{-1}\rS_n\p_j\Phi[\ti z_n]\U{N}_j(\ti z_n,\rS_n^{-1}\ti\y_n),}
where
\EQ{
 \pt eq_n(u):=(i\p_t-\De+V_n)u-|u|^2u, 
 \pr N(z,\y):=2\Phi[z]|\y|^2+\ba{\Phi[z]}\y^2+|\y|^2\y, 
 \pr \U{N}_j(z,\y)=\sum_{k=1,2}M^{-1}_{j,k}(z,\y)\LR{N(z,\y)|\p_k\Phi[z]},
 \pr M_{j,k}(z,\y)=\LR{i\p_j\Phi[z]|\p_k\Phi[z]}-\LR{i\y|\p_j\p_k\Phi[z]},}
and $\p_j\Phi[z]$ denotes the partial derivative in $(z_1,z_2)$ of $z=z_1+iz_2$. 
Apply the free profile decomposition to $\ti\y_n(0)$ on the time interval $[0,\I)$ (see, e.g., \cite[Lemma 5.2]{HR} or \cite[Proposition A.2]{NLS}):
\EQ{ \label{decop tiyn}
 e^{-it\De}\ti\y_n(0)=\sum_{0\le j<J}e^{-i(t-s_n^j)\De}\la^j+\ga_n^J(t),}
where the sequences of times $s_n^j\in[0,\I)$ satisfy 
\EQ{
 s_n^0\equiv 0,\pq s_n^j-s_n^k\to\pm\I\ (0\le j<k<J)} 
as $n\to\I$, and the linear profiles $\la^j\in H^1_r$ are defined by the weak limit
\EQ{ \label{def laj}
 \la^j=\weak\Lim_{n\to\I}e^{-is_n^j\De}\ti\y_n(0),} 
while the linear remainder $\ga^J_n\in C(\R;H^1_r)$ defined by \eqref{decop tiyn} satisfies
\EQ{ \label{small ga}
 \lim_{J\to J^*}\limsup_{n\to\I}\|\ga^J_n\|_{[L^\I_tL^4,\Stz^1]_\te(0,\I)}=0}
for some $J^*\in\N\cup\{\I\}$ and for all $\te\in[0,1)$, besides the weak vanishing 
\EQ{ \label{weak van}
 \weak\Lim_{n\to\I}\ga^J_n(s_n^j)=0,} 
which follows from \eqref{def laj}. 

Let $\La^j$ be the nonlinear profile associated with $\la^j$, and let $\Ga_n^J$ be the nonlinear remainder associated with $\ga_n^J$. More precisely, both $\La^j$ and $\Ga_n^J$ are solutions of \eqref{NLS} such that 
\EQ{ \label{def LaGa}
 \La^0(0)=\la^0,
 \pq \lim_{t\to-\I}\|\La^j(t)-e^{-it\De}\la^j\|_{H^1}=0\ (j\ge 1),
 \pq \Ga_n^J(0)=\ga_n^J(0).}
The small data scattering for \eqref{NLS} with \eqref{small ga} implies that if $J$ is close enough to $J^*$, then $\Ga_n^J$ for large $n$ is scattering as $t\to\I$ and small in $[L^\I_tL^4,\Stz^1]_\te(0,\I)$. 
In particular, it is small in the weaker norm 
\EQ{
 X:=L^{16}_tL^{24/7} \cap L^4_tL^6 \supset \Stz^{1/2},}
which is scaling invariant. Fix such $J<J^*$ for the rest of proof. 

Suppose that all the nonlinear profiles $\La^j$ ($0\le j<J$) are scattering as $t\to\I$, or equivalently $\La^j\in\Stz^1(0,\I)$. Then 
\EQ{
 \ck\y_n^J:=\sum_{0\le j<J}\La_n^j + \Ga_n^J, \pq \La_n^j:=\La^j(t+s_n^j),}
is an approximate sequence for $\ti\y_n$ in $X(0,\I)$. 
This claim is based on the long-time perturbation argument together with error estimates on $eq_n(\ti\y_n)$ and $eq_n(\ck\y_n^J)$. 
Note that $V_n$ in $eq_n$ is negligible when acting on the given approximate $\ck\y_n^J$, but not on $\ti\y_n$ enough to get a closed estimate. 
$eq_n(\ti\y_n)$ and $eq_n(\ck\y_n^J)$ are estimated in the norm
\EQ{
 Y:=L^4_tL^{6/5} + L^{8/3}_tL^{4/3}}
which is a non-admissible dual Strichartz norm, such that we have 
\EQ{ \label{Y2X}
 \left\|\int_0^t e^{i(t-s)(-\De+V_n)}P_nf(s)ds\right\|_X \lec \|f\|_Y,}
by rescaling \cite[Lemma 4.4]{NLSP1}, where $P_n$ is the rescaled projection 
\EQ{
 P_n:=\rS_n P_c \rS_n^{-1}.}

By the same argument as in \cite[(7.14)]{NLSP1}, we deduce from \eqref{weak van} that 
\EQ{ \label{local vanish Ga}
 \lim_{n\to\I}\|\Ga_n^J\|_{\st(|t-s_n^j|<\ta)}=0}
for any $\ta\in(0,\I)$ and any $0\le j<J$. 
Combining this, uniform bounds on $\La^j$ and $\Ga_n^J$ in $\Stz^1\subset X\cap L^4_tL^3$, the fact that the nonlinear part of $eq_n(\ck\y_n^J)$ is a linear combination of products of three from $\{\La_n^j,\Ga_n^J\}_j$ except the cubic power of each function, and H\"older estimates
\EQ{ \label{nonlin est 0mas scat}
 \pt \|fgh\|_{L^{8/3}_tL^{4/3}} \le \|f\|_{L^{16}_tL^{24/7}}\|g\|_{L^{16}_tL^{24/7}}\|h\|_{L^4_tL^6}, 
 \pr \|V_nf\|_{L^4_tL^{6/5}} \le \|V_n\|_{L^2}\|f\|_{L^4_tL^3}
 \lec \om_n^{-1/4}\|f\|_{L^4_tL^3},}
using $V\in L^2$, we obtain 
\EQ{
 \lim_{n\to\I}\|eq_n(\ck\y_n^J)\|_{Y(0,\I)}=0.}

For $\ti\y_n$, the scaling implies
\EQ{
 \|\ti\Phi_n\|_{L^\I_tL^3} = \|\Phi_n\|_{L^\I_tL^3} \sim \|z_n\|_{L^\I_t} \lec \om_n^{-1/4},}
hence by H\"older 
\EQ{
 \pt \|(\ti\Phi_n)^2\ti\y_n\|_{L^4_tL^{6/5}} \le \|\ti\Phi_n\|_{L^\I_tL^3}^2\|\ti\y_n\|_{L^4_tL^6} \lec \om_n^{-1/2}\|\ti\y_n\|_X, 
 \pr \|\ti\Phi_n(\ti\y_n)^2\|_{L^{8/3}_tL^{4/3}} \le \|\ti\Phi_n\|_{L^\I_tL^3}\|\ti\y_n\|_{L^4_tL^6}^{3/2}\|\ti\y_n\|_{L^\I_tL^3}^{1/2}
 \lec \om_n^{-1/4}\|\ti\y_n\|_X^{3/2}.}
Similarly we have, using $\|\p_j\Phi[z]\|_{L^{4/3}}\lec 1$,
\EQ{
 \pt\|\om_n^{-1}\rS_n\p_j\Phi[\ti z_n]\U{N}_j(\ti z_n,\y_n')\|_{L^{8/3}_tL^{4/3}}
 =\|\p_j\Phi[z_n]\U{N}_j(z_n,\y_n)\|_{L^{8/3}_tL^{4/3}}
 \pr\lec \|\p_j\Phi[z_n]\|_{L^\I_tL^{4/3}}\||\Phi[z_n]|+|\y_n|\|_{L^\I_tL^2}\|\y_n\|_{L^4_tL^6}^{3/2}\|\y_n\|_{L^\I_tL^3}^{1/2}
 \pr\lec \om_n^{-1/4}\|\ti\y_n\|_{L^4_tL^6}^{3/2}\|\ti\y_n\|_{L^\I_tL^3}^{1/2}.}
Summing these estimates yields 
\EQ{
 \|eq_n(\ti\y_n)\|_{Y(0,\I)} \lec \om_n^{-1/2}\|\ti\y_n\|_X+\om_n^{-1/4}\|\ti\y_n\|_X^{3/2}. }

Using the above estimates, and that $\|\ck\y_n^J\|_{X(0,\I)}$ is bounded by the assumption, we see that the error $e_n^J:=\ti\y_n-\ck\y_n^J$ satisfies 
\EQ{ \label{nonlin enJ}
 \pt\|(i\p_t-\De+V_n)e_n^J\|_{Y(I)}
 \pr\le\||\ti\y_n|^2\ti\y_n-|\ck\y_n^J|^2\ck\y_n^J\|_{Y(I)}
 +\|eq_n(\ti\y_n)\|_{Y(I)}+\|eq_n(\ck\y_n^J)\|_{Y(I)}
 \pr\lec [\|\ck\y_n^J\|_{X(I)}+\|e_n^J\|_{X(I)}+o(1)]^2\|e_n^J\|_{X(I)}+o(1),}
where $o(1)\to 0$ as $n\to\I$, uniformly on any interval $I\subset(0,\I)$. 
The global Strichartz estimate works only on the continuous spectrum part $P_ne_n^J$, but the other part is under control, because 
\EQ{
 e_n^J = \ti\y_n-\ck\y_n^J = R_n P_ne_n^J+(I-R_n P_n)\ck\y_n^J,}
where $R_n:=\rS_n R[\ti z_n] \rS_n^{-1}$, 
and so $1-R_nP_n=\rS_n\{1-R[\ti z_n]P_c\}\rS_n^{-1}$, while 
\EQ{
 1-R[z]P_c=\phi_0\be[z]} 
for some $\R$-linear operator $\be[z]:(L^1+L^\I)(\R^3) \to\C$, which is bounded uniformly for $z$. Hence 
\EQ{
 \pt\|[1-R_n P_n]\ck\y_n^J\|_{L^\I_tL^3}
 \lec \|\rS_n\phi_0\|_{L^3}\|\rS_n^{-1}\ck\y_n^J\|_{L^\I_tL^2}
 \lec \om_n^{-1/4}\|\ck\y_n^J\|_{L^\I_tL^2}=o(1),
 \pr\|[1-R_n P_n]\ck\y_n^J\|_{L^4_tL^6}
 \lec \|\rS_n\phi_0\|_{L^6}\|\rS_n^{-1}\ck\y_n^J\|_{L^4_tL^3}
 \lec \om_n^{-1/4}\|\ck\y_n^J\|_{L^4_tL^3}=o(1),}
using uniform bounds of $\ck\y_n^J$ in $\Stz^0\subset L^\I_tL^2\cap L^4_tL^3$. 
Since $L^\I_tL^3\cap L^4_tL^6$ on the left is stronger than $X$, and $R_n(t)$ is uniformly bounded on any $L^p(\R^3)$, we deduce that 
\EQ{ \label{P on enJ}
 \|e_n^J\|_{X(I)} \lec \|P_ne_n^J\|_{X(I)}+o(1),}
as $n\to\I$ uniformly on any interval $I\subset(0,\I)$. 
For a Gronwall-type argument by the non-admissible Strichartz, it is convenient to introduce the following norm
\EQ{
 \forall\fy\in H^1(\R^3), \pq \|\fy\|_{W_n}:=\|e^{it(-\De+V_n)}\fy\|_{X(0,\I)}.}
Let $(a,b)\subset(0,\I)$ such that $\|\ck\y_n^J\|_{X(a,b)}\ll 1$. Then by \eqref{Y2X}, \eqref{nonlin enJ} and \eqref{P on enJ}
\EQ{
 \pt \left.\begin{aligned}\|P_ne_n^J\|_{X(a,b)}\\ \|P_ne_n^J(b)\|_{W_n}\end{aligned}\right\}
 \le \|P_ne_n^J(a)\|_{W_n}+C\|(i\p_t-\De+V_n)e_n\|_{Y(a,b)},
 \prQ \|(i\p_t-\De+V_n)e_n\|_{Y(a,b)}\ll \|e_n^J\|_{X(a,b)}+o(1) \lec \|P_ne_n^J\|_{X(a,b)}+o(1),}
where the last term is absorbed by the first, and thus we obtain 
\EQ{
 \max(\|P_ne_n^J(b)\|_{W_n},\|P_ne_n^J\|_{X(a,b)}) \le 2\|P_ne_n^J(a)\|_{W_n}+o(1).}
Since $\ck\y_n^J$ is bounded in $X(0,\I)$, we can decompose $(0,\I)$ into intervals $I$ of a number $N$ independent of $n$ on which the above smallness in $X(I)$ is valid. Then iterating the above estimate on those intervals, and summing them up, we obtain 
\EQ{
 \|P_ne_n^J\|_{X(0,\I)} \le 2^{N+1}\|P_ne_n^J(0)\|_{W_n}+o(1)
 \lec \|e_n^J(0)\|_{H^1}+o(1)=o(1).}
Since it also implies $\|(i\p_t-\De+V^\om)e_n^J\|_{Y(0,\I)}\to 0$, the standard Strichartz with H\"older in $t$ implies that $\|e_n^J\|_{L^\I_tL^2(I)}\to 0$ on any bounded interval $I$. 
Then by interpolation with the $H^1$ bound, the convergence holds also in $L^\I_tH^s(I)$ for all $s<1$. 
Thus we have proven that if all the nonlinear profiles $\La^j$ scatter then for any $s<1$ and $T<\I$, 
\EQ{
 \lim_{n\to\I}\|\ti\y_n-\ck\y_n^J\|_{X(0,\I)\cap L^\I_tH^s(0,T)}=0, \pq \sup_{n\in\N}\|\ck\y_n^J\|_{X(0,\I)}<\I,}
which contradicts that $\|\ti\y_n\|_{X(0,\I)}=\|\y_n\|_{X(0,\I)}\sim\|\x_n\|_{X(0,\I)}\to\I$. 
Therefore, at least one profile $\La^j$ does not scatter. 
The definition \eqref{def LaGa} implies $\bA(\La^j)=\frac12\|\la^j\|_{H^1}^2\ge 0$ for $j\ge 1$, then using the asymptotic orthogonality at $t=0$, we have 
\EQ{ 
 \bA(\ti\y_n(0))\pt=\bA(\La^0)+\sum_{1\le j<J}\frac 12\|\la^j\|_{H^1}^2+\frac12\|\ga_n^J\|_{H^1}^2+o(1)
 \pr=\sum_{0\le j<J}\bA(\La^j)+\bA(\Ga^J_n)+o(1)}
as $n\to\I$.  
We also have $\bE^0(\La^0)\ge 0$, since otherwise $\La^0$ blows up in $t>0$, contradicting that $\La^0=\weak\Lim_{n\to\I}\ti\y_n$ is bounded in $H^1$ on $t\ge 0$. 

Since $\bM(\ti u_n)=\bM(\ti \Phi[z_n])+\bM(\ti \eta_n)$, $\|\ti\Phi_n\|_{\dot H^1}\lec\om_n^{-1/2}$ and $|\ml{V_n}(\ti u_n)|\lec\om_n^{-1/4}$ by \eqref{Vom small}, we have
\EQ{ \label{EM lim}
 \bA(\ti\y_n(0)) \pt\le \bA(\ti\y_n(0))+\bM(\ti\Phi_n(0))
 \pn= \bA^{\om_n}(\ti u_n)+o(1)
 \pn\le \bA(Q)+A^*+o(1)}
as $n\to\I$. Since $A^*\ll \bA(Q)$, we deduce that at most one profile can satisfy $\bA(\La^j)\ge\bA(Q)$, and all the others are below the ground state $Q$, and so scattering by \cite{HR}. 

Hence there is exactly one profile $\La^j$ which is not scattering and $\bA(\La^j)\ge\bA(Q)$. 
Then the above approximation by $\ck\y_n^J$ works up to $t=s_n^j+O(1)$, which, together with \eqref{local vanish Ga}, implies that 
\EQ{
 \ti\y_n(s_n^j+t)\to\La^j(t) \IN{\weak{H^1_r}}\pq(n\to\I)} 
for any $t\in\R$ if $j\ge 1$ and for any $t\ge 0$ if $j=0$. 
Hence 
\EQ{
 \pt\limsup_{n\to\I}\frac12\|\ti\y_n(s_n^j+t)-\La^j(t)\|_{H^1}^2+\bM(\ti\Phi_n(s_n^j+t))
 \pr=\limsup_{n\to\I}\bA(\ti\y_n(s_n^j+t))+\bM(\ti\Phi_n(s_n^j+t))-\bA(\La^j(t))
 \pn\le A^* \ll \de_*^2.}
On the other hand, $d_{\om_n}(\ti u_n(s_n^j+t))\ge\de_*$, $\|Q_{\om_n}-Q\|_{H^1}\to 0$ and $\|\ti\Phi_n\|_{\dot H^1}\to 0$ imply 
\EQ{
 \|\ti\y_n(s_n^j+t)-e^{i\te}Q\|_{H^1} \gec \de_* - C\|\ti\Phi_n(s_n^j+t)\|_2 + o(1)\gec \de_*.}
Combining the above two estimates and $\|\La^j(t)-e^{i\te}Q\|_{H^1}\ll\de_*$ yields a contradiction. 

Therefore we have a uniform lower bound $\inf_{\te\in\R}\|\La^j(t)-e^{i\te}Q\|_{H^1}\gec\de_*$, as well as energy bound $\bA(\La^j)\le\bA(Q)+A^*$. Hence by the result in \cite{NLS} for the NLS without potential, if $A^*\ll\de_*^2$, then $\La^j$ scatters to $0$ as $t\to\I$, which is a contradiction. 

Thus we have reached contradiction both for $\OM<\I$ and for $\OM=\I$. 
Therefore $A^*\gec\de_*^2$ and we have proven 
\begin{lem} \label{lem:scatafter}
There is a constant $c_*\in(0,c_X]$ such that for every $\om\ge\om_\star$ and every solution $u_\om$ of \eqref{rscNLS} satisfying 
\EQ{
 \pt \bA^\om(u_\om)\le \bA^\om(Q_\om)+c_*\de_*^2, \pq \inf_{t\ge 0}d_\om(u_\om(t))\ge\de_*, \pq \Sg_\om(u_\om(0))=+1,}
scatters to $\Phi$ as $t\to\I$. 
\end{lem}

\subsection{Classification of the dynamics}
Let $\om\ge\om_\star$ and let $u_\om$ be a solution of \eqref{rscNLS} from $t=0$ with the maximal existence time $T_+\in(0,\I]$ satisfying the constraint: 
\EQ{ \label{energy const}
  \bA^\om(u_\om)< \bA^\om(Q_\om) + c_*\de_*^2,}
where $\de_*,c_*>0$ are the small constants introduced in Lemmas \ref{lem:one pass} and \ref{lem:scatafter}. $u(t):=\rS_\om^{-1}u_\om(\om t)$ solves the original equation \eqref{NLSP} on $[0,T_+/\om)$ with 
\EQ{
 \bA_\om(u) < \bA_\om(\Psi[\om])+\om^{1/2}c_*\de_*^2.}
The distance function in the rescaled variable is abbreviated as before by
\EQ{
 d(t):=d_\om(u_\om(t)).}

If $\inf_{0\le t<T_+}d(t)\ge\de_*$, 
then \eqref{energy const} and $c_*\le c_X$ imply that $u_\om(t)\in\ck\cH_\om$ for all $t\in[0,T_+)$. 
Moreover, Lemmas \ref{lem:bupafter} and \ref{lem:scatafter} imply
\EQ{ \label{bs after}
 \Sg_\om(u_\om(0))=\CAS{+1 \implies \text{$u_\om$ scatters to $\Phi$ as $t\to\I$},\\
 -1 \implies \text{$u_\om$ blows up in $t>0$.}}}

If $\inf_{0\le t<T_+}d(t)<\de_*$, then the one-pass lemma \ref{lem:one pass} implies that $d(t)<\de_*$ on $t\in(t_1,t_2)$ and $d(t)>\de_*$ on $t\in(t_2,T_+)$, for $t_1,t_2\in[0,T_+]$ defined by 
\EQ{
 t_1:=\inf\{t\in[0,T_+)\mid d(t)<\de_*\},
 \pq t_2:=\sup\{t\in[0,T_+)\mid d(t)<\de_*\}.} 

If $t_1>0$, then applying the ejection lemma \ref{lem:eject} from $t=t_1$ backward, there exists $t_0<t_1$ such that $d(t)$ is strictly and exponentially decreasing on $[t_0,t_1]$ with 
\EQ{
 d(t) \sim e^{-\al_\om(t-t_0)}\de_X \sim e^{-\al_\om(t_1-t)}\de_*,
 \pq d(t_0)=\de_X>d(t_1)=\de_*.}

If $t_2<T_+$, then we have the same dichotomy as in \eqref{bs after} at $t=t_2$. 

If $t_2=T_+$, then the uniform bound $d(t)\le\de_*$ for $t\ge t_1$ implies $t_2=T_+=\I$, and by the trapping lemma \ref{lem:stay}, there exists $t_3\in[t_1,\I]$ such that $d(t)$ is strictly and exponentially decreasing on $[t_1,t_3)$ with 
\EQ{
 \CAS{t_1\le t<t_3 \implies d(t)\sim e^{-\al_\om(t-t_1)}\de_*, \\
 t_3<t<\I \implies c_Xd(t)^2 < \bA^\om(u_\om)-\bA^\om(Q_\om).}}
This implies that $u(t)\in\cN_{\de}$ for large $t$ and for some $\de\sim(c_*/c_X)^{1/2}\de_*$, so $u$ is trapped by $\Psi$ as $t\to\I$. 
We have $t_3=\I$ if and only if $u(t)$ is strongly convergent to $e^{-i\om(t-a)}\Psi[\om]$ in $H^1_r$ as $t\to\I$ for some $a\in\R$.

Now that we have proven the classification part of the main Theorem \ref{thm:main}, together with some description of each behavior, it remains to see for which initial data each of the possibilities occurs, especially for the trapping and the transition. 

\section{Center-stable manifold of the excited solitons} \label{s:mfd}
In this section, we show that the set of initial data for which the solution is trapped by $\Psi$ is a $C^1$ manifold of codimension $1$, and that it is a threshold between the scattering to $\Phi$ and the blow-up. 
It is a center-stable manifold of $\Soli_1|_{\bM\ll 1}$, its time inversion is a center-unstable manifold, and there are all the 9 types of solutions around the transversal intersection of them. 

\subsection{Construction around the excited states}
First we construct a manifold around a fixed excited soliton $e^{-it}Q_\om$ by the bisection argument as a graph of $(\lb_-(0),\z(0))\mapsto \lb_+(0)$ in the decomposition \eqref{exp v}. 
\begin{thm} \label{thm:mfd}
Take $\de_\pm\in(0,\de_X)$ such that $\de_-/\de_+,\de_+/\de_X$ and $(\de_-\de_++\de_+^3)/(c_X\de_X^2)$ are all small enough. Then a unique $C^1$ function $G_\om$ is defined for each $\om\ge\om_*$ on 
\EQ{
 U_\om:=\{(\lb_-,\z)\in\R\times\cZ_\om  \mid \max(|\lb_-|,\|\z\|_\om)<\de_-\}, }
such that for any 
\EQ{
 (\te,\lb_+,\lb_-,\z)\in \cU_\om:=(\R/2\pi\Z)\times(-\de_+,\de_+)\times U_\om,} 
the solution of \eqref{rscNLS} with the initial condition $u_\om(0)=\sC_\om(\te,\lb_+,\lb_-,\z)$ satisfies 
\EN{
\item If $\lb_+=G_\om(\lb_-,\z)$ then $d_\om(u_\om(t))<\de_X/2$ for all $t\ge 0$. 
\item If $\lb_+\not=G_\om(\lb_-,\z)$ then $d_\om(u_\om(t))$ reaches $\de_X$ at some $t_X>0$, where 
\EQ{
 \sB^\om_1(u_\om(t_X))\sim \sign(\lb_+-G_\om(\lb_-,\z))\de_X.}
}
Moreover, we have $|G_\om(\lb_-,\z)|\lec|\lb_-|^2+\|\z\|_\om^2$. 
\end{thm}
The above characterization (1)-(2) implies that the value of $G_\om$ is independent of the choice of $\de_\pm$. 
The distance upper bound $\de_X/2$ in the case (1) is chosen just for distinction from the case (2), but it can be made arbitrarily small by taking $\de_\pm$ smaller. 

\subsubsection{Existence of $G_\om$}
First, we prove the existence of a value of $\lb_+$ for which $u_\om$ is trapped. 
Fix $\de_\pm$ such that $0<\de_-\ll\de_+\ll\de_X$ and $\de_-\de_++\de_+^3\ll c_X\de_X^2$. Take any $(\lb_-,\z)\in U_\om$. Since $d_\om(u_\om(0))\lec\de_+\ll\de_X$, if $d_\om(u_\om(t))$ reaches $\de_X$ at some $t=t_X>0$, then the ejection lemma \ref{lem:eject} implies $|b_1(t_X)|\sim\de_X$. 
Let $B_\pm$ be the sets of such $b_+\in(-\de_+,\de_+)$ that $b_1(t_X)\sim\pm\de_X$ at the first ejection time in $t>0$. 

$B_\pm$ are open, because the ejection lemma applies to perturbed solutions $u_\om'$ as long as $d_\om(u_\om'(t_X))\sim\de_X\gg d_\om(u_\om'(0))$, while $\sign b_1$ remains constant. 
$B_+\cap B_-=\emptyset$ by definition. 
Both sets are non-empty, because for $|b_+|\gg|b_-|+\|\z\|_\om$, \eqref{exp Aom} and \eqref{eq d2} imply that the ejection condition \eqref{eject0} is satisfied at $t=0$, then $\sign b_1(t)=\sign b_+(t)$ is preserved until $d_\om(u_\om(t))$ reaches $\de_X$. 
 
Hence by the connectedness, $(-\de_+,\de_+)\setminus(B_+\cup B_-)$ is not empty either. If $b_+$ is in this set, then by definition of $B_\pm$, we have $d_\om(u_\om(t))<\de_X$ for all $t\ge 0$, and so the trapping lemma \ref{lem:stay} applies to $u_\om$ on $t\ge 0$. Since
\EQ{
 \bA^\om(u_\om)-\bA^\om(Q_\om) \pt=-2b_+b_-+\frac12\LR{\cL^\om\z|\z}-C^\om(b_+g^\om_+ + b_-g^\om_-+\z) 
 \pr\lec \de_-\de_+ +\de_+^3,}
the trapping lemma implies that for all $t\ge 0$
\EQ{ 
 d_\om(u_\om(t))^2 \pt\le \min(d_\om(u_\om(0))^2,c_X^{-1}(\bA^\om(u_\om)-\bA^\om(Q_\om))) 
 \pr\lec \min(\de_+^2,c_X^{-1}(\de_-\de_++\de_+^3)) \ll\de_X^2.}

\subsubsection{Lipschitz estimate}
Next we prove a key Lipschitz estimate for a generalized difference equation of \eqref{eq v} for trapped solutions, which will imply the uniqueness and Lipschitz continuity of $G_\om$. 

Before taking the difference, we prepare time-local bound on the Strichartz norm. Let $v$ be a solution of \eqref{eq v} on an interval $I$. Applying the Strichartz estimate of $e^{-it\De}$ to the equation of $v$, we deduce that there exists a small constant $\de_S\in(0,1)$ such that 
\EQ{
 \|v\|_{L^\I_tH^1(I)}\le\de_S \tand |I|\le 1
 \implies \|v\|_{\Stz^1(I)} \lec \|v\|_{L^\I_tH^1(I)}.}
In particular, denoting 
\EQ{
 \|u\|_{\Stz^1\ul(I)}:=\sup_{J\subset I,\ |J|\le 1}\|u\|_{\Stz^1(J)},}
we have 
\EQ{
 \|v\|_{L^\I_tH^1(0,\I)}\le\de_S
 \implies \|v\|_{\Stz^1\ul(0,\I)} \lec \|v\|_{L^\I_tH^1(0,\I)}.}

Now let $v\zr,v\on$ be two solutions of \eqref{eq v}, and let $\vec v:=(v\zr,v\on)$. 
Then the difference $\diff v\pa=v\on-v\zr$ satisfies
\EQ{ \label{eqdifv}
 (\p_t-i\cL^\om)\diff v\pa \pt= -iQ_\om \diff m^\om(v\pa) + \diff{\cN^\om(v\pa)}
 \pn= [-iQ_\om \deri m^\om(\vec v) + \deri \cN^\om(\vec v)]\diff v\pa,}
where $\deri m^\om(\vec v)$ and $\deri \cN^\om(\vec v)$ are operators defined by the following: for any function $X(v)$ which is Fr\'echet differentiable in $v$, and for $\vec v=(v\zr,v\on)$, define  
\EQ{ \label{def deri}
 \deri X(\vec v):=\int_0^1 X'((1-\te)v\zr+\te v\on) d\te,}
so that the difference of $X$ at $v\zr$ and $v\on$ can be written as 
\EQ{
 \diff X(v\pa)=\deri X(\vec v)\diff v\pa.}
The Fr\'echet derivatives $(N^\om)'(v), (\cN^\om)'(v):H^1\to H^{-1}$ and $(m^\om)'(v):H^1\to\R$ can be written explicitly as follows.
\EQ{
 (N^\om)'(v)\fy \pt= 2Q_\om(3v_1\fy_1+v_2\fy_2)+3v_1^2\fy_1+2v_1v_2\fy_2
 \prq + i[2Q_\om(v_2\fy_1+v_1\fy_2)+2v_1v_2\fy_1+3v_2^2\fy_2],
 \\ (m^\om)'(v)\fy &=[\LR{Q_\om|Q_\om'}+\LR{v|Q_\om'}]^{-2}[\LR{v|Q_\om}+\LR{N^\om(v)|Q_\om'}]\LR{Q_\om'|\fy}
 \prq-[\LR{Q_\om|Q_\om'}+\LR{v|Q_\om'}]^{-1}[\LR{Q_\om|\fy}+\LR{(N^\om)'(v)\fy|Q_\om'}],
 \\ (i\cN^\om)'(v)\fy &= [(m^\om)'(v)\fy]v+m^\om(v)\fy+(N^\om)'(v)\fy.}
A similar expression is obtained for $\deri N^\om(\vec v)$ by replacing $v$ in $(N^\om)'(v)$ with $v\zr$ and $v\on$, taking the linear combination of such terms. 
The computation for $\deri m^\om(\vec v)$ is slightly more complicated because of the quotient, but still elementary. 

The above equation \eqref{eqdifv} is linear in $\diff v\pa$, so the difference quotient, as well as its limit, namely the derivative, solves the same form of equation. 
Hence it is convenient to derive a Lipschitz estimate for general solutions $v\pb$ of the linear equation 
\EQ{ \label{eqpb}
 (\p_t-i\cL^\om)v\pb = [-iQ_\om \deri m^\om(\vec v)+\deri \cN^\om(\vec v)]v\pb,
 \pq \vec v:=(v\zr,v\on),}
where $v\zr,v\on$ are given functions satisfying for some small $\de>0$, 
\EQ{
 \max_{j=0,1}\|v\oj\|_{\Stz^1\ul(0,\I)} \le \de.}
In other words, we ignore the relation $\diff v\pa=v\on-v\zr$ in \eqref{eqdifv}. 

It is easy to see, using the Strichartz estimate, that \eqref{eqpb} is wellposed for $H^1\ni v\pb(0)\mapsto v\pb\in\Stz^1\loc(0,\I)$, and that the solution satisfies 
\EQ{
 \p_t \LR{iv\pb|Q_\om'}=0} 
(by differentiating the equation \eqref{eq m} of $m^\om$). 
Hence the orthogonality $\LR{iv\pb|Q_\om'}=0$ is preserved if it is initially fulfilled. 

$v\pb$ is decomposed as before by the symplectic orthogonality 
\EQ{
 v\pb = \lb_+\pb g^\om_+ + v\pc = \lb_+\pb g^\om_+ + \lb_-\pb g^\om_- + \z\pb,
 \pq \lb_\pm\pb:=P^\om_\pm v\pb, \pq \LR{i\z\pb|g^\om_\pm}=0.}
Then using the equation \eqref{eqpb}, we obtain 
\EQ{ \label{lapb inc}
 |(\p_t \mp 2\al_\om)|\lb_\pm\pb|^2|
 \pt=2|\LR{P^\om_\pm \deri \cN^\om(\vec v)v\pb|\lb_\pm\pb}| 
 \pr\lec \|\vec v\|_{H^1}\|v\pb\|_{H^1}|\lb_\pm\pb|
 \lec \de \|v\pb\|_{H^1}|\lb_\pm\pb|,}
and, using H\"older and partial integration in $x$,
\EQ{ 
  |\p_t\|\z\pb\|^2_\om|\pt=|\LR{P^\om_\ce \deri \cN^\om(\vec v)v\pb|\cL^\om \z\pb}|
 \pr\lec \BR{\|\vec v\|_{L^\I}+\|\vec v\|_{L^\I}^2}\|v\pb\|_{H^1}\|\z\pb\|_{H^1}
 \prQ + \BR{\|\na \vec v\|_{L^3}+\|\vec v\|_{L^\I}\|\na \vec v\|_{L^3}}\|v\pb\|_{L^6}\|\z\pb\|_{H^1}.}
Hence for any interval $I\subset[0,\I)$ with length $|I|\le 1$, 
\EQ{ \label{gapb inc}
 [\|\z\pb\|^2_\om]_{\p I} \lec \de \|v\pb\|_{L^\I_t H^1(I)}\|\z\pb\|_{L^\I_t H^1(I)},}
using that $\|\vec v\|_{L^4_t(W^{1,3}\cap L^\I)(I)} \lec \|\vec v\|_{\Stz^1(I)}$. 

Combining \eqref{lapb inc} and \eqref{gapb inc}, we deduce that there exist absolute constants $C_0\in(1,\I)$ and $\de_0\in(0,1)$ such that if $\de\le\de_0$ then for every $t_0\ge 0$
\EQ{ \label{vpb locbd}
 \CAS{\sup_{t_0\le t\le t_0+1}\|v\pb\|_\om \le C_0\|v\pb(t_0)\|_\om,
 \\ \sup_{t_0\le t\le t_0+1}\|v\pc\|_\om \le \|v\pc(t_0)\|_\om + C_0\de\|v\pb(t_0)\|_\om.} }

Suppose that for some $\ell>0$ and $t_0\ge 0$, we have 
\EQ{
 \|v\pc(t_0)\|_\om \le \ell|\lb_+\pb(t_0)|,}
and define $t_1\in(t_0,\I]$ by 
\EQ{
 t_1=\inf\{t>t_0 \mid \|v\pc(t)\|_\om > (\ell+C_0\de+C_0\de \ell)|\lb_+\pb(t)|\}.}
If $\de,\ell>0$ are chosen such that 
\EQ{ \label{cond1 deL}
  \de(\ell+C_0\de+C_0\de \ell) \ll \al,}
then for $t\in[t_0,t_1)$ we have, using $\al_\om\in(\frac{9}{10}\al,\frac{11}{10}\al)$,  
\EQ{
 \de\|v\pb\|_\om \ll \al_\om |\lb_+\pb|,}
and injecting this into \eqref{lapb inc}, 
\EQ{ \label{lapb grw}
 \p_t|\lb_+\pb|^2 \ge \al_\om|\lb_+\pb|^2.}
Hence $|\lb_+\pb(t)|$ is increasing on $[t_0,t_1)$. 
On the other hand, \eqref{vpb locbd} implies 
\EQ{ \label{vpc bd}
 t_0\le t\le t_0+1 \implies \|v\pc\|_\om \pt\le (1+C_0\de)\|v\pc(t_0)\|_\om+C_0\de|\lb_+\pb(t_0)|
 \pr\le (\ell+C_0\de+C_0\de \ell)|\lb_+\pb(t_0)|.}
Therefore by the definition of $t_1$, we deduce that 
\EQ{
 t_1 > t_0+1.}
In particular, we obtain from \eqref{lapb grw} and \eqref{vpc bd}, 
\EQ{ \label{vpc+1}
 \|v\pc(t_0+1)\|_\om \le (\ell+C_0\de+C_0\de \ell)e^{-\al_\om/2}|\lb_+\pb(t_0+1)|.}
Then imposing another condition on $(\de,\ell)$: 
\EQ{ \label{cond2 deL}
 (\ell+C_0\de+C_0\de \ell) \le \ell e^{\al/3}}
leads to 
\EQ{
 \|v\pc(t_0+1)\|_\om \le \ell|\lb_+\pb(t_0+1)|,}
so by induction we deduce that for all $n\in\N$, 
\EQ{
 \pt \|v\pc(t_0+n)\|_\om \le \ell|\lb_+\pb(t_0+n)|.}
Moreover, $t_1=\I$ and \eqref{lapb grw} is valid for all $t\ge t_0$. 
Thus we have obtained the following key lemma, choosing $C_1\gg C_0$. 

\begin{lem} \label{lem:Lip}
There is a constant $C_1\in(2C_0,\I)$ such that if $\de>0$ is small enough and $\om\ge \om_*$, 
$v\oj$ satisfies  
$\|v\oj\|_{\Stz^1\ul(0,\I)}\le\de$ for $j=0,1$, 
and $v\pb\in C([0,\I);H^1)$ satisfies the equation \eqref{eqpb} for $t\ge 0$, together with $\LR{iv\pb(0)|Q_\om'}=0$ and 
\EQ{
 \|P^\om_\cs v\pb(0)\|_\om \le \ell|P^\om_+ v\pb(0)|}
for some $\ell>0$ in the range 
\EQ{ \label{range L}
 \frac{C_1\de}{\al} \le \ell \le \frac{\al}{C_1\de},}
then for all $n\in\N$ and all $t\ge 0$ we have 
\EQ{
 \pt \|P^\om_\cs v\pb(n)\|_\om \le \ell|P^\om_+ v\pb(n)|,
 \pq \|P^\om_\cs v\pb(t)\|_\om \le \ell(1+C_1\de)|P^\om_+v\pb(t)|,}
and
\EQ{
 |P^\om_+v\pb(t)| \ge e^{\al_\om t/2}|P^\om_+v\pb(0)|.}
\end{lem}
Note that \eqref{range L} is a sufficient condition to have \eqref{cond1 deL} and \eqref{cond2 deL}, and the range of $\ell$ is non-empty for $0<\de\le\al/C_1$. 

\subsubsection{Uniqueness and regularity of $G_\om$}
Let $0<\de\ll 1$ and $\ell>0$ satisfy \eqref{range L}. Let $v\zr,v\on$ be two solutions of \eqref{eq v} satisfying $v\oj(0)\in\V_\om$ and $\|\vec v\|_{\Stz^1\ul(0,\I)}\le\de$. 
Then $\diff v\pa=v\on-v\zr$ satisfies the equation \eqref{eqpb}. 
Suppose that at some $t_0\ge 0$ we have 
\EQ{
 \|\diff v\pa_\cs(t_0)\|_\om \le \ell|\diff\lb\pa_+(t_0)|.}
Then the above lemma implies that $\diff\lb\pa_+$ is exponentially growing for $t\ge t_0$, which contradicts $\vec v\in L^\I_tH^1(0,\I)$, unless $\diff v\pa(t_0)=0$. Hence for all $t\ge 0$, we have 
\EQ{
 |\diff\lb\pa_+(t)| \le \ell^{-1}\|\diff v\pa_\cs(t)\|_\om,}
where the Lipschitz constant can be optimized by taking the largest possible $\ell=O(\de^{-1})$ in the lemma. 
Then going back to \eqref{vpb locbd}, we also obtain 
\EQ{
 \|\diff v\pa_\cs(t)\|_\om \le  e^{2C_0\de(t+1)}\|\diff v\pa_\cs(0)\|_\om}
for all $t\ge 0$. 
Thus we obtain (using $C_1\ge 2C_0$) 
\begin{lem}
Let $0<\de<\al/C_1$ be small enough and $\om\ge \om_*$. 
Let $v\zr,v\on$ be two solutions of \eqref{eq v} on $t\in[0,\I)$ satisfying the orthogonality $v\oj(0)\in\V_\om$ and $\|v\oj\|_{L^\I_t H^1(0,\I)}\le \de$. 
Then we have, for all $t\ge 0$, 
\EQ{
 \al |\diff P^\om_+v\pa(t)| \le C_1\de\|\diff P^\om_\cs v\pa(t)\|_\om,}
and for all $t\ge 0$, 
\EQ{
 \|\diff P^\om_\cs v\pa(t)\|_\om \le e^{C_1\de(t+1)}\|\diff P^\om_\cs v\pa(0)\|_\om.}
\end{lem}
The above lemma implies the uniqueness of $G_\om(\lb_-,\z)$ for each small $(\lb_-,\z)$, as well as the Lipschitz continuity. 
To show the G\^ateaux differentiability, fix arbitrary $\fy,\psi\in \cZ_\om$ and $a,b\in\R$ such that $\|\fy\|_\om+|a|\ll 1$, and let $v\zr,v\on$ be two solutions of \eqref{eq v} satisfying $\|v\oj\|_{L^\I_tH^1(0,\I)}\le\de$ and 
\EQ{
 P^\om_\cs v\zr(0)=a g^\om_- + \fy, \pq P^\om_\cs v\on(0)=(a+hb)g^\om_- + (\fy+h\psi)}
with a small parameter $\R\ni h\to 0$. 

Then $w:=\diff v\pa/h$ solves the equation \eqref{eqpb} with the initial condition $P^\om_\cs w(0)=bg^\om_-+\psi$  independent of $h$, and the above lemma implies that for all $t\ge 0$ 
\EQ{ \label{bd La'}
 \al|P^\om_+w(t)| \le C_1\de\|P^\om_\cs w(t)\|_\om,
 \pq \|P^\om_\cs w(t)\|_\om \le e^{C_1\de(t+1)}\|P^\om_\cs w(0)\|_\om.}
Using the local wellposedness of \eqref{eqpb} as well, we deduce that $w$ is bounded in $\Stz^1(0,T)$ as $h\to 0$ for any $T<\I$. 
The uniform bound together with the equation implies that there is a sequence of $h\to 0$ along which $w$ converges to some $w_\I\in\Stz^1\loc(0,\I)$ in $C([0,\I);\weak{H^1})$. The limit $w_\I$ solves the equation \eqref{eqpb} with $\vec v=(v,v)$, satisfying 
\EQ{ \label{bd wI}
 \pt \al|P^\om_+w_\I(t)| \le C_1\de\|P^\om_\cs w_\I(t)\|_\om,
 \pq \|P^\om_\cs w_\I(t)\|_\om \le e^{C_1\de(t+1)}\|P^\om_\cs w_\I(0)\|_\om, }
where the last normand is the prescribed $bg^\om_-+\psi$. 
If there is another limit $w_\I'$ along another sequence of $h\to 0$, then $w_\I-w_\I'$ satisfies the same equation \eqref{eqpb} and the same estimates with $P^\om_\cs(w_\I-w_\I')(0)=0$, therefore $w_\I\equiv w_\I'$. 
Hence the limit is unique, and so the convergence holds for the entire limit $h\to 0$. Thus we obtain the G\^ateaux derivative of $G_\om$ at $(a,\fy)$ in the direction $(b,\psi)$
\EQ{
 G_\om'(a,\fy)(b,\psi)=P^\om_+w_\I(0)\in\R,}
which is bounded linear on $(b,\psi)\in\R\times\cZ_\om$, because it is determined by the linear equation \eqref{eqpb} with the boundedness by \eqref{bd wI}:
\EQ{
 \|G_\om'(a,\fy)\|_{\B(\R\times\cZ_\om,\R)} \lec C_1\al^{-1}\|(a,\fy)\|_{\R\times H^1}.} 

To show that the G\^ateaux derivative is continuous with respect to $(a,\fy)$ in the operator norm, take any sequence $(a_n,\fy_n)\in\R\times \cZ_\om$ strongly convergent to $(a,\fy)$, and any sequence $(b_n,\psi_n)\in\R\times \cZ_\om$ weakly convergent to $(b,\psi)$. 
Let $v_n$ be the solution of \eqref{eq v}, and $w_n$ be the solution of \eqref{eqpb} with $\vec v=(v_n,v_n)$, satisfying 
\EQ{
 \pt v_n(0)=G_\om(a_n,\fy_n)g^\om_+ + a_ng^\om_- + \fy_n,
 \pr w_n(0)=G_\om'(a_n,\fy_n)(b_n,\psi_n)g^\om_+ + b_n g^\om_- + \psi_n.}
By the local wellposedness for \eqref{eq v}, we have $v_n\to v_\I$ in $\Stz^1\loc(0,\I)$, where $v_\I$ is the solution of \eqref{eq v} with 
\EQ{
 v_\I(0)=G_\om(a,\fy)g^\om_+ + ag^\om_- + \fy.}
Also we have 
\EQ{
 \|w_n(t)\|_\om \lec e^{C_1\de t}[|b_n|+\|\psi_n\|_{H^1}],}
which is uniformly bounded on any finite interval. 
These uniform bounds together with the equation for $w_n$ imply that $w_n$ converges to some $w\in\Stz^1\loc(0,\I)$ in $C([0,\I);\weak{H^1})$, at least along a subsequence. Then the limit $w_\I$ solves \eqref{eqpb} with $\vec v=(v_\I,v_\I)$ and $P^\om_\cs w_\I(0)=b g^\om_-+\psi$, satisfying the orthogonality and 
\EQ{
 \max\BR{\frac{\al}{C_1\de}|P^\om_+w_\I(t)|, \|P^\om_\cs w_\I(t)\|_\om} \le e^{C_1\de(t+1)}\|P^\om_\cs w_\I(0)\|_\om.}
The uniqueness of such a solution implies 
\EQ{
 P^\om_+ w_\I(0)=G_\om'(a,\fy)(b,\psi)}
as well as the convergence of $w_n$ along the full sequence $n\to\I$. Hence 
\EQ{
 \lim_{n\to\I}G_\om'(a_n,\fy_n)(b_n,\psi_n)=G_\om'(a,\fy)(b,\psi).}
Since this holds for any weakly convergent $(b_n,\psi_n)$, we have the convergence of $G_\om'(a_n,\fy_n)\to G_\om'(a,\fy)$ in the operator norm. In short, $G_\om$ is a $C^1$ function on a small neighborhood of $0$ in $\R\times \cZ_\om$. 

\subsection{Nine sets of solutions around the excited states}
We have obtained a manifold for each $\om\ge \om_*$ in the local coordinate $\sC_\om$ of Lemma \ref{lem:sum}, that is
\EQ{ \label{def Mom0}
 \M^\om_0 \pt:= \{  \sM_\om(\te,b_-,\z) \mid (\lb_-,\z)\in U_\om,\ \te\in\R/2\pi\Z \}, 
 \prQ \sM_\om(\te,b_-,\z):=\sC_\om(\te,G_\om(\lb_-,\z),\lb_-,\z), }
consisting of trapped solutions, 
in the open neighborhood $\sC_\om(\cU_\om)$ of $\cQ_\om$. 
$\M^\om_0$ has codimension $1$, separating the complement into two open sets 
\EQ{
 \M^\om_\pm := \sC_\om(\{ (\te,\lb_+,\lb_-,\z)\in\cU_\om \mid \pm(\lb_+-G_\om(\lb_-,\z))>0 \}). }
The solutions in $\M^\om_\pm$ are ejected with $b_1(t_X)\sim\pm\de_X$ at some ejection time $t_X>0$ (which depends on the solution).

The time inversion of $\M^\om_0$ is the complex conjugate 
\EQ{
 \ba{\M^\om_0}=\{\sC_\om(\te,\lb_+,G_\om(\lb_+,\ba{\z}),\z) \mid (\lb_+,\z)\in U_\om, \te\in\R/2\pi\Z \}, } 
which is a $C^1$ manifold of codimension $1$ consisting of solutions trapped by $\cQ_\om$ for $t\le 0$. 
Since $\ba{g^\om_\pm}=g^\om_\mp$ and $|G_\om'|\ll 1$, $\M^\om_0$ and $\ba{\M^\om_0}$ intersect transversely: the implicit function theorem yields a unique $C^1$ function $\vec G_\om:\{\z\in\cZ_\om\mid\|\z\|_\om<\de_-\}\to\{(b_+,b_-)\in\R^2\mid \max |b_\pm|<\de_-\}$ such that 
\EQ{ 
 (\lb_+,\lb_-)=\vec G_\om(\z) \iff \lb_+=G_\om(\lb_-,\z) \tand \lb_-=G_\om(\lb_+,\ba\z),}
and $\cU_\om\setminus(\M^\om_0\cap\ba{\M^\om_0})$ consists of four open sets with distinct local behavior in $t>0$ and in $t<0$, according to $\sign(\lb_+ - G_\om(\lb_-,\z))$ and $\sign(\lb_- - G_\om(\lb_+,\ba{\z}))$. 
Thus all solutions of \eqref{rscNLS} starting near $\cQ_\om$ are classified into 9 non-empty sets of solutions: 
\EQ{
 \{\M^\om_j\cap\ba{\M^\om_k}\}_{j,k\in\{0,\pm\}}. }
All these about the manifolds rely only on the instability, and they are independent of the global dynamics investigated in the previous sections. 

Under the constraints \eqref{energy const} and $\om\ge\om_\star$, the scattering/blow-up away from the excited states also applies to them, leading to the characterization by global behavior: 
\EQ{
 \M^\om_+\cap\cH^\om_*\subset \rS_\om\cS,\pq \M^\om_-\cap\cH^\om_*\subset \rS_\om\B, 
 \pq \M^\om_0\cap\cH^\om_* \subset \rS_\om\T_{C\de_*},} 
where $\cH^\om_*$ denotes the constrained region for $\om\ge\om_\star$
\EQ{
 \cH^\om_*:=\{\fy\in H^1_r(\R^3)\mid \bA^\om(\fy)<\bA^\om(Q_\om)+c_*\de_*^2\}.}

Thus we have proven the existence of infinitely many orbits of the 9 cases in the main Theorem \ref{thm:main}. 
It remains to see that the manifold extends to the entire set of trapped solutions, together with the threshold property. 

\subsection{Extension of the manifold}
Let $\de_M\in(0,\de_X)$ so small that $\de_M\ll c_X$ and for any $\om\ge\om_*$ and any $(\te,b_+,b_-,\z)\in\sU_\om$, 
\EQ{
 d_\om(\sC_\om(\te,b_+,b_-,\z))<\de_M \implies |b_+|<\de_+ \tand \max(|b_-|,\|\z\|_\om)<\de_-.}
Then the trapping lemma \ref{lem:stay} implies that 
\EQ{ \label{mfd restr}
 \M^\om_1:=\{\fy\in \M^\om_0 \mid d_\om(\fy)<\de_M,\ \bA^\om(\fy)<\bA^\om(Q_\om)+c_X\de_M^2\}}
is forward invariant by the flow of \eqref{rscNLS}. 
Let $\M^\om_2$ be the maximal backward extension by the flow of this set. 
Then $\M^\om_2$ is the union of all orbits of forward global solutions $u_\om$ of \eqref{rscNLS} satisfying  
\EQ{ \label{trap for mfd}
 \bA^\om(u_\om)<\bA^\om(Q_\om)+c_X\de_M^2,
 \pq \limsup_{t\to\I}d_\om(u_\om(t))<\de_X.}
Indeed, the trapping lemma automatically improves the last bound to 
\EQ{ \label{improved dist bd}
 \limsup_{t\to\I}d_\om(u_\om(t))^2 \le c_X^{-1}(\bA^\om(u_\om)-\bA^\om(Q_\om))<\de_M^2 \ll\de_X^2,}
so $u_\om(t)$ belongs to \eqref{mfd restr} for large $t$. 
Hence $\M^\om_2$ is a $C^1$ manifold with codimension $1$ and invariant by the flow of \eqref{rscNLS}. $\M^\om_2$ is unbounded, since it contains solutions that blow up in $t<0$. 

To see that $\M^\om_2$ is connected, let $u\zr$ and $u\on$ be two solutions of \eqref{rscNLS} on $\M^\om_2$. 
By the above argument, there exists $T>0$ such that both $u\zr$ and $u\on$ are in $\M^\om_1$ for all $t\ge T$. 
Let $u\zr(T)=\sM_\om(\te\zr,b_-\zr,\z\zr)$, then $\max(|b_-\zr|,\|\z\zr\|_\om)<\de_-$. 
For each $\y\in[0,1]$, let $u\zr_\y$ be the solution of \eqref{rscNLS} with the initial condition 
\EQ{
 u\zr_\y(T)=\sM_\om(\te\zr,b_-\zr,\y\z\zr),}
then $u\zr_1=u\zr$ and $u\zr_\y(T)\in\M^\om_0$. 
Moreover, $\bA^\om(u\zr_\y)$ is decreasing as $\y<1$ decreases until $\|\y \z\zr\|_\om\lec |b_-\zr|^2$, because 
\EQ{ \label{ene deg mfd}
 \y\frac{d}{d\y}\bA^\om(u\zr_\y(T))
 \pt= -2b_-\zr b_+'+\LR{\cL^\om\y\z\zr|\y\z\zr}-\LR{N^\om(v_\y)|b_+'g^\om_++\y\z\zr}
 \pr=\LR{\cL^\om\y\z\zr|\y\z\zr} + O((|b_-\zr|+\|\y\z\zr\|_\om)^3),}
where $b_+':=\p_\z G_\om(b_-\zr,\y\z\zr)\y\z\zr$ and $v_\y:=G_\om(b_-\zr,\y\z\zr)g^\om_++b_-\zr g^\om_-+\z\zr$. 
Hence there exists $\y_0\in(0,1)$ such that $\bA^\om(u\zr_\y)$ is increasing for $\y\in[\y_0,1]$ and $\|\y_0\z\zr\|_\om\lec|b_-\zr|^2\lec\de_M^2$. 
Since the energy constraint is preserved, those solutions $u\zr_\y$ are also on $\M^\om_2$ for $\y\in[\y_0,1]$. 
In the same way, we obtain a continuous family of solutions $u\on_\y$ in $\M^\om_2$ for $\y\in[\y_1,1]$ with some $\y_1\in(0,1)$ such that 
\EQ{
 u\on_\y(T)=\sM_\om(\te\on,b_-\on,\y\z\on),
 \pq \|\y_1\z\on\|_\om\lec\de_M^2.}
Let $\fy_s$ be the linear interpolation on $\M^\om_0$ between $u\zr_{\y_0}(T)$ and $u\on_{\y_1}(T)$, namely 
\EQ{
 \fy_s:=\sM_\om((1-s)\te\zr+s\te\on,(1-s)b_-\zr+sb_-\on,(1-s)\y_0\z\zr+s\y_1\z\on)}
for $s\in[0,1]$. Then the same estimate as in \eqref{ene deg mfd} yields
\EQ{
 \bA^\om(\fy_s)-\bA^\om(Q_\om) \lec \de_M^3 \ll c_X\de_M^2}
and so $\fy_s\in\M^\om_2$. Thus we have obtained a path connecting $u\zr(0)$ and $u\on(0)$ in $\M^\om_2$: 
\EQ{
 \{u\zr(t) \mid t:0\nearrow T\} \pt\cup \{u\zr_\y(T)\mid \y:1\searrow \y_0\}
 \cup\{\fy_s \mid s:0\nearrow 1\} 
 \pr\cup \{u\on_\y(T) \mid \y:\y_1\nearrow 1\} \cup\{u\on(t) \mid t:T\searrow 0\}.}

The trapping characterization \eqref{trap for mfd} of $\M^\om_2$, together with the distance gap \eqref{improved dist bd} from the ejected solutions, implies that for any solution $u_\om$ on $\M^\om_2$, the rescaled solution $u_\be(t):=\rS_{\be/\om} u_\om(\om t/\be)$ is also on $\M^\be_2$ if $\be/\om$ is close enough to $1$. 
Hence rescaling and unifying over $\om$ yields a $C^1$ manifold of codimension $1$:
\EQ{
 \M_3:=\Cu_{\om>\om_*}\rS_\om^{-1}\M^\om_2,} 
around $\Soli_1|_{\bM<\mu_*}$ in $H^1_r(\R^3)$. 
Since $\M^\om_2$ is invariant by the rescaled NLS, the above manifold $\M_3$ is invariant by \eqref{NLSP} in the original scaling. 
$\M_3$ is also connected\footnote{Let $X,Y$ be topological spaces and $M:X\to\cP(Y)$. Suppose that $X$ is connected, and that for every $x\in X$, $M(x)$ is connected and $M(x)\cap M(z)\not=\emptyset$ for all $z$ in a neighborhood of $x$. Then $\Cu_{x\in X}M(x)$ is also connected.} and unbounded, and it is the union of all orbits of forward global solutions $u$ of \eqref{NLSP} such that $u_\om:=\rS_\om u(t/\om)$ satisfies \eqref{trap for mfd} for some $\om>\om_*$. 

Restricting $\om\ge\om_\star$ and $\de_M\le(c_*/c_X)^{1/2}\de_*$, 
the scattering/blow-up after departure is applicable to the solutions off the manifold. Hence putting 
\EQ{
 \pt \cH_\star:=\{\fy\in H^1_r(\R^3)\mid \exists\om>\om_\star,\ \bA^\om(\rS_\om\fy)<\bA^\om(Q_\om)+c_X\de_M^2\},
 \pr \M_\star:=\Cu_{\om>\om_\star}\rS_\om^{-1}\M^\om_2}
we have 
\EQ{
  \M_\star = \cH_\star\cap \T_{C\de_M}, \pq \cH_\star\setminus\M_\star = \cH_\star\cap(\cS\cup\B).} 
Moreover, around each point $\fy\in\M_\star$, we can find a small open ball $B(\fy)\subset\cH_\star$ which is separated by $\M_\star$ into $\cS$ and $\B$. 
More precisely, $B(\fy)\setminus\M_\star$ is open with two connected components $B^\pm(\fy)$ such that $B^+(\fy)\subset\cS$ and $B^-(\fy)\subset\B$. 
Then $B_\star:=\Cu_{\fy\in\M_\star}B(\fy)$ is an open neighborhood of $\M_\star$ in $\cH_\star$, separated by $\M_\star$ into two disjoint open sets: $B^\pm_\star:=\Cu_{\fy\in\M_\star}B^\pm(\fy)$.  $B_\star$ and $B^\pm_\star$ are also connected sets, because $\M_\star$ is\addtocounter{footnote}{-1}\footnotemark.

The time inversion $\ba{\M_\star}$ is a $C^1$ manifold with codimension $1$, consisting of solutions in $\cH_\star$ trapped by $\Psi$ as $t\to-\I$. 
Hence $\M_\star \cap \ba{\M_\star}$ consists of solutions trapped by $\Psi$ as $t\to+\I$ and as $t\to-\I$. 
The one-pass lemma \ref{lem:one pass} implies that such a solution under the constraint stays within $O(\de_M)$ distance in $H^1_\om$ around $\Psi[\om]$ of some $\om>\om_\star$ for all $t\in\R$. 
Hence, taking $\de_M\ll\de_-$, we have  
\EQ{
 \M_\star\cap\ba{\M_\star} = \Cu_{\om>\om_\star}\{\rS_\om^{-1}\fy \mid \fy\in\M^\om_0\cap\ba{\M^\om_0},\ \bA^\om(\fy)<\bA^\om(Q_\om)+c_X\de_M^2\}.}
Actually, for any $\fy\in\M_\star\cap\ba{\M_\star}$ and any $\om>\om_\star$ such that $\bA^\om(\rS_\om\fy)<\bA^\om(Q_\om)+c_X\de_M^2$, we have $\rS_\om\fy\in\M^\om_0\cap\ba{\M^\om_0}$. 
Hence $\M_\star \cap \ba{\M_\star}$ is a $C^1$ invariant manifold with codimension $2$. The connectedness of $\M_\star\cap\ba{\M_\star}$ as well as $\M_2^\om\cap\ba{\M_2^\om}$ is proved in the same way as $\M_2^\om$, namely by reducing the dispersive component $\z$ in the local coordinate on $\M_0^\om\cap\ba{\M_0^\om}$. 
The same is for the connectedness of $\cS\cap\ba{\M_\star}$ and that of $\B\cap\ba{\M_\star}$, after rescaling and applying the backward flow in order to send them into the domain of the local coordinate around $Q_\om$. 

\appendix
\section{Table of Notation}
{\small
\begin{longtable}{l|l|l}
 \hline 
 symbols & description & defined in \\
 \hline
 $H=-\De+V$ & Schr\"odinger operator with the potential & \eqref{NLSP}, Section \ref{ss:asm V} \\ 
 $e_0$, $\phi_0$ & its eigenvalue and ground state &  \eqref{def e0phi0} \\  
 $Q$ & the ground state for NLS & \eqref{eq Q} \\
 $\ml{\cdot}$, $\bG$, $\bE,\bM,\bK_2$ & some functionals & \eqref{def ME}, \eqref{def funct}, \eqref{def K2}, \eqref{def K2'}\\
 $\bH^0$, $\bE^0$, $\bA$ & functionals without the potential & \eqref{def funct}, \eqref{def A} \\
 $\bA_\om$, $\bK_{0,\om}$ & functionals with frequency $\om$ & \eqref{def funcom} \\ 
 $\bE^\om$, $\bA^\om$, $\bK^\om_2$, $\bJ^\om$ & rescaled functionals & \eqref{def funcom} \\ 
 $\Soli$, $\Soli_j$, $\sE_j$ & all solitons, $j$-th solitons and energy & \eqref{def Soli}, \eqref{def Solij}, \eqref{def E0}, \eqref{def Ej} \\
 $(\Phi,\Om)$, $\Psi$ & the ground and first excited states & \eqref{def PhiPsi} \\
 $\mu_*$, $z_*,Z_*$, $\om_*$ & size of the above coordinates & Lemma \ref{lem:sum}\\
 $H^1_\om$, $\|\cdot\|_\om$ & rescaled energy norms & \eqref{def H1om}, \eqref{def normom} \\ 
 $\cN_\de(\Psi)$, $\cN_\om$ & neighborhoods of $\Psi$ (unscaled/rescaled) & \eqref{def Nde}, \eqref{def sNom} \\ 
 $\cS$, $\B$, $\T_\de$ & classification of initial data & \eqref{def SBT} \\ 
 $L^p$, $H^s_p$, $H^s$, $\dot H^s$, $B^s_{p,q}$ & Lebesgue, Sobolev and Besov spaces & Section \ref{ss:nota} \\
 $H^s_r$, $X_r$, $L^p_tX(I)$ & radial subspaces and  $X$-valued $L^p$ in $t$ & Section \ref{ss:nota} \\
 $(\cdot|\cdot)$, $\LR{\cdot|\cdot}$ & inner products on $L^2(\R^3)$ & Section \ref{ss:nota} \\
 $\Stz^s$, $\ST$ & Strichartz norms & Section \ref{ss:nota}\\
 $\cS^t_p$, $\cS'_p$, $\rS_\om$, $V^\om$ & scaling operators and scaled potential & \eqref{def Spt}, \eqref{def rSom} \\ 
 $\fy^\perp$, $P_\fy^\perp$, $P_c$ & orthogonal subspace and projection & \eqref{def perp}, \eqref{def Pc} \\ 
 $\diff l(a\pa):= l(a\on) - l(a\zr)$ & difference & Section \ref{ss:nota}\\
 $\CD$, $\CM,\CS$, $\CK$, $\om_\star$ & large constants & Lemmas \ref{lem:dich}, \ref{lem:minid}, \ref{lem:eject}, \ref{lem:smallM}\\ 
 $C_0$, $C_1$ & large constants & \eqref{vpb locbd}, Lemma \ref{lem:Lip} \\ 
 $\de_C,\de_D$, $\de_E$, $c_X,\de_I$, $\de_X$ & small constants & Lemmas \ref{lem:sum}, \ref{lem:LWP}, \ref{lem:inst}, \ref{lem:eject},\\ 
 $\de_U$, $\de_V,\e_S$, $\de_*$, $c_*$ & small constants & Lemmas  \ref{lem:var}, \ref{lem:sign}, \ref{lem:one pass}, \ref{lem:scatafter} \\
 $\e_V,\ka_V$ & small numbers in variational estimates & Lemma \ref{lem:var} \\
 $\cL,\cL^\om$, $L_\pm,L^\om_\pm$ & linearized operators & \eqref{def L}, \eqref{def Lom}\\
 $Q_\om$, $\cQ_\om$ & rescaled first excited state and its orbit & \eqref{def Psi Qom}, \eqref{def Qom orbit} \\ 
 $Q_\om'$, $Q'$ & frequency derivatives of solitons & \eqref{def Qom'}, \eqref{def Q'} \\ 
 $\al,\al_\om$, $g_\pm,g^\om_\pm$ & (un)stable eigenvalues/eigenfunctions & \eqref{eq g}, \eqref{def gom} \\
 $C^\om(v)$, $N^\om(v)$ & nonlinear part of energy and derivative & \eqref{def Com}, \eqref{def Nom} \\
 $P^\om_*$, $\sB^\om_*$, $\sZ^\om$ & symplectic projections around $Q^\om$ & \eqref{def Pom}, \eqref{def Pomc}, \eqref{def sBsZ} \\ 
 $d_{0,\om}$, $d_{1,\om}$, $d_\om$ & energy-distances to $\cQ_\om$ & \eqref{def d0om}, \eqref{def d1om}, \eqref{def dom} \\ 
 $\V_\om$, $\cZ_\om$ & $H^1$ subspaces with orthogonality & \eqref{def Vom}, \eqref{def Zom} \\
 $\sC_\om$, $\sU_\om$ & local chart around $\cQ_\om$ & \eqref{def sCom}, \eqref{def sUom} \\ 
 $m^\om(v)$ & modulation of phase & \eqref{eq m} \\ 
 $\cN^\om(v)$, $\cN^\om_*(v)$ & nonlinearity and its spectral projection & \eqref{eq v}, \eqref{def cNom*} \\
 $\chi$ & smooth cut-off function & \eqref{def chi} \\ 
 $\Sg_\om$, $\Sg$ & sign functionals & Lemma \ref{lem:sign} \\ 
 $I_H,I_V$ & sets of hyperbolic and variational times & \eqref{def IHV} \\ 
 $\sV_m$ & localized virial (depending on $\Sg_\om$) & \eqref{def virb}, \eqref{vir-s} \\ 
 $\cH_c[z]$, $R[z]$ & subspace and projection around $\Soli_0$ & \eqref{def HcRz} \\
 $B[z]$ & linear interaction with $\Soli_0$ & \eqref{def Bz} \\
 $\FS_\om(A)$ & set of global solutions away from $\cQ_\om$ & \eqref{scat after region} \\
 $\STN^\om$, $A^*_\om$, $A^*$ & Strichartz/energy bounds for scattering & \eqref{def A*} \\ 
 $\la^j_n$, $s^j_n$, $\ga^J_n$ & linear profile, its center and remainder & \eqref{linear prof decop}, \eqref{def laj}\\
 $\x^j_\I,\La^j_n$, $\Ga^J_n$ & nonlinear profiles and remainder & \eqref{def La}, \eqref{def LaGa}\\
 $G_\om$, $U_\om,\cU_\om$ & the graph of manifold and its domains & Theorem \ref{thm:mfd} \\ 
 $\deri{X}(\vec v)$ & operator for the difference & \eqref{def deri} \\
 $\M^\om_0$, $\sM_\om$ & local manifold and its coordinate & \eqref{def Mom0} \\ 
 \hline
\end{longtable}
}


\begin{thebibliography}{10}
\bibitem{DHR} T.~Duyckaerts, J.~Holmer and S.~Roudenko, {\em Scattering for the non-radial 3D cubic nonlinear Schr\"odinger equation}, Math. Res. Lett. {\bf 15} (2008), no.~6, 1233--1250. 
\bibitem{gnt} S.~Gustafson, K.~Nakanishi and T.~Tsai, {\it Asymptotic stability and completeness in the energy space for nonlinear Schr\"odinger equations with small solitary waves}, Int. Math. Res. Not. (2004) no. 66, 3559--3584.
\bibitem{HR} J.~Holmer and S.~Roudenko, {\em A sharp condition for scattering of the radial 3D cubic nonlinear Schr\"odinger equation},  Comm. Math. Phys.  {\bf 282}  (2008),  no.~2, 435--467. 
\bibitem{KM} C.~Kenig and F.~Merle, {\it Global well-posedness, scattering, and blow-up for the energy-critical focusing nonlinear Schr\"odinger equation in the radial case}, Invent. Math. {\bf 166} (2006), no.~3, pp.~645--675.
\bibitem{NLSP1} K.~Nakansihi, {\it Global dynamics below excited solitons for the nonlinear Schr\"odinger equation with a potential}, preprint,  arXiv:1504.06532. 
\bibitem{book}  K.~Nakanishi and W.~Schlag, 
``Invariant Manifolds and Dispersive Hamiltonian Evolution Equations" 
Z\"urich lectures in advanced mathematics, European Mathematical Society, 2011. 
\bibitem{NLS} K.~Nakanishi and W.~Schlag, {\it  Global dynamics above the ground state energy  for the cubic NLS equation in 3D}, Calc. Var. and PDE, {\bf 44} (2012), no.~1-2, 1--45. 
\bibitem{OT} T.~Ogawa, and Y.~Tsutsumi, 
{\it Blow-up of $H^1$ solution for the nonlinear Schr\"odinger equation}, J. Diff. Eq. {\bf 92} (1991), 317--330.
\end{thebibliography}
\end{document}